\documentclass[9pt,a4paper]{article}
\usepackage[T1]{fontenc}
\usepackage[francais,english]{babel}
\usepackage[applemac]{inputenc}
\usepackage{textcomp}
\usepackage{pb-diagram}
\usepackage{enumerate}
\usepackage[vmargin=3.5cm]{geometry}%
\usepackage{color}%
\usepackage{mdwlist}%
\usepackage{graphicx}%
\usepackage{float}%
\usepackage{makeidx}%
\usepackage{amsmath}%
\usepackage{amssymb}%
\usepackage{amsthm}%
\usepackage{amsfonts}%
\usepackage{mathrsfs}%
\usepackage{bbm}%
\usepackage[all,cmtip]{xy}%
\usepackage{rotate}%
\usepackage{yfonts}%
\usepackage{array}%
\newtheorem{Prop}{Proposition}[subsection]%
\newtheorem{Conj}[Prop]{Conjecture}%
\newtheorem{TheoEnglish}[Prop]{Theorem}%
\newtheorem{DefEnglish}[Prop]{Definition}
\newtheorem{DefProp}[Prop]{Proposition-Definition}
\newtheorem{CorEnglish}[Prop]{Corollary}

\newtheorem{LemEnglish}[Prop]{Lemma}
\newtheorem{HypEnglish}[Prop]{Assumption}%
\newcommand{\A}{\mathbb A}%

\newcommand{\Afiniq}{\mathbb A^{(\infty)}_{\Q}}

\newcommand{\C}{\mathbb C}%
\newcommand{\Fp}{\mathbb F}%
\newcommand{\G}{\mathbf G}%

\newcommand{\N}{\mathbb N}%
\newcommand{\Q}{\mathbb Q}%
\newcommand{\qp}{\mathbb Q_{p}}%
\newcommand{\R}{\mathbb R}%
\newcommand{\x}{\mathbf x}%
\newcommand{\y}{\mathbf y}%
\newcommand{\Z}{\mathbb Z}%
\newcommand{\zp}{\mathbb Z_{p}}%
\newcommand{\Fcal}{\mathcal F}%
\newcommand{\Hcal}{\mathcal H}%
\newcommand{\Mcal}{\mathcal{M}}%
\newcommand{\Ncal}{\mathcal N}%
\newcommand{\Ocal}{\mathcal O}%
\newcommand{\Kcal}{\mathcal K}%
\newcommand{\Wcal}{\mathcal W}%
\newcommand{\Fcali}{\mathscr F}%
\newcommand{\Gcali}{\mathscr G}%
\newcommand{\Xcali}{\mathscr X}%
\newcommand{\aid}{\mathfrak a}%
\newcommand{\cid}{\mathfrak c}%
\newcommand{\pid}{\mathfrak p}%
\newcommand{\mgot}{\mathfrak m}%
\newcommand{\gldeux}{\operatorname{GL}_{2}}%
\newcommand{\GL}{\operatorname{GL}}%
\newcommand{\SL}{\operatorname{SL}}
\newcommand{\Sh}{\operatorname{Sh}}
\newcommand{\et}{\operatorname{et}}

\newcommand{\diamant}[1]{\langle#1\rangle}%
\newcommand{\somme}[2]{\underset{#1}{\overset{#2}\sum}}%
\newcommand{\produit}[2]{\underset{#1}{\overset{#2}\prod}}%
\newcommand{\produittenseur}[2]{\underset{#1}{\overset{#2}\bigotimes}}%
\newcommand{\sommedirecte}[2]{\underset{#1}{\overset{#2}\bigoplus}}%
\newcommand{\applicationsimple}[3]{\begin{equation}%
\nonumber%
#1 :#2\longrightarrow #3%
\end{equation}}%
\newcommand{\application}[5]{\begin{eqnarray}%
\nonumber%
#1 :&#2&\longrightarrow #3\\
\nonumber%
&#4&\longmapsto #5
\end{eqnarray}}%
\newcommand{\suiteexacte}[5]{0\fleche#3\overset{#1}{\fleche}#4\overset{#2}{\fleche}#5\fleche0}

\newcommand{\limproj}[1]{\underset{\underset{#1}\longleftarrow}\lim}
\newcommand{\liminj}[1]{\underset{\underset{#1}\longrightarrow}\lim}

\newcommand{\Hom}{\operatorname{Hom}}

\newcommand{\isom}{\overset{\sim}{\longrightarrow}}

\newcommand{\plonge}{\hookrightarrow}
\newcommand{\matrice}[4]{\begin{pmatrix}#1&#2\\ #3&#4\end{pmatrix}}

\newcommand{\rank}{\operatorname{rank}}%
\newcommand{\ord}{\operatorname{ord}}

\newcommand{\loc}{\operatorname{loc}}

\newcommand{\tenseur}{\otimes}
\newcommand{\Ltenseur}{\overset{\operatorname{L}}{\tenseur}}
\newcommand{\modulo}{\operatorname{ mod }}
\newcommand{\Spec}{\operatorname{Spec}}
\newcommand{\Id}{\operatorname{Id}}%
\newcommand{\Aut}{\operatorname{Aut}}%
\newcommand{\Frac}{\operatorname{Frac}}%
\newcommand{\Tate}{\operatorname{Ta}}%
\newcommand{\Cone}{\operatorname{Cone}}%
\newcommand{\fleche}{\longrightarrow}%
\newcommand{\croix}{^{\times}}%
\newcommand{\idele}[1]{\widehat{#1}^{\times}}%
\newcommand{\surjection}{\twoheadrightarrow}%
\newcommand{\rhobar}{\bar{\rho}}%
\newcommand{\vide}{\varnothing}%
\newcommand{\red}{\operatorname{red}}
\newcommand{\new}{\operatorname{new}}
\newcommand{\Sym}{\operatorname{Sym}}
\newcommand{\BdR}{B_{\operatorname{dR}}}%
\newcommand{\dR}{\operatorname{dR}}%
\newcommand{\Fil}{\operatorname{Fil}}

\newcommand{\Hun}{H^{1}}

\newcommand{\Htilde}{\tilde{H}}
\newcommand{\deltatilde}{\tilde{\delta}}

\newcommand{\Htildeun}{\tilde{H}^{1}}

\newcommand{\RGamma}{\operatorname{R}\Gamma}%
\newcommand{\Det}{\operatorname{{D}et}}%
\newcommand{\Ccont}{C^{\bullet}_{\textrm{cont}}}%
\newcommand{\Sel}{\operatorname{Sel}}%

\newcommand{\tr}{\operatorname{tr}}
\newcommand{\Fr}{\operatorname{Fr}}%
\newcommand{\Gal}{\operatorname{Gal}}

\newcommand{\Qbar}{\bar{\Q}}%
\newcommand{\Fpbar}{\bar{\mathbb F}}%
\newcommand{\Dbar}{\bar{D}}
\newcommand{\kg}{\kappa}%
\newcommand{\kbold}{\mathbf{k}}%
\newcommand{\s}{\sigma}%
\newcommand{\hgot}{\mathfrak h}%
\newcommand{\Hecke}{\mathbf{T}}%
\newcommand{\Eul}{\operatorname{Eul}}
\DeclareFontEncoding{OT2}{}{} 
\newcommand{\Nekovar}{Nekov\'a\v{r}}%
\newcommand{\cl}{\operatorname{cl}}
\numberwithin{equation}{subsubsection}%
\date{}
\begin{document}%
\title{The Equivariant Tamagawa Number Conjecture for modular motives with coefficients in Hecke algebras}
\author{Olivier Fouquet}%
\maketitle
\selectlanguage{francais}
\begin{abstract}
Sous des hypothèses faibles sur la représentation résiduelle, nous prouvons la Conjecture Équivariante sur les Nombres de Tamagawa pour les motifs modulaires à coefficients dans les anneaux de déformations universelles et les algèbres de Hecke en utilisant une combinaison nouvelle de la méthode des systèmes d'Euler et de celle des systèmes de Taylor-Wiles. Nous prouvons aussi la compatibilité de cette conjecture par spécialisation.
\end{abstract}
\selectlanguage{english}
\begin{abstract}
Under mild hypotheses on the residual representation, we prove the Equivariant Tamagawa Number Conjecture for modular motives with coefficients in universal deformation rings and Hecke algebras using a novel combination of the methods of Euler systems and Taylor-Wiles systems. We also prove the compatibility of this conjecture with specialization.
\end{abstract}
%

%
\newcommand{\hord}{\mathfrak h^{\ord}}%
\newcommand{\hdual}{\mathfrak h^{dual}}%
\newcommand{\matricetype}{\begin{pmatrix}\ a&b\\ c&d\end{pmatrix}}%
\newcommand{\Iw}{\operatorname{Iw}}%
\newcommand{\Hi}{\operatorname{Hi}}
\newcommand{\cyc}{\operatorname{cyc}}
\newcommand{\ab}{\operatorname{ab}}
\newcommand{\can}{\operatorname{can}}%
\newcommand{\Fitt}{\operatorname{Fitt}}%
\newcommand{\Tiwa}{\mathcal T_{\operatorname{Iw}}}%
\newcommand{\Af}{\operatorname{A}}%
\newcommand{\Dunzero}{D_{1,0}}%
\newcommand{\Uun}{U_{1}}%
\newcommand{\Uzero}{U_{0}}%
\newcommand{\Uundual}{U^{1}}%
\newcommand{\Uunun}{U^{1}_{1}}%
\newcommand{\Wdual}{\Wcal^{dual}}
\newcommand{\JunNps}{J_{1,0}(\Ncal, P^{s})}%
\newcommand{\Tatepord}{\Tate_{\pid}^{ord}}%
\newcommand{\Kum}{\operatorname{Kum}}%
\newcommand{\zcid}{z(\cid)}%
\newcommand{\kgtilde}{\tilde{\kappa}}%
\newcommand{\kiwa}{\varkappa}%
\newcommand{\kiwatilde}{\tilde{\varkappa}}%
\newcommand{\Hbar}{\bar{H}}%
\newcommand{\Tred}{T/\mgot T}%
\newcommand{\Riwa}{R_{\operatorname{Iw}}}%
\newcommand{\Kiwa}{\Kcal_{\operatorname{Iw}}}%
\newcommand{\Sp}{\mathbf{Sp}}%
\newcommand{\Aiwa}{\mathcal A_{\operatorname{Iw}}}%
\newcommand{\Viwa}{\mathcal V_{\operatorname{Iw}}}%
\newcommand{\pseudiso}{\overset{\centerdot}{\isom}}%
\newcommand{\pseudisom}{\overset{\approx}{\fleche}}%
\newcommand{\carac}{\operatorname{char}}%
\newcommand{\length}{\operatorname{length}}
\newcommand{\eord}{e^{\ord}}%
\newcommand{\eordm}{e^{\ord}_{\mgot}}%
\newcommand{\hordinfini}{\hord_{\infty}}%
\newcommand{\Mordinfini}{M^{\ord}_{\infty}}%
\newcommand{\hordm}{\hord_{\mgot}}%
\newcommand{\hminm}{\hgot^{min}_{\mgot}}%
\newcommand{\Mordm}{M^{\ord}_{\mgot}}%
\newcommand{\Mtwist}{M^{tw}_{\mgot}}
\newcommand{\Xun}{X_{1}}%
\newcommand{\Xundual}{X^{1}}%
\newcommand{\Xunun}{X^{1}_{1}}%
\newcommand{\Xtw}{X^{tw}}
\newcommand{\Inert}{\mathfrak{In}}%
\newcommand{\Tsp}{T_{\Sp}}%
\newcommand{\Asp}{A_{\Sp}}%
\newcommand{\Vsp}{V_{\Sp}}%
\newcommand{\SK}{\mathscr{S}}%
\newcommand{\Rord}{R^{\ord}}%
\newcommand{\per}{\operatorname{per}}
\newcommand{\z}{\mathbf{z}}
\newcommand{\zs}{\tilde{\mathbf{z}}}
\newcommand{\Ebarbar}{\bar{\bar{E}}}
\newcommand{\Grsym}{\mathfrak S}
\newcommand{\epsi}{\varepsilon}
\newcommand{\Fun}[2]{F^{{\mathbf{#1}}}_{#2}}
\newcommand{\triv}{\operatorname{triv}}
\newcommand{\aidIwx}{(\aid_{x})_{\Iw}}
\newcommand{\Heckes}{\Hecke_{\Sigma,\Iw}}
\newcommand{\Hs}{\Hecke_{\Sigma,\Iw}}
\newcommand{\Ts}{T_{\Sigma,\Iw}}
\newcommand{\Raid}{R(\aid)_{\Iw}}
\newcommand{\Taid}{T(\aid)_{\Iw}}

\newcommand{\isocan}{\overset{\can}{\simeq}}
\newcommand{\cusps}{\operatorname{cusps}}
\newcommand{\ad}{\operatorname{ad}}
\newcommand{\Lie}{\operatorname{Lie}}
\newcommand{\eqdef}{\overset{\operatorname{def}}{=}}
\makeatletter
\newcommand\@biprod[1]{%
  \vcenter{\hbox{\ooalign{$#1\prod$\cr$#1\coprod$\cr}}}}
\newcommand\biprod{\mathop{\mathpalette\@biprod\relax}\displaylimits}
\makeatother
{\footnotesize \tableofcontents}

\selectlanguage{english}%

\section{Introduction}
\subsection{Equivariant conjectures and Iwasawa theory of modular forms}\label{SubIwa}
The aim of this manuscript is to describe in terms of cohomological data the special values of $L$-functions of eigencuspforms in $p$-adic  families parametrized by Hecke algebras.
 
Fix once and for all a prime $p\geq3$. For $\Lambda$ a $p$-adic ring (typically a compact $p$-ring), K.Kato stated in \cite{KatoViaBdR} an influential conjecture about smooth étale sheaves of $\Lambda$-modules on $\Spec\Ocal_{F}[1/p]$ (here $\Ocal_{F}$ is the ring of integers of a number field $F$) which he named there the generalized Iwasawa main conjecture for motives with coefficients in $\Lambda$ although it is nowadays more commonly called the Equivariant Tamagawa Number Conjecture (or ETNC, for short) with coefficients in $\Lambda$. When applied to the étale sheaf of $\zp$-modules attached to the $p$-adic étale realization of a motive $M$ over a number field $F$, the ETNC recovers the original Tamagawa Number Conjectures of \cite{BlochKato} (see also \cite{FontainePerrinRiou}) on the $p$-adic valuation of the special values of the $L$-function of $M$. Compared to the original conjectures of Bloch and Kato, the novelty of the ETNC is that it is inherently a variational statement: it predicts not only the exact valuation of the algebraic part of the values at integers of the $L$-function of a motive $M$ but also the $p$-adic variation of these special values as the étale realization of $M$ ranges over the geometric points of a $p$-adic analytic family of $G_{F}$-representations parametrized by $\Spec\Lambda$. 

The historically earliest non-trivial examples of such $\Lambda$-adic families were given by classical Iwasawa theory for Galois representations and motives, that is to say by the families which arise, in the language of deformation theory of Galois representations of \cite{MazurDeformation}, by deforming the determinant of a single motivic Galois representation. In that case, $\Lambda$ is equal to the classical Iwasawa algebra $\Lambda_{\Iw}$ of \cite{SerreIwasawa}. The ETNC for the family of cyclotomic twists of the Galois representation attached to an eigencuspform $f\in S_{k}(\Gamma_{1}(N))$, for instance, recovers the Iwasawa Main Conjecture for modular forms  \cite[Conjecture 12.10]{KatoEuler} and the Iwasawa Main Conjectures of \cite{GreenbergIwasawaRepresentation,GreenbergIwasawaMotives,PerrinRiouLpadique,PollackSupersingular,KobayashiIMC} relating the analytic $p$-adic $L$-function of $f$ to the direct limit on $n$ of the Selmer groups $\Sel_{\Q(\zeta_{p^{n}})}(\rho_{f})$. In this guise, it is known to hold in many cases by the combined results of \cite{KatoEuler,SkinnerUrban,KobayashiIMC,XinWanIMC} (see also \cite{EmertonPollackWeston,OchiaiMainConjecture} for important partial results). The ETNC (\cite[Conjecture 3.2]{KatoViaBdR}) is a precise expression of the hope that there should exist an Iwasawa theory with coefficients in $\Lambda$ describing the $p$-adic variation of special values in $\Lambda$-adic families just like classical Iwasawa theory describes the variation of special values in families of twists by characters.

In the last three decades, $p$-adic families of automorphic Galois representations parametrized by Hecke algebras have proven to be of considerable arithmetic interest, if only because they are believed to coincide with the $p$-adic families arising as universal deformations of residual automorphic Galois representations. This manuscript states and proves many cases of the ETNC with coefficients in Hecke algebras for the motives attached to rational eigencuspforms. 

More precisely, let $\rhobar:G_{\Q}\fleche\GL_{2}(\Fpbar_{p})$ be an irreducible modular (equivalently, odd) representation. Denote by $N(\rhobar)$ its tame Artin conductor and choose $\Sigma\supset\{\ell|N(\rhobar_{f})p\}$ a finite set of primes. To this data is attached a local, reduced, $p$-adic Hecke algebra $\Hs$ (the local factor corresponding to $\rhobar$ in the inverse limit on the $p$-part of the level of the Hecke algebras generated by operators outside $\Sigma$) and an étale sheaf $T_{\Sigma,\Iw}$ of $\Hs$-modules on $\Spec\Z[1/\Sigma]$ (the subscript $\Iw$ is here to remind the reader that $\Hs$ is an algebra over $\Lambda_{\Iw}$). A prime $x\in\Spec\Hs[1/p]$ is said to be classical if and only if the morphism $\Hs[1/p]\fleche\mathbf{k}(x)$ with values in the residue field at $x$ is the system of eigenvalues of a classical eigencuspform $f_{x}$ of weight $k_{x}\geq2$ twisted by a character $\chi_{x}$ of pro-$p$ order. If $x$ is classical and if $M(f_{x})_{\et,p}$ is the $p$-adic étale realization of the Grothendieck motive attached to $f_{x}$, then the fibre $M_{x}$ of $T_{\Sigma,\Iw}$ at $x$ is isomorphic as $\mathbf{k}(x)[G_{\Q}]$-module to $M(f_{x})_{\et,p}\tenseur\chi_{x}$.%

An outline of our main results is as follows (we refer the reader to section \ref{SubConj} and to theorem \ref{TheoCorps} in the body of the text for details).
\begin{TheoEnglish}\label{TheoIntro}
Assume that $\rhobar$ satisfies the following hypotheses.
\begin{enumerate}
\item Let $p^{*}$ be $(-1)^{(p-1)/2}p$. Then $\rhobar|_{G_{\Q(\sqrt{p^{*}})}}$ is irreducible.
\item The semisimplification of $\rhobar|_{G_{\qp}}$ is not scalar.
\suspend{enumerate}
Assume moreover that $\rhobar$ satisfies at least one of the following conditions.
\resume{enumerate}
\item There exists $\ell$ dividing exactly once $N(\rhobar)$, the representation $\rhobar|_{G_{\qp}}$ is reducible and $\det\rhobar$ is unramified outside $p$.
\item There exists $\ell$ dividing exactly once $N(\rhobar)$ and there exists an eigencuspform $f\in S_{2}(\Gamma_{0}(N))$ attached to a classical $x\in\Spec\Hs[1/p]$ with rational coefficients, zero eigenvalue at $p$ and such that $N$ is square-free.
\item The order of the image of $\rhobar$ is divisible by $p$ and there exists an eigencuspform $f$ attached to classical $x\in\Spec\Hs[1/p]$ for which the ETNC with coefficients in $\Lambda_{\Iw}$ holds. 
\end{enumerate}
Then the ETNC with coefficients in $\Hs$ is true for $\Ts$ and both the ETNC with coefficients in $\Hs$ and the ETNC with coefficients in $\Lambda_{\Iw}$ are true for $M(f)_{\et,p}$ for all $f$ attached to classical primes of $\Spec\Hs[1/p]$. In particular, there exist a free $\Hs$-module $\Delta_{\Sigma,\Iw}(\Ts)$ called the \emph{universal fundamental line} (see definition \ref{DefDeltaUniv}) and a basis $\z_{\Sigma,\Iw}$ of $\Delta_{\Sigma,\Iw}$ called the \emph{universal zeta element} satisfying the following properties. 

Let $f\in S_{k}(\Gamma_{1}(N))$ be a classical eigencuspform with eigenvalues in a number field $F$, let $\psi$  be a character of sufficiently large finite order of $\Gal(\Q(\zeta_{p^{\infty}})/\Q)$, let $1\leq r\leq k-1$ be an integer and let $x\in\Spec\Hs[1/p]$ be the classical prime attached to the pair $(f,\chi_{\cyc}^{r}\psi)$. Denote by
\begin{equation}\nonumber
\Delta_{\kbold(x)}(M_{x})\eqdef\Det^{-1}_{\mathbf{k}(x)}\RGamma_{\et}(\Z[1/p],M_{x})\tenseur_{}\Det^{-1}_{\mathbf{k}(x)}M_{x}(-1)^{+}
\end{equation}
the fundamental line of the fiber of $T_{\Sigma,\Iw}$ at $x$ and by $\z_{x}$ the basis of $\Delta_{\kbold(x)}(M_{x})$ given by the choice of any $G_{\Q}$-stable lattice inside $M_{x}$. Then the following holds.
\begin{enumerate}[(i)]
\item\label{EqSpecIntro} There is canonical isomorphism
\begin{equation}\nonumber
\Delta_{\Sigma,\Iw}(\Ts)\tenseur_{\Hs}\mathbf{k}(x)\isocan\Delta_{\kbold(x)}(M_{x})
\end{equation}
sending $\z_{\Sigma,\Iw}\tenseur1$ to $\z_{x}$. 
\item\label{EqPeriodIntro}There exist an $F$-vector space $\Delta_{F}(M_{x})$ and canonical $p$-adic and complex period maps independent of $r$ and $\psi$
\begin{equation}\nonumber
\per_{f,p}^{-1}:\Delta_{\kbold(x)}(M_{x})\isocan\Delta_{F}(M_{x})\tenseur_{F}\kbold(x),\ \per_{f,\infty}:\Delta_{F}(M_{x})\tenseur\C\isocan\C 
\end{equation}
such that  $\per_{f,\C}(\per_{f,p}^{-1}(\z_{x})\tenseur1)$ is well-defined and is equal to $L_{\{p\}}(M(f)^{*}(1),\psi,r)$.
\end{enumerate}
\end{TheoEnglish}
The knowledge of $\z_{\Sigma,\Iw}$, or equivalently of $\Delta_{\Sigma,\Iw}(\Ts)$, thus implies the knowledge of the special values of the $L$-function at integers in the critical strip of all eigencuspforms attached to classical primes of $\Hs$. Because $\Hs$ is the universal deformation ring of $\rhobar$ under the hypotheses of the theorem, these are the eigencuspforms with residual representation isomorphic to $\rhobar$ and theorem \ref{TheoIntro} in particular predicts congruences between special values of $L$-functions of congruent modular forms.\footnote{Congruent motives do not have congruent special values in general. See subsection \ref{SubExamples} for examples of the congruences that do arise in this way.} This settles in the affirmative the question asked at the end of the introduction of \cite{MazurValues}. Reversing the perspective as in \cite{KatoICM}, or indeed as in the original works of Gauss and Dirichlet on class groups, theorem \ref{TheoIntro} also implies that the structure as $\Hs$-module of the Galois cohomology of $M(f)_{\et,p}$ is encoded in the special values of the $L$-functions of forms congruent to $f$.

Let $f$ be an eigencuspform corresponding to a classical prime of $\Hs$. If in addition to the hypotheses of theorem \ref{TheoIntro}, one assumes that the semisimplification of $\rho_{f}|_{G_{\qp}}$ is reducible and that $\pi(f)_{p}$ is principal series, then the assertion of theorem \ref{TheoIntro} that the ETNC with coefficients in $\Lambda_{\Iw}$ holds for $M(f)_{\et,p}$ holds (equivalently that the main conjecture \cite[Conjecture 12.10]{KatoEuler} holds for $M(f)_{\et,p}$, equivalently that the Iwasawa Main Conjecture \cite[Conjecture 2.2]{GreenbergIwasawaMotives} holds for $M(f)_{\et,p}$) follows under a couple of technical hypotheses from a combination of \cite[Theorem 12.4]{KatoEuler} and of \cite[Theorem 3.29]{SkinnerUrban}. A well-known interpolation technique in Hida theory then proves the result under the same hypotheses except that $\pi(f)_{p}$ is allowed to be Steinberg. Likewise, if $M(f)_{\et,p}$ is the motive attached to a supersingular elliptic curve with vanishing $p$-eigenvalue, the main conjecture \cite[Conjecture 12.10]{KatoEuler} is often known to hold by \cite{KatoEuler,KobayashiIMC,XinWanIMC}. In contrast with these results, the hypotheses of theorem \ref{TheoIntro} involve only $\rhobar_{f}$. This allows to reduce the Iwasawa Main Conjecture to more favorable cases by congruences (see subsection \ref{SubExamples} for examples; the case of forms of weight $k>2$ with finite, non-zero slope at $p$ being probably the hardest to treat prior to this work). 

More importantly, it should be noted that even for a single ordinary eigencuspform, the ETNC with coefficients in $\Hs$ for $M(f)_{\et,p}$ is a strong refinement of the ETNC with coefficients in $\Lambda_{\Iw}$ for $M(f)_{\et,p}$ (or equivalently of Greenberg's Iwasawa main conjecture for $f$): the former proves an equality of special values up to a unit in the Hecke algebra whereas the latter predicts such an equality only up to a unit in the normalization of the Hecke algebra, a potentially much larger ring (to illustrate with a close analogy, it is immediate to prove that the first étale cohomology group of the modular curves is free over the normalization of the Hecke algebra but much harder, and in fact not always true, that it is free over the Hecke algebra itself). As already mentioned, the ETNC with coefficients in $\Hs$ is in turn a much stronger statement than the collection of the ETNC for all individual classical primes of $\Hs[1/p]$ as it provides a supplementary description of congruence between special values.  %

\subsection{Outline of the proof}\label{SubOutline}
\subsubsection{Weight-Monodromy, completed cohomology and fundamental lines}

In analogy with the case of separated schemes of finite type over a finite field and in agreement with \cite{KatoHodgeIwasawa,FontainePerrinRiou}, we understand the Tamagawa Number Conjecture to be a description of the values at integers of the $L$-function of a motive $M^{*}(1)$ in terms of a \textit{fundamental line}  and a \textit{zeta element}; respectively a graded invertible $\zp$-module $\Delta_{\zp}(M)$ manufactured from the determinant of the étale cohomology of $\Spec\Z[1/p]$ with coefficients in the $p$-adic étale realization of $M$ and a basis $\z(M)\in\Delta_{\zp}(M)$ whose image through $p$-adic and complex period maps computes $L(M^{*}(1),0)$.\footnote{When $\Spec\Z[1/p]$ is replaced by a finite separated scheme $X/\Fp_{q}$ and $M_{\et,p}$ by a smooth étale sheaf on $X$, this formulation amounts to the completion, mostly by Grothendieck, of Weil's program of studies of $L$-functions as achieved in \cite[Exposé III]{DixExposes} and \cite{SGA41/2}. In personal discussions with the author, K.Kato recalled that his motivation in restating the Tamagawa Number Conjectures of \cite{BlochKato} in this framework was to show that they could take the same external form as the Main Conjecture of Iwasawa theory while J-M.Fontaine credited P.Deligne for the insight that such a formulation could be desirable.} Crucially, the conjecture posits the existence of an absolute rational cohomology theory which admits realization functors to Betti, de Rham and étale cohomology (this is the space $\Delta_{F}(M_{x})$ of theorem \ref{TheoIntro}). The equivariant refinement of \cite{KatoViaBdR} (see also \cite{KatoHodgeIwasawa,BurnsFlachTate} for the case of group algebra) is a generalization to motives with coefficients in $\Lambda$ (equivalently, to $p$-adic families of motives parametrized by $\Spec\Lambda$) together with the compatibility of the conjecture with base change of ring of coefficients (see conjecture \ref{ConjETNC} below for a prototypical example in which $\Lambda=\Lambda_{\Iw}$)

Seen under this perspective, the first difficulty in the study of the ETNC for modular motives with coefficients in the Hecke algebra is in specifying what is the precise result that is to be proven. Indeed, the standard formulation of the ETNC, that of \cite[Conjecture 3.2.1]{KatoViaBdR}, assigns a central role to the determinant of the étale cohomology of $\Spec\Z[1/p]$ with coefficients in a bounded complex $\Fcali$ of smooth étale sheaves of $\Lambda$-modules on $\Spec\Z[1/p]$ and consequently requires crucially that $\RGamma_{\et}(\Spec\Z[1/p],\Fcali)$ be a perfect complex of $\Lambda$-modules. It is not known (and perhaps not expected) that the étale sheaf of Hecke-modules $\Fcali$ coming from the first étale cohomology group of a modular variety $X$, or more generally from the étale cohomology in middle degree of a Shimura variety, is such that $\RGamma_{\et}(\Spec\Z[1/p],\Fcali)$ is a perfect complex of Hecke-modules (the problem being that the inertia invariants submodule of $\Hun_{\et}(X\times_{\Q}\Qbar,\zp)$ at a prime of bad reduction is not known to be a Hecke-module of finite projective dimension). Moreover, even when the formulation of the ETNC of \cite{KatoViaBdR} is known to apply (for instance if the coefficient ring is assumed to be local regular so that all bounded complexes are perfect by the theorem of Auslander-Buchsbaum and Serre), it is typically not known that the ETNC as formulated in \cite{KatoViaBdR} is compatible with arbitrary characteristic zero specialization (and this is in fact very likely to be false in general). Finally, several different $p$-adic Hecke algebras may act on a given modular motive depending on whether Hecke operators at places of bad reduction and $p$ are considered and on whether they act on the full space of modular forms or just on newforms. These different choices lead to different conjectures which are \textit{a priori} neither equivalent nor in fact obviously compatible with each other. For these reasons, even a precise unconditional formulation of the ETNC with coefficients in Hecke algebras for modular motives seems to have been heretofore missing from the literature. We note that each of these problems is a manifestation of the fact that Tamagawa numbers (in the usual sense) are not well behaved in $p$-adic families.

The first main idea of this manuscript simultaneously solves these problems thanks to the following crucial observation: the severe constraints conjecturally put on the action of the inertia group on the $p$-adic étale realization of a motive by the Weight-Monodromy Conjecture allow to refine the definition of the local complexes involved in the statement of the ETNC. This process yields refined fundamental lines which are not in general determinants of perfect complexes but rather canonical trivializations of invertible graded modules which themselves are the determinants of the sought for perfect complexes when these are known to exist. When the motive is of automorphic origin, the Local Langlands Correspondence further constrains the inertia action and our constructions are in this way shown to be compatible with the action of the Hecke algebra. Indeed, the very definition of the refined fundamental line for an automorphic motive singles out a specific local factor of the Hecke algebra which coincides with the universal deformation ring subject to natural conditions.

A conceptually satisfying property of the refined fundamental lines is that they are almost by construction shown to be compatible with change of rings of coefficients at motivic points; a property which generalizes the control theorem of \cite{MazurRational} (and much subsequent works) in a probably optimal way. In order for them to be compatible with change of levels in the automorphic sense and with specializations at non-motivic points, it is necessary to alter further the definition suggested in \cite{KatoViaBdR}. The reason is as follows. The complex period map intervening in the Tamagawa Number Conjecture compares de Rham and Betti cohomology. For a single motive, this makes sense, but what is the Betti realization of a $\Lambda$-adic family of motives? A natural answer-the one implicit in \cite{KatoViaBdR} for instance-would be that it is its $\Lambda$-adic étale cohomology. It turns out however that in the Shimura variety context, it is the choice of the completed cohomology of \cite{EmertonInterpolationEigenvalues} which yields the correct equivariance properties (this is closely related to the fact that local-global compatibility in $p$-adic families involves the normalization of the Langlands correspondance of \cite{BreuilSchneider} and not the classical normalization). Consequently, the fundamental lines of this manuscript are refined fundamental lines tensored with the determinant of completed cohomology. Independently of the hypotheses of theorem \ref{TheoIntro}, we formulate precise conjectures relating this object to the variation of special values of $L$-functions and show that they are compatible with change of levels, passage to the new quotient and characteristic zero specializations. 

\subsubsection{Euler systems and Taylor-Wiles systems}
The proof of theorem \ref{TheoIntro} is then by a novel amplification of the method of Euler/Kolyvagin systems, that is to say the combination of V.Kolyvagin's observation in \cite{KolyvaginEuler} that Galois cohomology classes satisfying compatibility relations in towers of extensions reminiscent of the properties of partial Euler products yield systems of classes with coefficients in principal artinian rings whose local properties are sufficiently constrained to establish a crude bound on the order of some Galois cohomology groups or Selmer groups and the descent principle due to K.Rubin, which allows under suitable assumptions to translate a collection of crude bounds for many specializations with coefficients in artinian rings into a sharp bound in the limit, that is for objects with coefficients in Iwasawa algebras. When the ring of coefficients of the limit object is not known to be normal, as is the case with Hecke algebra, this descent principle meets quite formidable challenges, as it is of course entirely possible for an invertible module to be non-integral while all its specializations to discrete valuations rings are integral, in which case no contradiction can arise by naïve descent. For this reason, most accounts of the Euler/Kolyvagin systems method (\cite{PerrinRiouEuler,RubinEuler,KatoEuler,MazurRubin,HowardKolyvagin,HowardGLdeux,OchiaiEuler,FouquetRIMS} for instance) assume that the ring of coefficients is normal, or even regular, and those which do not (\cite{KatoEulerOriginal,FouquetDihedral} for instance) typically prove weaker statement at the locus of non-normality of the coefficient ring.

Our second main novel contribution allows us to bypass this difficulty by first resolving the singularities of the Hecke algebra using the method of Taylor-Wiles of \cite{WilesFermat,TaylorWiles} systems as axiomatized in \cite{DiamondHecke,FujiwaraDeformation} before applying the descent procedure. Under the two first hypotheses of theorem \ref{TheoIntro}, there exists a Taylor-Wiles system $\{\Delta_{Q}\}_{Q}$ (indexed by finite set of well-chosen primes) of refined fundamental lines. This system yields a limit object $\Delta_{\infty}$ over a regular local ring $R_{\infty}$. If the limit object $\Delta_{\infty}$ is not integral, then it has non-integral specializations to discrete valuation rings. Even though $\Delta_{\infty}$ itself has no Galois interpretation, its specializations do, so that this non-integrality contradicts Kolyvagin's bound (or more accurately the sharper results of \cite{KatoEulerOriginal}). Hence $\Delta_{\infty}$ is integral. Then so are the $\Delta_{Q}$ and in particular the fundamental line $\Delta_{\Sigma,\Iw}$ we started with. We note that this argument is by nature extremely sensitive to the existence of a potential error term at any step and thus relies critically on the exact control property of the refined fundamental lines.

We record the following observation, which lies at the conceptual core of this manuscript: just as the conjectured compatibility of the Tamagawa Number Conjecture with the $\Gal(\Q(\zeta_{Np^{s}})/\Q)$-action coming from the covering $\Spec\Z[\zeta_{Np^{s}},1/p]\fleche\Spec\Z[1/p]$ implies that the collection of motivic zeta elements should form an Euler system, the conjectured compatibility of the Tamagawa Number Conjecture with the action of the Hecke algebra coming from the covering $X_{U'}\fleche X_{U}$ of Shimura varieties with $U'\subset U$ implies that the collection of refined fundamental lines should form a Taylor-Wiles system. In both cases, the compatibilities we hope that conjectures on special values  of $L$-functions satisfy therefore suggest powerful tools to actually establish the conjectures.

\subsubsection{Organization of the manuscript}
The second section of the manuscript is devoted to a precise statement of the ETNC with coefficients in $\Lambda_{\Iw}$ for modular motives and of our choice of normalizations therein. We mostly follow the practices of \cite{KatoViaBdR,KatoEuler} and when these two differ (for instance in the normalization of the $p$-adic period map), it is usually the first that we follow. We also review the known results on the Iwasawa Main Conjecture for modular forms that we need in subsection \ref{SubBiblio}. The third section is devoted to a precise unconditional statement of the ETNC with coefficients in various Hecke algebras and to the proof that it is compatible with specializations. We treat small as well as large Hecke algebras and pay particular attention to the $p$-adic variation of Euler factors at places of bad reduction. The fourth and last section contains the strongest form of our main results, numerical examples of modular forms satisfying the Iwasawa Main Conjecture and of congruences between special values of $L$-functions that seem to be new and the proof of our main result (as well as concluding questions about the $p$-adic properties of $L$-function of modular motives). The manuscript ends with an appendix giving our conventions relative to Selmer complexes.  

\paragraph{Acknowledgements:}It is a pleasure to thank J-B.Bost, L.Clozel, T.Fukaya, B.Mazur, F.Jouve, K.Kato and J.Riou for helpful comments and advice and especially R.Pollack for many crucial insights on how to conduct the computations of subsection \ref{SubExamples}.

\subsubsection*{Notations}
All rings are assumed to be commutative (and unital). The total quotient ring of a reduced ring $R$ is denoted by $Q(R)$ and the fraction field of a domain $A$ is denoted by $\Frac(A)$. If $R$ is a ring, we denote by $\RGamma_{\et}(R,-)$ the étale cohomology complex $\RGamma_{\et}(\Spec R,-)$.

If $F$ is a field, we denote by $G_{F}$ the Galois group of a separable closure of $F$. If $F$ is a number field and $\Sigma$ is a finite set of places of $F$, we denote by $F_{\Sigma}$ the maximal Galois extension of $F$ unramified outside $\Sigma\cup\{v|\infty\}$ and by $G_{F,\Sigma}$ the Galois group $\Gal(F_{\Sigma}/F)$. The ring of integer of $F$ is written $\Ocal_{F}$. If $v$ is a finite place of $F$, then $\Ocal_{F,v}$ is the unit ball of $F_{v}$, $\varpi_{v}$ is a fixed choice of uniformizing parameter and $k_{v}$ is the residual field of $\Ocal_{F,v}$. The reciprocity law of local class field theory is normalized so that $\varpi_{v}$ is sent to (a choice of lift of) the geometric Frobenius morphism $\Fr(v)$. For all rational primes $\ell$, we fix an algebraic closure $\Qbar_{\ell}$ of $\Q_{\ell}$, an embedding of $\Qbar$ into $\Qbar_{\ell}$ and an identification $\iota_{\infty,\ell}:\C\simeq\Qbar_{\ell}$ extending $\Qbar\plonge\Qbar_{\ell}$.

The non-trivial element of $\Gal(\C/\R)$ is denoted by $\tau$. If $R$ is a ring in which $2$ is a unit and $M$ is an $R[\Gal(\C/\R)]$-module (resp. $m$ is an element of $M$), then $M^{\pm}$ (resp. $m^{\pm}$) denotes the eigenspace on which $\tau$ acts as $\pm 1$ (resp. the projection of $m$ to the $\pm$-eigenspace).

The field $\Q_{\infty}/\Q$ is the cyclotomic $\zp$-extension of $\Q$, that is to say the only Galois extension of $\Q$ with Galois group $\Gamma$ isomorphic to $\zp$. For $n\in\N$, the number field $\Q_{n}$ is the sub-extension of $\Q_{\infty}$ with Galois group $\Z/p^{n}\Z$. If $S$ is a reduced $\zp$-algebra, we write $S_{\Iw}$ for the completed group-algebra $S[[\Gamma]]$. To stick to usual notations, the 2-dimensional regular local ring $\Z_{p,\Iw}=\zp[[\Gamma]]\simeq\zp[[X]]$ is denoted by $\Lambda_{\Iw}$.

For $G$ a group, a $G$-representation $(T,\rho,R)$ is an $R$-module $T$ free of finite rank together with a morphism
\begin{equation}\nonumber
\rho:G\fleche\Aut_{R}(T)
\end{equation}
which is assumed to be continuous if $G$ is a topological group. If $F$ is a number field, $\Sigma$ is a finite set of finite primes of $\Ocal_{F}$ containing $\{v|p\}$, $R$ is complete local noetherian ring of residual characteristic $p$ which is a reduced $\zp$-algebra and $(T,\rho,R)$ is a $G_{F,\Sigma}$-representation, then the $G_{F,\Sigma}$-representation $(T_{\Iw},\rho_{\Iw},R_{\Iw})$ is the $R_{\Iw}$-module $T\tenseur_{R}R_{\Iw}$ with $G_{F,\Sigma}$-action on $T$ through $\rho$ and on $R_{\Iw}$ through the composition $\chi_{\Gamma}:G_{F,\Sigma}\surjection\Gal(F\Q_{\infty}/F)\plonge R_{\Iw}\croix$. 
We do not distinguish between the $G_{F,\Sigma}$-representation $(T,\rho,R)$ and the corresponding étale sheaf $T$ on $\Spec\Ocal_{F}[1/\Sigma]$.

We refer to appendices \ref{AppDeterminant} and \ref{AppSelmer} for notations and conventions regarding the determinant functor and complexes of cohomology with local conditions.
\section{The ETNC for modular motives with coefficients in $\Lambda_{\Iw}$}
\subsection{Modular curves, modular forms}
\subsubsection{Modular curves}
Let $\G$ be the reductive group $\GL_{2}$ over $\Q$ and let $\Sh(\G,\C-\R)$ be the tower of Shimura curves attached to the Shimura datum $(\G,\C-\R)$. For $U\subset\G(\A_{\Q}^{(\infty)})$ a compact open subgroup, the curve $Y(U)=\Sh_{U}(\G,\C-\R)$ and its compactification along cusps $j:Y(U)\plonge X(U)$ are regular schemes over $\Z$ which are smooth over $\Z_{\ell}$ if $U_{\ell}$ is maximal and $U$ is sufficiently small; \textit{e.g} $U=U(N)$ and $N\geq3$ (see \cite[p. 305]{KatzMazur}). The set of complex points of $Y(U)$ is given by the double quotient
\begin{equation}\nonumber
Y(U)(\C)\simeq\G(\Q)\backslash\left(\C-\R\times\G(\A_{\Q}^{(\infty)})/U\right)
\end{equation}
and is an algebraic variety if $U$ is sufficiently small. We consider the following compact open subgroups of $\G(\A_{\Q}^{(\infty)})$.
\begin{align}\nonumber
&U_{0}(N)=\produit{\ell}{}U_{0}(N)_{\ell}=\produit{\ell}{}\left\{g\in\gldeux(\Z_{\ell})|g\equiv\matrice{*}{*}{0}{*}\modulo \ell^{v_{\ell}(N)}\right\}\\\nonumber
&U_{1}(N)=\produit{\ell}{}U_{1}(N)_{\ell}=\produit{\ell}{}\left\{g\in\gldeux(\Z_{\ell})|g\equiv\matrice{*}{*}{0}{1}\modulo \ell^{v_{\ell}(N)}\right\}\\\nonumber
&U(M,N)=\produit{\ell}{}U(M,N)_{\ell}=\produit{\ell}{}\left\{g\in U_{1}(N)_{\ell}|g\equiv\matrice{1}{0}{*}{*}\modulo\ell^{v_{\ell}(M)}\right\}\\\nonumber
&U(N)=U(N,N).
\end{align}
For $U=U_{?}(*)$ with $?=\vide,0$ or $1$ and $*=N$ or $N,M$, we write $Y_{?}(*)$ for $Y(U)$ and $X_{?}(*)$ for $X(U)$. For $U\subset\G(\A_{\Q}^{(\infty)})$, we write $U=U_{p}U^{(p)}$ with $U_{p}=U\cap\G(\qp)$ and $U^{(p)}\subset\G(\A_{\Q}^{(p\infty)})$ and we denote by $\Sigma(U)$ the finite set of finite places $\ell$ such that $U_{\ell}$ is not compact open maximal.
\subsubsection{Hecke correspondences and twisted projections}
Let $g$ be an element of $\G(\A_{\Q}^{(\infty)})$. Right multiplication by $g$ induces a finite flat $\Q$-morphism
\applicationsimple{[\cdot g]}{X(U\cap gUg^{-1})}{X(g^{-1}Ug\cap U)}
which defines the Hecke correspondence $T(g)=[UgU]$ on $X(U)$.
\begin{equation}\label{DiagDefHecke}
\xymatrix{
X(U\cap gUg^{-1})\ar[r]^{[\cdot g]}\ar[d]&X(g^{-1}Ug\cap U)\ar[d]\\
X(U)\ar@{-->}[r]^{[UgU]}&X(U)
}
\end{equation}
For $\ell$ a prime number and $a\in\idele{\Q}$ a finite idèle, we denote by $T(\ell)$ the Hecke correspondence $[U\matrice{\ell}{0}{0}{1}U]$ and by $\diamant{a}$ the diamond correspondence $[U\matrice{a}{0}{0}{a}U]$. The full classical Hecke algebra $\hgot(U)$ of level $U$ is the $\Z$-algebra generated by Hecke and diamond correspondences acting on $X(U)$. 

Let $U$ be a compact open subgroup and let $\ell$ be a finite prime which does not belong to $\Sigma(U)\cup\{p\}$, so that in particular $U_{\ell}$ is a maximal compact open subgroup of $\G(\Q_{\ell})$. For all $g\in\G(\Afiniq)$, the inclusion $U'=U\cap U_{1}(\ell)\subsetneq U$ induces a twisted projection
\begin{equation}\nonumber
\pi_{U',U,g}:X(U')\fleche X(U)
\end{equation}
which is the composition of the right-multiplication $[\cdot g]$ with the natural projection from $X(U')$ to $X(U)$. The twisted projection $\pi_{U',U,g}$ depends only on the image of $g$ in $\G(\Afiniq)/U$. If $x$ belongs to $\Q_{\ell}$ and if 
\begin{equation}\nonumber
g=\matrice{x}{0}{0}{1}
\end{equation}
in the sense that $g$ is the identity except at $\ell$ where it has the shape above, we denote $\pi_{U',U,g}$ by $\pi_{U',U,x}$.
\subsubsection{Cohomology}
\paragraph{Betti and étale cohomology}\label{SubBetti}Let $\pi:E\fleche Y(N)$ be the universal elliptic curve over $Y(N)$ and let $\bar{\pi}:\bar{E}\fleche X(N)$ be the universal generalized elliptic curve over $X(N)$. For $k\geq2$ an integer, let $\Hcal_{k-2}$ be the local system $\Sym^{k-2}R^{1}\pi_{*}\Z$ on $Y(N)(\C)$ and let $\Fcal_{k-2}$ be $j_{*}\Hcal_{k-2}$.

 If $N\geq3$, let $\RGamma_{B}(X(N)(\C),\Fcal_{k-2})$ be the singular cohomology complex of the complex points of $X(N)$. If $X$ is a quotient curve $G\backslash X(N)$ with $N\geq3$ under the action of a finite group $G$ and if $A$ is a ring in which $|G|$ is invertible, we denote by $H^{i}(X(\C),\Fcal_{k-2}\tenseur_{\Z}A)$ the cohomology group $H^{i}(X(N)(\C),\Fcal_{k-2}\tenseur_{\Z}A)^{G}$ and note that it is also the cohomology of the complex $\RGamma_{B}(X(\C),\Fcal_{k-2}\tenseur_{\Z}A)$ where $X$ is seen as a Deligne-Mumford stack over $A$ (in particular $H^{i}(X(\C),\Fcal_{k-2}\tenseur_{\Z}A)$ is independent of the choice of $N$ and $G$). We denote by $\RGamma_{\et}(X\times_{\Q}\Qbar,\Fcal_{k-2}\tenseur_{\Z}\zp)$ the complex computing the cohomology groups 
 \begin{equation}\nonumber
H^{i}_{\et}(X\times_{\Q}\Qbar,\Fcal_{k-2}\tenseur_{\Z}\zp)=H^{i}(X(\C),\Fcal_{k-2}\tenseur_{\Z}\zp).
\end{equation}
As usual, we denote by 
\begin{equation}\nonumber
\Mcal_{k}(U(N))=H^{0}(X(N),\pi_{*}(\Omega^{1}_{\bar{E}/X(N)})^{\tenseur k})
\end{equation}
the space of holomorphic modular forms of weight $k$ and by
\begin{equation}\nonumber
S_{k}(U(N))=H^{0}(X(N),\pi_{*}(\Omega^{1}_{\bar{E}/X(N)})^{\tenseur (k-2)}\tenseur_{\Ocal(X(N))}\Omega^{1}_{X(N)/\Q})
\end{equation}
the space of holomorphic cusp forms.

For $U^{(p)}\subset\G(\A_{\Q}^{(p\infty)})$ a compact open subgroup and $\Ocal$ a discrete valuation ring finite and flat over $\zp$, we denote by $\Htilde^{1}_{c}(U^{(p)},\Ocal)$ the completed cohomology with compact support of \cite{EmertonInterpolationEigenvalues}, that is the direct limit
\begin{equation}\nonumber
\Htilde^{1}_{c}(U^{(p)},\Ocal)=\limproj{s}\ \liminj{U_{p}}\ \Hun_{c}(X(U^{(p)}U_{p})\times_{\Q}\Qbar,\Ocal/\varpi^{s})
\end{equation}
of the étale cohomology groups with compact support over all compact open subgroup $U_{p}$ of $\G(\qp)$ followed by $\varpi$-completion.

\paragraph{Hecke action}The Hecke algebra acts contravariantly on cohomological realizations of $X(U)$. As the Hodge decomposition realizes the $\C$-vector space of complex cusp forms $S_{k}(U)$ as a direct summand of $\Hun(X(U)(\C),\Fcal_{k-2}\tenseur_{\Z}\C)$, the complex Hecke algebra $\hgot(U)\tenseur_{\Z}\C$ acts on $S_{k}(U)$. The $\Z$-submodule $S_{k}(U,\Z)\subset S_{k}(U)$ of cusp forms with integral $q$-expansion is stable under the action of $\hgot(U)$ thereby induced. This defines an action of $\hgot(U)\tenseur_{\Z}A$ on $S_{k}(U,\Z)\tenseur_{\Z}A$ for all ring $A$. The complex $\RGamma_{B}(X(U)(\C),\Fcal_{k-2})$ admits a representation as a bounded below (but not necessarily bounded above) complex of projective $\hgot(U)$-modules. 

An eigenform $f\in S_{k}(U)$ is an eigenvector under the action of all $T(\ell)$. To an eigenform $f$ is attached an automorphic representation $\pi(f)$ of $\G(\A_{\Q}^{(\infty)})$ and the conductor of $f$ is the the conductor $\cid(\pi(f))$as in \cite[Theorem 1]{CasselmanAtkin}. Two eigenforms are equivalent in the sense of Atkin-Lehner if they are eigenvectors for the same eigenvalues for all $T(\ell)$ except possibly finitely many. A newform $f\in S_{k}(U)$ is an eigenform such that for all $g\in S_{k}(U')$ equivalent to $f$ in the sense of Atkin-Lehner, $\cid(\pi(f))$ divides $\cid(\pi(g))$.

We call $\hgot(U)\tenseur_{\Z}\zp$ the classical $p$-adic Hecke algebra and denote it by $\Hecke_{\cl}(U)$. It is a semi-local ring finite and free as $\zp$-module. To an eigenform $f$ is attached a map $\lambda_{f}$ from $\Hecke_{\cl}(U)$ to $\Qbar_{p}$ by $T(\ell)f=\lambda_{f}(T(\ell))f$ and, conversely, we say that a map $\lambda$ from a quotient or sub-algebra of $\Hecke_{\cl}(U)$ to a discrete valuation ring in $\Qbar_{p}$ is modular if there exists an eigenform $f$ such that $\lambda=\lambda_{f}$. The reduced Hecke algebra $\Hecke^{\red}(U)\subset\Hecke_{\cl}(U)$ is the sub $\zp$-algebra generated by the diamond operators and the Hecke operators $T(\ell)$ for $\ell\notin\Sigma(U)\cup\{p\}$ (recall here that $U_{\ell}$ is a maximal compact open subgroup if $\ell\notin\Sigma(U)$). The new Hecke algebra $\Hecke^{\new}(U)$ is the quotient of $\Hecke_{\cl}(U)$ acting faithfully on the space of newforms of level $U$. Both $\Hecke^{\red}(U)$ and $\Hecke^{\new}(U)$ are finite, flat, reduced, semi-local $\zp$-algebras.

If $V_{p}\subset U_{p}\subset\G(\qp)$ and $U^{(p)}\subset\G(\A_{\Q}^{p\infty})$ are compact open subgroups, there is a canonical surjection $\Hecke^{\red}(V_{p}U^{(p)})\surjection\Hecke^{\red}(U_{p}U^{(p)})$. Denote by $\Hecke^{\red}(U^{(p)})$ the inverse limit of the $\Hecke^{\red}(U_{p}U^{(p)})$ over all compact open subgroup $U_{p}\subset\G(\qp)$. Then $\Hecke^{\red}(U^{(p)})_{\Ocal}=\Hecke^{\red}(U^{(p)})\tenseur_{\zp}\Ocal$ acts faithfully on $\Htilde^{1}_{c}(U^{(p)},\Ocal)$ for all discrete valuation ring $\Ocal$ finite flat over $\zp$.

\subsection{Motives attached to modular forms}\label{SubMotives}
We review the properties of Grothendieck motives attached to eigencuspforms with coefficients in group used in the formulation of the Tamagawa Number Conjecture of \cite{BlochKato} and its equivariant refinement in \cite{KatoHodgeIwasawa}.
\subsubsection{Modular motives}
The category $CH(\Q)$ of Chow motives is the pseudo-abelian envelope of the category of proper smooth schemes over $\Q$ with Tate twists inverted and with degree zero correspondences modulo rational equivalence as morphisms. A Chow motive is thus a triplet $(X,e,r)$ where $X/\Q$ is proper and smooth, $e$ is a projector of $CH^{\dim X}(X\times X)_{\Q}$ and $r$ is an integer. For $k\geq2$ and $r\in\Z$, a Chow motive of weight $k-1-2r$ (in the sense that the weight filtration of its Betti or étale realization is pure of weight $k-1-2r$) attached to the modular curve is constructed in \cite{SchollMotivesModular} (see also \cite{WildeshausChow} for a different construction using weight structures).

Let $\bar{E}^{(k-2)}$ be the $(k-2)$-fold fiber product of $\bar{E}$ with itself over $X(N)$. Let $KS_{k}$ be the canonical desingularization of $\bar{E}^{(k-2)}$ constructed in \cite[n°5]{DeligneModulaires} (see also \cite[Section 3]{SchollMotivesModular}). The symmetric group $\Grsym_{k-2}$ acts on $\bar{E}^{(k-2)}$ by permutations, the $(k-2)$-th power of $(\Z/N\Z)^{2}$ acts by translation and $\mu_{2}^{k-2}$ acts by inversion in the fibers. Let $\tilde{G}_{k-2}$ be the wreath product of $((\Z/N\Z)^{2}\rtimes\mu_{2})^{k-2}$ with $\Grsym_{k-2}$. Then $\tilde{G}_{k-2}$ acts by automorphisms on $\bar{E}^{(k-2)}$ and thus on $KS_{k}$. Let $\epsi$ be the character of $\tilde{G}_{k-2}$ which is trivial on $(\Z/N\Z)^{2(k-2)}$, the product map on $\mu_{2}^{k-2}$ and signature on $\Grsym_{k-2}$. Let $\Pi_{\epsi}\in\Z[\frac{1}{2Nk!}][\tilde{G}_{k-2}]$ be the projector attached to $\epsi$. For $r\in\Z$, the triplet $(KS_{k},\Pi_{\epsi},r)$ is a Chow motive which we denote by $\Wcal_{N}^{k-2}(r)$. Let $^{B}\Wcal_{N}^{k-2}(r)$ (resp. $^{\et}\Wcal_{N}^{k-2}(r)$ resp. $^{\dR}\Wcal_{N}^{k-2}(r)$) be its Betti (resp. étale $p$-adic resp. de Rham) realization. 

For a number field $L$, a Grothendieck motive over $\Q$ with coefficients in $L$ is an object in the category of motives over $\Q$ in which $\Hom(h(X),h(Y))$ is the group of algebraic cycles on $X\times Y$ of codimension $\dim Y$ tensored over $\Q$ with $L$ modulo homological equivalence. If $M$ is a Grothendieck motive, we denote by $M^{*}(1)$ its Cartier dual motive.

Fix a number field $F$ containing all the eigenvalues of Hecke operators acting on eigenforms in $S_{k}(U(N))$. The image of  $\Wcal_{N}^{k-2}$ in the category of Grothendieck motive over $\Q$ with coefficients in $F$ decomposes under the action of the Hecke correspondences.  Let $f\in S_{k}(U_{1}(N))$ be an eigencuspform and denote by $\lambda_{f}$ the corresponding modular map. Let $\Wcal(f)(r)$ be the largest Grothendieck sub-motive of $\Wcal_{N}^{k-2}(r)$ over $\Q$ with coefficients in $F$ on which $\Hecke^{\red}(U(N))$ acts through $\lambda_{f}$. For $\s:F\plonge\C$ an embedding of $F$ into $\C$, we denote by $\Wcal(f)_{B,\s}(r)$ (resp. $\Wcal(f)_{\dR}(r)$, resp. $\Wcal(f)_{\et,p}(r)$) the Betti (resp. de Rham, resp. $p$-adic étale) realization of $\Wcal(f)(r)$.

\subsubsection{Realizations and comparison theorems}
By \cite[Theorem 1.2.1]{SchollMotivesModular} and the comparison theorems in cohomology, $\Wcal_{N}^{k-2}(r)$ and its realizations satisfy the following compatibilities. 

\paragraph{Betti realization} There is a canonical (and in particular Hecke-equivariant) isomorphism of $\Q[\Gal(\C/\R)]$-modules
\begin{align}\nonumber
^{B}\Wcal_{N}^{k-2}(r)=(2\pi i)^{r}H^{k-1}(KS_{k}(\C),\Q)(\epsi)\isocan(2\pi i)^{r}\Hun(X(N)(\C),\Fcal_{k-2}\tenseur_{\Z}\Q)
\end{align}
between the Betti realization of $\Wcal_{N}^{k-2}(r)$ and the Betti cohomology of the modular curve with coefficients in $\Fcal_{k}\tenseur_{\Z}\Q$. For $\s:F\plonge\C$, this isomorphism induces an isomorphism
\begin{align}\nonumber
\Wcal(f)_{B,\s}(r)\isocan(2\pi i)^{r}\Hun(X(N)(\C),\Fcal_{k-2}\tenseur_{\Z}\Q)(f)
 \end{align}
on the $f$-part of both sides. Here, the right-hand side is the quotient is the largest quotient on which $\Hecke^{\red}(U(N))$ acts through $\s\circ\lambda_{f}$.
  
\paragraph{de Rham realization} There is a canonical isomorphism of $\Q$-vector spaces
\begin{equation}\nonumber
^{\dR}\Wcal_{N}^{k-2}(r)\isocan H^{k-1}_{\dR}(KS_{k}/\Q)(\epsi)(r).
\end{equation} 
Hence, the de Rham realization $^{\dR}\Wcal_{N}^{k-2}(r)$ is equipped with a decreasing filtration $\Fil^{i}(^{\dR}\Wcal_{N}^{k-2}(r))$ satisfying
\begin{equation}\nonumber
\Fil^{i}\left(^{\dR}\Wcal_{N}^{k-2}(r)\right)=\begin{cases}
H^{k-1}_{\dR}(KS_{k}/\Q)(\epsi)(r)&\textrm{ if $i+r\leq0$,}\\
S_{k}(U(N))&\textrm{ if $1\leq i+r\leq k-1$,}\\
0&\textrm{ if $i+r>k-1$.}
\end{cases}
\end{equation}
There is a canonical comparison isomorphism of $\C$-vector spaces
\begin{equation}\label{EqCompBettideRham}
^{\dR}\Wcal_{N}^{k-2}(r)\tenseur_{\Q}\C\isocan\ \!^{B}\Wcal_{N}^{k-2}(r)\tenseur_{\Q}\C
\end{equation}
compatible with the action of $\Gal(\C/\R)$ on $\C$ on the left-hand side and the diagonal action on the right-hand side. When $1\leq r\leq k-1$, the isomorphism \eqref{EqCompBettideRham} induces on $\Fil^{0}(\ \!^{\dR}\Wcal_{N}^{k-2}(r))$ a complex period map
\begin{equation}\label{EqPerCun}
S_{k}(U(N))\tenseur_{\Q}\R\plonge\left[(2\pi i)^{r}\Hun(X(N)(\C),\Fcal_{k-2}\tenseur_{\Z}\Q)\tenseur_{\Q}\C\right]^{+}.
\end{equation}
Composed with surjection on the quotient by $\left[(2\pi i)^{r}\Hun(X(N)(\C),\Fcal_{k-2}\tenseur_{\Z}\Q)\tenseur_{\Q}\R\right]^{+}$, this map yields an isomorphism 
\begin{equation}\label{EqPerC}
\per_{\C}:S_{k}(U(N))\tenseur_{\Q}\R\simeq(2\pi i)^{r-1}\Hun(X(N)(\C),\Fcal_{k-2}\tenseur_{\Z}\Q)^{(-1)^{r-1}}\tenseur_{\Q}\R.
\end{equation}
whose target is $\!^{B}\Wcal_{N}^{k-2}(r-1)^{+}\tenseur_{\Q}\R$. For $\s:F\plonge\C$, the complex period map is compatible with projection on both sides onto the largest quotient on which $\Hecke^{\red}(U(N))$ acts through $\s\circ\lambda_{f}$.
 
 \paragraph{$p$-adic étale realization}
There is a canonical isomorphism of $\qp[\Gal(\Qbar/\Q)]$-modules
 \begin{align}\nonumber
 ^{\et}\Wcal_{N}^{k-2}(r)=H^{k-1}_{\et}(KS_{k}\times_{\Q}\Qbar,\qp)(\epsi)(r)\isocan\Hun_{\et}(X(N)\times_{\Q}\Qbar,\Fcal_{k-2}\tenseur_{\Z}\qp)(r).
 \end{align}
There is a canonical comparison isomorphism
 \begin{equation}\nonumber
^{B}\Wcal_{N}^{k-2}(r)\tenseur_{\Q}\qp\isocan\ \!^{\et}\Wcal_{N}^{k-2}(r)
\end{equation}
of $\qp[\Gal(\C/\R)]$-modules. The $G_{\qp}$-representation $^{\et}\Wcal_{N}^{k-2}(r)$ is a de Rham $p$-adic representation in the sense of \cite{FontaineDeRham,PeriodesPadiques}. Its de Rham module
\begin{equation}\nonumber
D_{\dR}(^{\et}\Wcal_{N}^{k-2}(r))=H^{0}(G_{\qp},^{\et}\Wcal_{N}^{k-2}(r)\tenseur_{\qp}\BdR)
\end{equation}
is equipped with a descending filtration 
\begin{equation}\nonumber
D^{i}_{\dR}(^{\et}\Wcal_{N}^{k-2}(r))=H^{0}(G_{\qp},^{\et}\Wcal_{N}^{k-2}(r)\tenseur_{\qp}B^{i}_{\dR})
\end{equation}
indexed by $\Z$ such that
\begin{equation}\label{EqDeRham}
D_{\dR}^{i}(^{\et}\Wcal_{N}^{k-2}(r))=\begin{cases}
D_{\dR}(^{\et}\Wcal_{N}^{k-2}(r))&\textrm{ if $i+r\leq0$,}\\
S_{k}(U(N))\tenseur_{\Q}\qp&\textrm{ if $1\leq i+r\leq k-1$,}\\
0&\textrm{ if $i+r>k-1$.}
\end{cases}
\end{equation}
The dual exponential map $\exp^{*}$ of \cite{BlochKato} is a map 
\begin{equation}\nonumber
\exp^{*}:\Hun(G_{\qp},^{\et}\Wcal_{N}^{k-2}(r))\fleche D_{\dR}^{0}(^{\et}\Wcal_{N}^{k-2}(r)).
\end{equation}
For all integers $1\leq r\leq k-1$, the dual exponential map induces an inverse $p$-adic period map
\begin{equation}\label{EqPerP}
\per^{-1}_{p}:\Hun(G_{\qp},^{\et}\Wcal_{N}^{k-2}(r))\fleche S_{k}(U(N))\tenseur_{\Q}\qp.
\end{equation}
The choice of the inverse normalization comes from the fact that the $\Q$-vector spaces of eigencuspforms should be thought as the absolute cohomology underlying all realizations and thus as the source of all comparison maps.

For $\pid\subset\Ocal_{F}$ over $p$, there are variants of all the results above with coefficients in $F_{\pid}$ for the largest quotient on which $\Hecke^{\red}(U(N))$ acts through $\lambda_{f}$.
\subsubsection{Group-algebra coefficients}
Let $K/\Q$ be a finite abelian extension with Galois group $G$. The motive $h^{0}(\Spec K)$ is defined by the compatible system of realizations
\begin{align}\nonumber
&h^{0}(\Spec K)_{B}=\produit{\s:K\plonge\C}{}\Q\simeq\Q[G],\\\nonumber
&h^{0}(\Spec K)_{\et,\ell}=\produit{\s:K\plonge\Qbar}{}\Q_{\ell}\simeq\Q_{\ell}[G],\\\nonumber
&h^{0}(\Spec K)_{\dR}=K.
\end{align}
The group $G$ acts naturally on $h^{0}(\Spec K)_{\dR}$ and the de Rham filtration is given by $\Fil^{0}(K)=K$ and $\Fil^{1}(K)=0$.

Let $M$ be a Grothendieck motive over $\Q$ with coefficients in $L$ and let $\iota:L\plonge\C$ be an embedding of $L$ into $\C$. Denote by
\begin{equation}\nonumber
L(M,s)=\produit{\ell}{}\frac{1}{\Eul_{\ell}(M,\ell^{-s})}=\somme{n=1}{\infty}a_{n}n^{-s}\in\C^{\C}
\end{equation}
the $L$-function of the motive $M$ (though we suppressed it from the notation, $L(M,s)$ depends on the choice of $\iota$). We assume $L(M,s)$ to be well-defined and independent of the choice of the auxiliary prime used to define $\Eul_{\ell}(M,\ell^{-s})$. There exists a motive $M_{K}=M\otimes h^{0}(\Spec K)$ over $\Q$ endowed with a natural action of the group-algebra $L[G]$. In a slight abuse of terminology, we say that $M_{K}$ has coefficients in $L[G]$. Let $S$ be a finite set of rational primes containing $p$ and all the primes of ramification of $K$. For $\ell\notin S$, the Euler factor at $\ell$ of the $p$-adic étale realization of $M_{K}$ is
\begin{equation}\nonumber
\Eul_{\ell}(M_{K},X)=\det(1-\Fr(\ell)X|M_{K,\et,p}^{I_{\ell}})=\produit{\chi\in\hat{G}}{}(1-\chi(\Fr(\ell))\Fr(\ell)X|M_{\et,p}^{I_{\ell}}).
\end{equation}
and the $S$-partial $L$-function of $M_{K}$ is 
\begin{equation}\nonumber
L_{S}(M_{K},s)=\produit{\ell\notin S}{}\frac{1}{\Eul_{\ell}(M_{K},\ell^{-s})}=\left(\somme{(n,S)=1}{}a_{n}\chi(n)n^{-s}\right)_{\chi\in\hat{G}}\in\C[G]^{\C}
\end{equation}
with the natural action of $G$ on $L_{S}(M_{K},s)$. For $\s\in G$, denote by $\N_{\s}$ the set of natural numbers prime to $S$ such that the Artin reciprocity map sends $n\in\N_{\s}$ to $\s$ and write
\begin{equation}\nonumber
L_{S}(M,\s,s)=\somme{n\in\N_{\s}}{}a_{n}n^{-s}
\end{equation}
for the $\s$-component of the $L$-function of $M$. Then $L_{S}(M_{K},\cdot)$ is also equal to 
\begin{equation}\nonumber
\somme{\s\in G}{}L_{S}(M,\s,\cdot)\s\in\C[G]^{\C}.
\end{equation}

\subsection{The ETNC with coefficients in $F[\Gal(\Q_{n}/\Q)]$ and $\Ocal_{\Iw}$}
Let $f\in S_{k}(U(N))$ be an eigencuspform with eigenvalues in a number field $F$. Denote by $M$ the motive $\Wcal(f)$ over $\Q$ with coefficients in $F$ and let $M^{*}$ be the dual motive of $M$. We fix an embedding $\iota:F\plonge\C$ with respect to which we compute the Betti realization, the coefficients of $L$-functions and the comparison theorem(s).

Let $\pid\subset\Ocal_{F}$ be a prime ideal of $F$ above $p$ and denote by $\Ocal$ the ring $\Ocal_{F,\pid}$. Recall that $\Q_{n}/\Q$ is the subfield of $\Q(\zeta_{p^{n+1}})$ with Galois group $G_{n}$ over $\Q$ equal to $\Z/p^{n}\Z$. We denote by $V_{\C,n}$ the Betti realization of $M\times_{\Q}\Q_{n}$. Write $V$ for the étale $p$-adic realization of $M$. For $T\subset V$ a stable $\Ocal[G_{\Q}]$-lattice, the $\Ocal_{\Iw}$-module $T_{\Iw}=T\tenseur_{\Ocal}\Ocal_{\Iw}$ is a $G_{\Q}$-stable module inside the $G_{\Q}$-representation $V_{\Iw}=M_{\et,p}\tenseur_{F_{\pid}}\Lambda_{\Iw}[1/p]$ and likewise $T\tenseur_{\Ocal}\Ocal[G_{n}]$ is a Galois stable sub-module inside $V[G_{n}]=M_{\et,p}\tenseur_{F_{\pid}}F_{\pid}[G_{n}]$ for all $n\in\N$.

Fix integers $n\geq1$ and $1\leq r\leq k-1$.
\subsubsection{Equivariant period maps}
Denote $V(k-r-1)\tenseur_{F_{\pid}}F_{\pid}[G_{n}]$ by $V_{n}(k-r-1)$. The inverse of the $p$-adic period map \eqref{EqPerP} composed with the determinant functor induces an isomorphism
\begin{equation}\nonumber
\xymatrix{
\Det_{F_{\pid}[G_{n}]}\Hun(G_{\qp(\zeta_{p^{n}})},V(k-r))\ar[d]^{\per^{-1}_{p}}\\
\Det_{F_{\pid}[G_{n}]}\left(S(U(N))(f)\tenseur_{F}F_{\pid}[G_{n}]\right)
}
\end{equation}
Here $S(U(N))(f)$ is the largest quotient of $S(U(N))$ on which the Hecke algebra acts through the eigenvalues of $f$ and $\Gal(\Q_{p}(\zeta_{p^{n}})/\Q_{p})$ acts on the cohomology through its quotient $G_{n}$. Taking the tensor product with the determinant of $V_{n}(k-r-1)^{+}$ yields an equivariant $p$-adic period map  
\begin{equation}\label{EqPerDetP}
\xymatrix{
\Det_{F_{\pid}[G_{n}]}\Hun(G_{\qp(\zeta_{p^{n}})},V(k-r))\tenseur_{F_{\pid}[G_{n}]}\Det^{-1}_{F_{\pid}[G_{n}]}(V_{n}(k-r-1))^{+}\ar[d]^{\per^{-1}_{p}}\\
\Det_{F_{\pid}[G_{n}]}\left(S(U(N))(f)\tenseur_{F}F_{\pid}[G_{n}]\right)\tenseur_{F_{\pid}[G_{n}]}\Det^{-1}_{F_{\pid}[G_{n}]}(V_{n}(k-r-1))^{+}.
}
\end{equation}
The complex period map \eqref{EqPerC} composed with the determinant functor induces an equivariant complex period map 
\begin{equation}\label{EqPerDetC}
\xymatrix{
\left[\Det^{}_{\C[G_{n}]}\left(S(U(N))(f)^{}\tenseur_{F}\C[G_{n}]\right)\right]\tenseur\left[\Det^{-1}_{\C[G_{n}]}(V_{\C,n}(k-r-1)^{+}\tenseur_{F}\C)\right]\ar[d]^{\per_{\C}}\\
\C[G_{n}]
}
\end{equation}  
from the $F$-rational subspace
\begin{equation}\label{EqQsub}
\Det^{-1}_{F[G_{n}]}\left(S(U(N))(f)\tenseur_{F}F[G_{n}]\right)\tenseur_{F[G_{n}]}\Det^{}_{F[G_{n}]}(V_{\C,n}(k-r-1))^{+}
\end{equation}  
of the target of \eqref{EqPerDetP} tensored with $\C$ to the group-algebra $\C[G_{n}]$. For $\chi\in\hat{G}_{n}$, we denote by $\per_{\C,\chi}$ the composition of $\per_{\C}$ with $\chi$ seen as group-algebra morphism with values in $\C$. Consequently, to any element $\z$ of the source of \eqref{EqPerDetP} whose image through $\per^{-1}_{p}$ lands in the $F$-rational submodule \eqref{EqQsub} is attached an element $\per_{\C}(\per^{-1}_{p}(\z)\tenseur1)$ of $\C[G]$ and a complex number $\per_{\C,\chi}(\per^{-1}_{p}(\z)\tenseur1)$ for all $\chi\in\hat{G}_{n}$.
\subsubsection{Statement of the ETNC}
In this subsection, we review the statement of the ETNC of \cite[Conjecture 3.2.1]{KatoViaBdR} for the $p$-adic family of motives $\{M_{n}=M_{\Q_{n}}\}_{n\geq1}$ with coefficients in $F[G_{n}]$. To our fixed embedding of $F$ into $\C$ is attached the $p$-partial $G$-equivariant complex $L$-function 
\begin{equation}\nonumber
L_{\{p\}}(M_{n}^{*}(1),s)=\somme{\s\in G_{n}}{}L_{\{p\}}(M^{*}(1),\s,s)\s\in\C[G_{n}]^{\C}
\end{equation} 
of the twisted dual motive of $M$. 

A specialization of $\Ocal_{\Iw}$ is an $\Ocal$-algebras morphism $\psi:\Ocal_{\Iw}\fleche S$ with values in a characteristic zero reduced ring $S$. We denote by $T_{\psi}$ the $G_{\Q}$-representation $T_{\Iw}\tenseur_{\Ocal_{\Iw},\psi}S$ and say it is an $S$-specialization of $T_{\Iw}$.
\begin{Conj}\label{ConjETNC}
There exist a \emph{zeta morphism}
\begin{equation}\nonumber
Z(f)_{\Iw}:V_{\Iw}(-1)^{+}\fleche\RGamma_{\et}(\Z[1/p],V_{\Iw})[1],
\end{equation}
a \emph{fundamental line}
\begin{equation}\nonumber
\Delta_{\Ocal_{\Iw}}(T_{\Iw})\eqdef\Det^{-1}_{\Ocal_{\Iw}}\RGamma_{\et}(\Z[1/p],T_{\Iw})\tenseur_{\Ocal_{\Iw}}\Det^{-1}_{\Ocal_{\Iw}}T_{\Iw}(-1)^{+}\subset\Det_{\Ocal_{\Iw}[1/p]}\Cone(Z(f)_{\Iw})
\end{equation}
and a basis $\z(f)_{\Iw}$ of $\Delta_{\Ocal_{\Iw}}(T_{\Iw})$ called the \emph{zeta element} of $M$ with coefficients in $\Ocal_{\Iw}$. For all $S$-specialization $T_{\psi}$, there exist a morphism
\begin{equation}\nonumber
Z(f)_{\psi}:V_{\psi}(-1)^{+}\fleche\RGamma_{\et}(\Z[1/p],V_{\psi})[1],
\end{equation}
a fundamental line
\begin{equation}\nonumber
\Delta_{S}(T_{\psi})\eqdef\Det^{-1}_{S}\RGamma_{\et}(\Z[1/p],T_{\psi})\tenseur_{S}\Det^{-1}_{S}T_{\psi}(-1)^{+}\subset\Det_{S[1/p]}\Cone(Z(f)_{\psi})
\end{equation}
and a basis $\z_{\psi}$ of $\Delta_{S}(T_{\psi})$. The collection of pairs $(\z_{\psi},\Delta_{S}(T_{\psi}))$ satisfies the following properties.
\begin{enumerate}
\item\label{ItZetaS}There exists a canonical isomorphism $\psi^{\Delta}:\Delta_{\Ocal_{\Iw}}(T_{\Iw})\tenseur_{\Ocal_{\Iw}}S\isocan\Delta_{S}(T_{\psi})$ sending $\z(f)_{\Iw}\tenseur1$ to $\z_{\psi}$.
\suspend{enumerate}
\resume{enumerate}
\item\label{ItInterpolation}The basis $\z_{\psi}$ defines an isomorphism
\begin{equation}\nonumber
\triv_{\psi}:\Delta_{S}(T_{\psi})\isocan S
\end{equation}
compatible with change of rings in the sense that the diagram
\begin{equation}\nonumber
\xymatrix{
\Delta_{S}(T_{\psi})\ar[r]^(0.6){\triv_{\psi}}\ar[d]_{-\tenseur_{S}S'}&S\ar[d]\\
\Delta_{S'}(T_{\phi})\ar[r]^(0.6){\triv_{\phi}}&S'
}
\end{equation}
is commutative whenever $\phi:\Ocal_{\Iw}\fleche S'$ factors through $\psi:\Ocal_{\Iw}\fleche S$. If $S$ is a domain and if the image of $Z(f)_{\psi}$ is not torsion, then $\triv_{\psi}$ coincides with the morphism
\begin{equation}\nonumber
\Delta_{S}(T_{\psi})\subset\Delta_{\Frac(S)}(T_{\psi}\tenseur\Frac(S))\isocan\Cone (Z_{\psi}\tenseur_{S}\Frac(S))\isocan\Frac(S)
\end{equation}
where the last isomorphism is induced by the canonical isomorphism between the determinant of an acyclic complex and the coefficient ring.
\item\label{ItValeurSpeciale} For $n\geq1$ an integer, let $\chi\in\hat{G}_{n}$ be a character with values in $F_{\pid}$ (this can always be achieved by replacing $F$ by a finite extension) and let $1\leq r\leq k-1$ be an integer. Let $\psi$ be the $F_{\pid}$-specialization such that $T_{\psi}$ is equal to $V(k-r)\tenseur_{F_{\pid}}F_{\pid}[G_{n}]$. Then the map
\begin{equation}\nonumber
\xymatrix{
\Delta_{\Ocal_{\Iw}}(T_{\Iw})\ar[d]^{\psi}\\
\Delta_{F_{\pid}}(T_{\psi})\ar[d]^{\loc_{p}}\\
\Det^{}_{F_{\pid}[G_{n}]}\Hun(G_{\qp(\zeta_{p^{n}})},V(k-r))\tenseur_{F_{\pid}[G_{n}]}\Det^{-1}_{F_{\pid}[G_{n}]}T_{\psi}(-1)^{+}\ar[d]^{\per^{-1}_{p}}\\
\Det^{}_{F_{\pid}[G_{n}]}\left(S(U(N))(f)\tenseur_{F}F_{\pid}[G_{n}]\right)\tenseur_{F_{\pid}[G_{n}]}\Det^{-1}_{F_{\pid}[G_{n}]}T_{\psi}(-1)^{+}
}
\end{equation}
obtained by localization at $p$ of the étale cohomology composed with the equivariant $p$-adic period map sends the $F$-submodule generated by $\z(f)_{\Iw}$ into the $F$-rational subspace
\begin{equation}\nonumber
\Det^{}_{F[G_{n}]}\left(S(U(N))(f)\tenseur_{F}F[G_{n}]\right)\tenseur\Det^{-1}_{F[G_{n}]}(V_{\C,n}(k-r-1)\tenseur_{F}F[G_{n}])^{+}.
\end{equation}  
Furthermore 
\begin{equation}\label{EqValeurSpeciale}
\per_{\C}(\per^{-1}_{p}\circ\loc_{p}(\z(f)_{\psi})\tenseur1)=L_{\{p\}}(M_{n}^{*}(1),r)\in\C[G]
\end{equation}
and in particular
\begin{equation}\label{EqValeurChi}
\per_{\C,\chi}(\per^{-1}_{p}\circ\loc_{p}(\z_{\psi})\tenseur1)=L_{\{p\}}(M^{*}(1),\chi,r)\in\C.
\end{equation}
\end{enumerate}
\end{Conj}
\paragraph{Remarks:}
(i) Assertion \ref{ItValeurSpeciale}, and especially equations \eqref{EqValeurSpeciale} and \eqref{EqValeurChi} therein, expresses in which sense conjecture \ref{ConjETNC} predicts the special values of the $L$-function of motivic points and their variations alongside $\Spec\Ocal_{\Iw}[1/p]$. Indeed, once the morphism $Z(f)_{\Iw}$ and the basis $\z(f)_{\Iw}$ of its cone are known, specializations and period maps compute the critical special values of $M^{*}(1)$. Conversely, because specializations extending $\chi^{r}_{\cyc}\chi$ for $1\leq r\leq k-1$ and $\chi\in\hat{G}_{n}$ for some $n\geq1$ form a Zariski-dense subset of $\Hom(\Ocal_{\Iw},\Qbar_{p})$, there can be at most one element of $\Delta_{\Ocal_{\Iw}}(T_{\Iw})\tenseur_{\Ocal_{\Iw}}\Frac(\Ocal_{\Iw})$ which satisfies assertion \ref{ItValeurSpeciale} independently of the truth of the other assertions of the conjecture. In particular, $\z(f)_{\Iw}$ does not depend on the choice of the lattice $T\subset V$. This is coherent with the assertion of the conjecture that $\z(f)_{\Iw}$ is a basis of $\Delta_{\Ocal_{\Iw}}(T_{\Iw})$ as this module also does not depend on the choice of $T$ by Tate's formula.

(ii) Let us assume the truth of assertion \ref{ItValeurSpeciale} independently from the other. Then the cohomology $\RGamma_{\et}(\Z[1/p],T_{\Iw})$ is concentrated in degree 1 and 2, the $\Ocal_{\Iw}$-module $H^{1}_{\et}(\Z[1/p],T_{\Iw})$ is of rank 1 and the $\Ocal_{\Iw}$-module $H^{2}_{\et}(\Z[1/p],T_{\Iw})$ is torsion by \cite[Theorem 12.4]{KatoEuler}. Equivalently, the Weak Leopoldt's Conjecture \cite[Conjecture Section 1.3]{PerrinRiouLpadique} is known for modular motives. By \cite{JacquetShalika,RohrlichNonVanishing}, the complex numbers $L_{\{p\}}(M^{*}(1),\chi,r)$ do not vanish for all $1\leq r\leq k-1$ and all characters $\chi\in\hat{G}_{n}$. Hence $\z(f)_{\Iw}$ is non-zero and it then follows that 
\begin{equation}\nonumber
\Delta_{\Ocal_{\Iw}}(T_{\Iw})\tenseur_{\Ocal_{\Iw}}\Frac(\Ocal_{\Iw})=\Det_{\Frac(\Ocal_{\Iw})}(0)\isocan\Frac(\Ocal_{\Iw}).
\end{equation}
There then exists an element of $\Hun_{\et}(\Z[1/p],T_{\Iw})$, which we also denote by $\z(f)_{\Iw}$ in a slight abuse of notation, such that the image of $\Delta_{\Ocal_{\Iw}}(T_{\Iw})$ inside $\Frac(\Ocal_{\Iw})$ through the natural morphism
\begin{equation}\nonumber
\Delta_{\Ocal_{\Iw}}(T_{\Iw})\plonge\Delta_{\Ocal_{\Iw}}(T_{\Iw})\tenseur_{\Ocal_{\Iw}}\Frac(\Ocal_{\Iw})=\Det_{\Frac(\Ocal_{\Iw})}(0)\isocan\Frac(\Ocal_{\Iw})
\end{equation}
is equal to 
\begin{equation}\nonumber
\Det^{-1}_{\Ocal_{\Iw}}H^{2}_{\et}(\Z[1/p],T_{\Iw})\tenseur_{\Ocal_{\Iw}}\Det^{}_{\Ocal_{\Iw}}H^{1}_{\et}(\Z[1/p],T_{\Iw})/\z(f)_{\Iw}\subset\Frac(\Ocal_{\Iw})
\end{equation}
and hence to
\begin{equation}\nonumber
\carac^{}_{\Ocal_{\Iw}}H^{2}_{\et}(\Z[1/p],T_{\Iw})\tenseur_{\Ocal_{\Iw}}\carac^{-1}_{\Ocal_{\Iw}}H^{1}_{\et}(\Z[1/p],T_{\Iw})/\z(f)_{\Iw}\subset\Frac(\Ocal_{\Iw})
\end{equation}
by the structure theorem for finitely generated torsion modules over regular local rings. Assertion \ref{ItInterpolation} for $\psi$ equal to the identity is then seen to be equivalent to the equality
\begin{equation}\nonumber
\carac_{\Ocal_{\Iw}}H^{2}_{\et}(\Z[1/p],T_{\Iw})\tenseur_{\Ocal_{\Iw}}\carac^{-1}_{\Ocal_{\Iw}}H^{1}_{\et}(\Z[1/p],T_{\Iw})/\z(f)_{\Iw}=\Ocal_{\Iw}\subset\Frac(\Ocal_{\Iw})
\end{equation}
or equivalently
\begin{equation}\nonumber
\carac_{\Ocal_{\Iw}}H^{2}_{\et}(\Z[1/p],T_{\Iw})=\carac_{\Ocal_{\Iw}}H^{1}_{\et}(\Z[1/p],T_{\Iw})/\z(f)_{\Iw}.
\end{equation}
Conjecture \ref{ConjETNC} thus recovers \cite[Conjecture (4.9)]{KatoHodgeIwasawa} for the modular motive $M$ and \cite[Conjecture 12.10]{KatoEuler}. By \cite[Section 17.13]{KatoEuler} (resp. \cite{KatoKuriharaTsuji} and \cite[Théorème 4.16]{ColmezBSD}), it thus also recovers the Iwasawa Main Conjecture \cite[Conjecture 2.2]{GreenbergIwasawaMotives} for $M$ when $M$ has potentially crystalline ordinary reduction at $p$ or equivalently when $\pi(f)_{p}$ is a principal series ordinary representation (resp. when $M$ has potentially semi-stable but not potentially crystalline ordinary reduction at $p$ or equivalently when $\pi(f)_{p}$ is a Steinberg representation). If $f$ belongs to $S_{2}(\Gamma_{0}(N))$ and if the abelian variety attached to $f$ in the Jacobian of $X_{0}(N)$ is a supersingular elliptic curve with $a_{p}(f)=0$, then conjecture \ref{ConjETNC} implies \cite[Conjecture]{KobayashiIMC}.

(iii) If assertion \ref{ItInterpolation} holds, then $\z(f)_{\Iw}$ (and so $\Delta_{\Ocal_{\Iw}}(T_{\Iw})$) determine $\z(f)_{\psi}$ (and so $\Delta_{S}(T_{\psi})$) for all $\psi:\Ocal_{\Iw}\fleche S$. This assertion thus encodes the interpolation property of $\z(f)_{\Iw}$.

(iv) Let $\psi$ be a specialization with values in a discrete valuation ring $S$ such that $\psi$ seen as having values in $\Frac(S)$ is as in assertion \ref{ItValeurSpeciale} with $L_{\{p\}}(M_{n}^{*}(1),r)\neq0$. Then assertion \ref{ItZetaS} for $S=\Ocal_{\Iw}$, assertion \ref{ItInterpolation} for the identity specialization and $\psi$ and assertion \ref{ItValeurSpeciale} together recover the Tamagawa Number Conjecture of \cite{BlochKato} for the motive $M$ twisted by $\psi$.

(v) Attentive readers will have remarked that our statement \eqref{EqValeurChi} uses the normalizations of \cite{KatoViaBdR}, an article which however does not make completely explicit the link between zeta elements and special values in our case of interest, and not the specific treatment of modular motives in \cite{KatoEuler}. The reason for this choice actually lies deep. For a general compact $p$-ring, what the ETNC with coefficients in $\Lambda$ predicts is the existence of a zeta element $\z_{\Lambda}$ with coefficients in $\Lambda$ whose image through a motivic specialization and then through the canonical period maps attached to the specialized motive computes the values of the $L$-function at 0. A bolder conjecture would be to reverse the order of the operations and to ask in addition for the existence of a $\Lambda$-adic period map $\per_{\Lambda}$ such that $\per_{\Lambda}(\z_{\Lambda})$ is a $\Lambda$-adic $L$-function which then interpolates special values after specialization at motivic points. The existence of such a universal normalization of the period maps is known for the family of motives $M_{n}$, and this is the choice made in \cite{KatoEuler}. However, even a precise formulation of the stronger conjecture is typically not known for the families of modular motives parametrized by Hecke algebras and deformations rings we consider in this manuscript so we stuck with the usual statement of the ETNC. For the convenience of the reader, we explain how to pass from the normalization of \cite{KatoViaBdR}, which we follow, to that of \cite{KatoEuler}, which deals with the same objects as we do. First note that \cite[Section 3.2.6]{KatoViaBdR} predicts that the complex period map for the Betti cohomology $N_{B}(-1)^{+}$ of a strictly critical motive $N$ is related to the value of the $L$-function of $N^{*}(1)$ at $0$. As the Betti cohomology appearing in \eqref{EqPerDetC} is the Betti cohomology of $M(k-r)$, our formulation of the conjecture computes the special value of $M^{*}(r+1-k)$ at zero, and hence the special value of the motive attached to the dual eigencuspform $f^{*}$ at $r$ (the dual eigencuspform is the eigencuspform whose eigenvalues are the complex conjugates of those of $f$ or, equivalently, the eigencuspform whose motive is the motive of $f$ with the dual action of the Hecke algebra). Hence, equation \eqref{EqValeurChi} is equivalent to the statement that
\begin{equation}\nonumber
\per_{\C}(\per^{-1}_{p}\circ\loc_{p}(\z(f)_{\Iw})\tenseur1)=L_{\{p\}}(f^{*},\chi,r)\in\C.
\end{equation}
In the comparable equation in \cite[Theorem 12.5]{KatoEuler}, the period map is universally normalized to have $V^{+}$ as its source whereas our period map for the specialization $\chi_{\cyc}^{r}\chi$ is normalized to have $(V\tenseur\chi)(k-r)(-1)^{+}$ as its source (and thus depends on $r$ and $\chi$). We thus expect that, in the formula of \cite[Theorem 12.5]{KatoEuler}, the value of the $L$-function is multiplied by $(2\pi i)^{k-r-1}$ and the $+$-eigenspace is replaced by the $(-1)^{k-r-1}\chi(-1)$-eigenspace, as is indeed the case.

\subsubsection{Review of known results on conjecture \ref{ConjETNC}}\label{SubBiblio}
Thanks to the awe-inspiring results of \cite{KatoEuler} and the remarkable progresses towards the Iwasawa Main Conjecture for modular forms in \cite{KobayashiIMC,PollackSupersingular,EmertonPollackWeston,OchiaiMainConjecture,SkinnerUrban,XinWanIMC}, much of conjecture \ref{ConjETNC} is known. We record here the following theorem.
\begin{TheoEnglish}[Kato, Skinner-Urban, Kobayashi, Wan]\label{TheoBibliographique}
Let $f\in S_{k}(\Gamma_{1}(Np^{s}))$ be a classical eigencuspform with coefficients in $F\subset\Qbar_{p}$. Then there exists a zeta element $\z(f)_{\Iw}\in\Delta_{\Ocal_{\Iw}}(T_{\Iw})\tenseur_{\Ocal_{\Iw}}\Frac(\Ocal_{\Iw})$ satisfying assertion \ref{ItValeurSpeciale} of conjecture \ref{ConjETNC}.  Assume that $f$ satisfies the following properties.
\begin{enumerate}
\item\label{ItemKatoIrr} The residual representation $\rhobar_{f}$ is absolutely irreducible.
\item\label{ItemKatoTechnique} The order of the image of $\rhobar_{f}$ is divisible by $p$.
\item\label{ItemKatoZero} Either $k>2$ or there exists a finite extension $K/\qp$ such that $\rho_{f}|G_{K}$ is crystalline (equivalently $\pi(f)_{p}$ is either principal series or supercuspidal). 
\suspend{enumerate}
Then the trivialization $\Delta_{\Ocal_{\Iw}}(T_{\Iw})\tenseur_{\Ocal_{\Iw}}\Frac(\Ocal_{\Iw})\simeq\Frac(\Lambda_{\Iw})$ induced by $\z(f)_{\Iw}$ sends $\Delta_{\Ocal_{\Iw}}(T_{\Iw})^{-1}$ inside $\Lambda_{\Iw}$. Assume in addition that $f$ satisfies the following properties.
\resume{enumerate}
\item\label{ItemKatoIMC} Either $f$ belongs to $S_{k}(\Gamma_{0}(N)\cap\Gamma_{1}(p^{s}))$, the semisimplification of $\rhobar_{f}|G_{\qp}$ is isomorphic to $\chi\oplus \psi$ with $\chi\neq\psi$ and $a_{p}(f)$ is a $p$-adic unit or $f$ belongs to $S_{2}(\Gamma_{0}(N))$ with $N$ square-free, $F$ is equal to $\Q$ and $a_{p}(f)$ is zero.
\item\label{ItemKatoMonodromy} There exists $\ell\nmid p$ dividing exactly once the Artin conductor of $\rhobar_{f}$. 
\end{enumerate}
Then the trivialization $\Delta_{\Ocal_{\Iw}}(T_{\Iw})\tenseur_{\Ocal_{\Iw}}\Frac(\Ocal_{\Iw})\simeq\Frac(\Lambda_{\Iw})$ induced by $\z(f)_{\Iw}$ sends $\Delta_{\Ocal_{\Iw}}(T_{\Iw})$ to $\Lambda_{\Iw}$.
\end{TheoEnglish}
\begin{proof}
The first assertion is \cite[Theorem 12.5 (1)]{KatoEuler}. Under the supplementary assumptions \ref{ItemKatoIrr}, \ref{ItemKatoTechnique} and \ref{ItemKatoZero}, the fact that  the trivialization of $\Delta_{\Ocal_{\Iw}}(T_{\Iw})\tenseur_{\Ocal_{\Iw}}\Frac(\Ocal_{\Iw})\simeq\Frac(\Lambda_{\Iw})$ induced by $\z(f)_{\Iw}$ sends $\Delta_{\Ocal_{\Iw}}(T_{\Iw})^{-1}$ inside $\Lambda_{\Iw}$, which more concretely means that
\begin{equation}\label{EqDivisibility}
\carac_{\Ocal_{\Iw}}H^{2}_{\et}(\Z[1/p],T_{\Iw})|\carac_{\Ocal_{\Iw}}H^{1}_{\et}(\Z[1/p],T_{\Iw})/\z(f)_{\Iw},
\end{equation}
is \cite[Theorem 12.5 (4)]{KatoEuler}. In fact, in \textit{loc. cit.}, such a statement is proved under the marginally logically stronger hypothesis that the image of $\rhobar_{f}$ contains $\SL_{2}(\zp)$ so we recall briefly how the argument goes under our hypotheses. Denote by $\Fp$ a finite field such that $\GL_{2}(\Fp)$ contains the image $G$ of $\rhobar_{f}$.Then $G$ acts absolutely irreducibly on $\Fp^{2}$ by hypothesis \ref{ItemKatoIrr} and in particular all $G_{\Q}$-stable $\Ocal$-lattices in $V(f)$ are isomorphic. Because $G$ is of order divisible by $p$ according to hypothesis \ref{ItemKatoTechnique}, the classification of subgroups of $\GL_{2}(\Fp)$ shows that $G$ contains a non-trivial unipotent element. That the hypothesis that the image of $\rho_{f}$ contains $\SL_{2}(\zp)$ is invoked in \cite{KatoEuler} to prove \eqref{EqDivisibility} only either because $\SL_{2}(\Fp_{p})$ acts absolutely irreducibly on $\Fp_{p}^{2}$ or because it contains a non-trivial unipotent element. 

That the reverse divisibility 
\begin{equation}\label{EqReverseDivisibility}
\carac_{\Ocal_{\Iw}}H^{1}_{\et}(\Z[1/p],T_{\Iw})/\z(f)_{\Iw}|\carac_{\Ocal_{\Iw}}H^{2}_{\et}(\Z[1/p],T_{\Iw}),
\end{equation}
holds under the assumptions \ref{ItemKatoIMC} and \ref{ItemKatoMonodromy} remains to be shown. Assume first that $\rhobar|G_{\qp}$ is reducible. By \cite[Section 17.13]{KatoEuler} (see especially the short exact sequence at the end of that section), the reverse divisibility \eqref{EqReverseDivisibility} for the eigencuspform $f$ is equivalent to the main conjecture in Iwasawa theory of modular forms of R.Greenberg and B.Mazur; see for instance \cite[Conjecture 7.4]{OchiaiMainConjecture} for a precise statement. Hence, it is true by \cite[Theorem 3.29]{SkinnerUrban} once we check that the hypotheses of this theorem are verified. The hypotheses \textbf{(dist)} and \textbf{(irr)} of \cite[Theorem 3.29]{SkinnerUrban} are true respectively by our assumption \ref{ItemKatoIMC} and assumption \ref{ItemKatoIrr}. The third hypothesis of \cite[Theorem 3.29]{SkinnerUrban} follows from assumption \ref{ItemKatoMonodromy}. The first, fourth and last hypotheses of \cite[Theorem 3.29]{SkinnerUrban} are imposed there in order to establish the divisibility \eqref{EqDivisibility} but we have already checked it holds under our hypotheses.

Now we assume that $\rhobar|G_{\qp}$ is irreducible and that $f$ satisfies the second set of assumptions of assumption \ref{ItemKatoIMC} of the theorem. By \cite[Theorem 7.4]{KobayashiIMC}, the reverse divisibility \eqref{EqReverseDivisibility} for the eigencuspform $f$ is then equivalent to the Iwasawa Main Conjecture of \cite{KobayashiIMC}. Hence, it is known by \cite{XinWanIMC}.
\end{proof}
Despite these results, no non-tautological set of hypotheses is currently known to be sufficient to prove assertion \ref{ItInterpolation} of conjecture \ref{ConjETNC} for specializations $\psi$ such that the $L$-value of $T_{\psi}$ at 0 vanishes at high order (indeed, such a result would imply in particular the Birch and Swinnerton-Dyer Conjecture for modular abelian variety over $\Q$). For this reason, we introduce the following weaker conjecture.
\begin{Conj}\label{ConjIMC}
Let $\z(f)_{\Iw}\in\Hun_{\et}(\Z[1/p],V_{\Iw})$ be the unique class satisfying assertion 3 of conjecture \ref{ConjETNC}. Then for all specializations $\psi:\Ocal_{\Iw}\fleche S$ with values in a characteristic zero reduced ring such that $\z_{\psi}\eqdef\psi(\z(f)_{\Iw})$ is non-zero, the trivialization $\Delta_{S}(T_{\psi})\tenseur_{S}\simeq Q(S)$ induced by $\z_{\psi}$ identifies $\Delta_{S}(T_{\psi})^{-1}$ and $S$.
\end{Conj}
A partial version of conjecture \ref{ConjIMC} is the following.
\begin{Conj}\label{WeakConjIMC}
With the same notations as in conjecture \ref{ConjIMC}, the trivialization $\Delta_{S}(T_{\psi})\tenseur_{S}\simeq Q(S)$ induced by $\z_{\psi}$ sends $\Delta_{S}(T_{\psi})^{-1}$ inside $S$.
\end{Conj}
Under assumptions \ref{ItemKatoIrr}, \ref{ItemKatoTechnique} and \ref{ItemKatoZero} of  \ref{TheoBibliographique}, the results of \cite{KatoEuler} summed up there precisely assert that conjecture \ref{WeakConjIMC} holds. In corollary \ref{CorWeak} below, we replace the assumption \ref{ItemKatoZero} by a weaker assumption on the local residual representation $\rhobar_{f}|G_{\qp}$.
\section{The ETNC with coefficients in Hecke rings}\label{SubHecke}
\subsection{Galois representations}\label{SubGalois}
\subsubsection{Modular levels}
Let $\Fpbar$ be the algebraic closure of $\Fp_{p}$ and let
\begin{equation}\nonumber
\rhobar:G_{\Q}\fleche\GL_{2}(\Fpbar)
\end{equation}
be an irreducible, modular, residual $G_{\Q}$-representation. As in the introduction, denote by $N(\rhobar)$ the Artin conductor of $\rhobar$ (defined as in \cite{SerreConjecture} and hence prime to $p$). 

A compact open subgroup $U^{(p)}$ is said to be allowable (with respect to $\rhobar$) if there exists a maximal ideal $\mgot_{\rhobar}$ of $\Hecke^{\red}(U^{(p)})$ such that
\begin{equation}\nonumber
\begin{cases}
\tr\rhobar(\Fr(\ell))=T(\ell)\modulo\mgot_{\rhobar}\\
\det\rhobar(\Fr(\ell))=\ell\diamant{\ell}\modulo\mgot_{\rhobar}
\end{cases}
\end{equation}
for all $\ell\notin\Sigma(U^{(p)})\cup\{p\}$ (we recall that $\Sigma(U^{(p)})$ is the set of primes out of which $U^{(p)}\subset\G(\A_{\Q}^{(p\infty)})$ is maximal). If $U^{(p)}$ is allowable and if $\Sigma=\Sigma(U^{(p)})\cup\{p\}$, we denote by $\Hs(U^{(p)})$ the localization of $\Hecke^{\red}(U^{(p)})$ at $\mgot_{\rhobar}$. By \cite{CarayolNiveau,LivneNiveau}, the $\zp$-algebra $\Hs(U^{(p)})$ depends up to isomorphism only on $\Sigma$ provided $U^{(p)}$ is sufficiently small. We denote by $\Hs$ this common isomorphism class. A finite set of primes $\Sigma\supset\{\ell|N(\rhobar)p\}$ is said to be allowable (with respect to $\rhobar$) if there exists an allowable compact open subgroup $U^{(p)}\subset\G(\A_{\Q}^{(p\infty)})$ such that $\Sigma\supset\Sigma(U^{(p)})$ and such that $\Hecke_{\Sigma(U^{(p)})}(U^{(p)},\Iw)=\Hecke_{\Sigma(U^{(p)}),\Iw}$ (or in other words such that $U^{(p)}$ is sufficiently small in the sense given above).  To an inclusion of allowable primes $\Sigma\subset\Sigma'$ is attached a surjective morphism $\Hecke_{\Sigma',\Iw}\fleche\Hs$. The $\zp$-algebra $\Hs$ is flat and its relative dimension is at least 3 (see \cite{GouveaMazur}). It is conjectured (and often known) that this lower bound is sharp.

For $\Sigma$ an allowable finite set of primes, we denote by $\Sigma^{(p)}$ the set $\Sigma(U^{(p)})$ and by
\begin{equation}\nonumber
N(\Sigma)=N(\rhobar)\produit{\ell\in\Sigma^{(p)}}{}\ell^{\dim_{\Fpbar}\rhobar_{I_{\ell}}}.
\end{equation}
the tame level attached to a modular specialization of $\Hs$. For such a $\Sigma$, there exists by definition a pseudocharacter 
\begin{equation}\nonumber
\tr\rho_{\mgot_{\rhobar}}:G_{\Q,\Sigma}\fleche\Hs
\end{equation}
of dimension 2 in the sense of \cite{WilesOrdinaryLambdaAdic,TaylorHilbert,BellaicheChenevier}. 
If $\psi:\Hs\fleche S$ is a map of local $\zp$-algebras, the map
\begin{equation}\nonumber
\tr\rho_{\psi}=\psi\circ\tr\rho_{\mgot_{\rhobar}}
\end{equation}
is a pseudocharacter coinciding with $\tr\rho_{\mgot_{\rhobar}}$ modulo the maximal ideal of $S$. By \cite{CarayolRepresentationsGaloisiennes,NyssenPseudo}, there exists a $G_{\Q,\Sigma}$-representation $(T_{\psi},\rho_{\psi},S)$ unique up to isomorphism whose trace is equal to $\tr\rho_{\psi}$. This holds in particular for $\psi=\Id$ in which case we write $(T_{\Sigma,\Iw},\rho_{\Sigma})$ for $(T_{\Id},\rho_{\Id})$. When $\psi$ has values in a domain, we write $V_{\psi}$ for $T_{\psi}\tenseur_{S}\Frac(S)$.
\subsubsection{Euler factors}\label{SubEulerFactors}
\begin{DefProp}\label{DefPropEulerUnram}
If $\psi:\Hs\fleche S$ is a map of flat $\zp$-algebras, the algebraic Euler polynomial at $\ell\notin\Sigma$ a finite prime of $T_{\psi}$ is 
\begin{equation}\nonumber
\Eul_{\ell}(T_{\psi},X)=\det(1-\Fr(\ell)X|T_{\psi})\in S[X]
\end{equation}
and the algebraic Euler factor at $\ell$ of $T_{\psi}$ is
\begin{equation}\nonumber
\Eul_{\ell}(T_{\psi})=\Eul_{\ell}(T_{\psi},1)\in S.
\end{equation}
If the diagram
\begin{equation}\label{DiagCompEulerUnram}
\xymatrix{
\Hs\ar[r]^{\phi}\ar[d]^{\psi}&S\\
R\ar[ru]_{\kg}
}
\end{equation}
is commutative, then 
\begin{equation}\label{EqCompEulerUnram}
\kg(\Eul_{\ell}(T_{\psi},X))=\Eul_{\ell}(T_{\phi},X).
\end{equation}
\end{DefProp}
\begin{proof}
Everything is clear once observed that $T_{\psi}=T_{\Sigma}\tenseur_{\Hs,\psi}S$ is an $S$-module free of rank 2 with a trivial action of $I_{\ell}$.
\end{proof}
When $S$ is not a domain and $\ell\in\Sigma$, the obvious generalization of proposition-definition \ref{DefPropEulerUnram} need not be true and it is thus not possible in general to extend the definition of the Euler factor at $\ell\in\Sigma$ to this case. For $\aid^{\red}$ a minimal prime ideal of $\Hs$, we denote by $(T(\aid^{\red}),\rho(\aid^{\red}),\Hecke(\aid^{\red}))$ the specialization $(T_{\psi},\rho_{\psi},\Hs/\aid^{\red})$ attached to the natural projection $\psi:\Hs\fleche\Hecke(\aid^{\red})$ and write $V(\aid^{\red})$ for $V_{\psi}$.
\begin{DefProp}\label{DefPropEuler}
If $\psi:\Hecke(\aid^{\red})\fleche S$ is a map of flat $\zp$-algebras with values in a domain, the algebraic Euler polynomial  of $T_{\psi}$ at $\ell\nmid p$ a finite prime is 
\begin{equation}\nonumber
\Eul_{\ell}(T_{\psi},X)=\det\left(1-\Fr(\ell)X|V_{\psi}^{I_{\ell}}\right)\in S[X]
\end{equation}
and the algebraic Euler factor at $\ell$ of $T_{\psi}$ is
\begin{equation}\nonumber
\Eul_{\ell}(T_{\psi})=\Eul_{\ell}(T_{\psi},1)\in S.
\end{equation}
If $\lambda_{f}$ is a modular map factoring through $T(\aid^{\red})$ and if the diagram
\begin{equation}\label{DiagCompEuler}
\xymatrix{
\Hecke(\aid^{\red})\ar[r]^{\lambda_{f}}\ar[d]^{\psi}\ar[rd]^{\phi}&\Qbar_{p}\\
R\ar[r]^{\kg}&S\ar[u]
}
\end{equation}
is commutative, then
\begin{equation}\label{EqCompEul}
\kg(\Eul_{\ell}(T_{\psi},X))=\Eul_{\ell}(T_{\phi},X).
\end{equation}
\end{DefProp}
\begin{proof}
We have to prove that Euler polynomials have coefficients in the ring of coefficients of the representation and that they obey the compatibility \eqref{EqCompEul}.

That $\Eul_{\ell}(T(\aid^{\red}),X)$, which \textit{a priori} is a polynomial with coefficients in the normalization of $\Hecke(\aid^{\red})$, actually has coefficients in $\Hecke(\aid^{\red})$ itself is presumably well-known, but we include a brief proof. Denote by $\Spec^{\cl}\Hecke(\aid^{\red})$ the set of modular maps of $\Hs$ factoring through $\Hecke(\aid^{\red})$. For $\psi\in\Spec^{\cl}\Hecke(\aid^{\red})$, let $D_{\psi}$ be the determinant $\det\rho_{\psi}^{I_{\ell}}$ (in the sense of \cite{ChenevierDeterminant}) and consider the collection of determinants
\begin{equation}\nonumber
\left\{D_{\psi}:G_{\Q_{\ell}}/I_{\ell}\fleche\Qbar_{p}|\psi\in\Spec^{\cl}\Hecke(\aid^{\red})\right\}.
\end{equation}
By the local-global compatibility in Langlands correspondance, $D_{\psi}$ has values in $\psi(\Hecke(\aid^{\red}))$ for all $\psi$ in $\Spec^{\cl}\Hecke(\aid^{\red})$. By Zariski-density of $\Spec^{\cl}\Hecke(\aid^{\red})$ in $\Spec\Hecke(\aid^{\red})$, there thus exists a unique determinant $D$ with values in $\Hecke(\aid^{\red})$ such that $D(1-\Fr(\ell)X)$ is the characteristic polynomial of $\rho(\aid^{\red})^{I_{\ell}}$.

Next we show assertion \eqref{EqCompEul} in the context of \eqref{DiagCompEuler}. It is enough by construction of $T_{\psi}$ and $T_{\phi}$ to show that $T_{\psi}^{I_{\ell}}$ and $T_{\phi}^{I_{\ell}}$ have the same ranks over $R$ and $S$ respectively. In turn, this is implied by the statement that $T(\aid^{\red})^{I_{\ell}}$ and $T_{\lambda_{f}}^{I_{\ell}}$ have the same rank over $R(\aid^{\red})$ and $\Qbar_{p}$ respectively. Non-zero elements of $\Qbar_{p}$ are not in the kernel of $\lambda_{f}$ so if $\s\in I_{\ell}$ acts on $V$ non-trivially through a finite quotient, then its action is also non-trivial on $T_{\lambda_{f}}$. It is thus further enough to prove that $\rank_{R(\aid^{\red})} T(\aid^{\red})^{U}$ is larger than $\rank_{\Qbar_{p}} T_{\lambda_{f}}^{U}$ for $U$ a finite index subgroup of $I_{\ell}$. By Grothendieck's monodromy theorem \cite[Page 515]{SerreTate}, we can choose $U$ such that $V(\aid^{\red})^{U}$ is quasi-unipotent, in which case $\rank_{R(\aid^{\red})}T(\aid^{\red})^{U}$ is at least 1 and is exactly 1 if the monodromy operator is of rank 1. If the action of monodromy on $T(\aid^{\red})$ is trivial, it is also trivial on $T_{\lambda_{f}}$ and we are done. Now suppose monodromy acts non-trivially on $T(\aid^{\red})$. Let $\s\in G_{\Q_{\ell}}$ be a lift of $\Fr(\ell)$. Because the representation $T_{\lambda_{f}}$ is a pure $G_{\Q_{\ell}}$-module by Ramanujan's conjecture (proved for modular forms in \cite[Théorème A]{CarayolHilbert}), the eigenvalues of $\s$ acting on $T_{\lambda_{f}}$ are all non zero. Hence, the eigenvalues of $\s$ acting on $V(\aid^{\red})$ are also non-zero and their quotient is equal to $\ell^{\pm1}$ by the monodromy relation. Hence, the eigenvalues of $\s$ on $T_{\lambda_{f}}$ have different Weil weights so by Ramanujan's conjecture again, there is a non-trivial, hence necessarily rank 1, monodromy operator acting on $T_{\lambda_{f}}$. The rank of $T_{\lambda_{f}}^{U}$ is then at most 1, and so is less than $\rank_{R(\aid^{\red})}T(\aid^{\red})^{U}$.

It remains to show that $\Eul_{\ell}(T_{\psi},X)$ has values in $S[X]$ for all specializations $\psi:\Hecke(\aid^{\red})\fleche S$. Whenever there exists a normal subgroup $U$ of $I_{\ell}$ with finite index such that $T(\aid^{\red})^{U}$ and $T_{\psi}^{U}$ have the same ranks over $\Hecke(\aid^{\red})$ and $S$ respectively, the first part of the proof shows that $\psi(\Eul_{\ell}(T(\aid^{\red}),X))=\Eul_{\ell}(T_{\psi},X)$ and so $\Eul_{\ell}(T_{\psi},X)$ has coefficients in $S$. The existence of such a $U$ does not obtain only if monodromy acts non-trivially on $T(\aid^{\red})$ and trivially on $T_{\psi}$. In that case, $T_{\psi}^{I_{\ell}}$ is a free $S$-module of rank 2 and so the algebraic Euler polynomial of $T_{\psi}$ has coefficients in $S$.
\end{proof}
\begin{DefEnglish}\label{DefXcaliv}
Let $\psi:\Hecke(\aid^{\red})\fleche S$ be a map of flat $\zp$-algebras with values in a domain and let $\ell\nmid p$ be a finite prime. The graded invertible module $\Xcali_{\ell}(T_{\psi})$ is defined as follows.
\begin{equation}\nonumber
\Xcali_{\ell}(T_{\psi})=\begin{cases}
\Det_{S}\RGamma(G_{K_{v}}/I_{v},T_{\psi}^{I_{v}})&\textrm{ if $\rank_{S}T_{\psi}^{I_{v}}\neq1$,}\\
\Det_{S}[S\overset{1-\Fr(v)}{\fleche} S]&\textrm{ if $\rank_{S}T_{\psi}^{I_{v}}=1$.}
\end{cases}
\end{equation}
Here, the complex $[S\overset{1-\Fr(v)}{\fleche} S]$ is placed in degree $0,1$.
\end{DefEnglish}
The module $\Xcali_{\ell}(T)$ recovers the determinant of the unramified cohomology of $T$ when both are defined and is compatible with change of rings provided the rank of inertia invariants remains constant in the sense of the following lemma.
\begin{LemEnglish}\label{LemXcaliBienDef}
Let 
\begin{equation}\nonumber
\xymatrix{
\Hecke(\aid^{\red})\ar[r]^(0.6){\phi}\ar[d]_{\psi}&S'\\
S\ar[ru]_{\kg}
}
\end{equation}
be a commutative diagram of integral flat $\zp$-algebras. If $T_{\psi}^{I_{\ell}}$ is a perfect complex of $S$-modules, then there is a canonical isomorphism
\begin{equation}\nonumber
\Xcali_{\ell}(T_{\psi})\overset{\can}{\simeq}\Det_{S}\RGamma(G_{\Q_{\ell}}/I_{\ell},T_{\psi}^{I_{\ell}}).
\end{equation}
If $\rank_{S'}T_{\phi}^{I_{\ell}}$ is equal to $\rank_{S}T_{\psi}^{I_{\ell}}$, then $\Xcali_{\ell}(T_{\psi})\tenseur_{S,\kg}S'$ is canonically isomorphic to $\Xcali_{\ell}(T_{\phi})$.
\end{LemEnglish} 
\begin{proof}
If $T_{\psi}^{I_{\ell}}$ is a perfect complex of $S$-modules, then so is $\RGamma(G_{\Q_{\ell}}/I_{\ell},T_{\psi}^{I_{\ell}})$ and the determinant $\Det_{S}\RGamma(G_{\Q_{\ell}}/I_{\ell},T_{\psi}^{I_{\ell}})$ is well-defined. The first assertion of the lemma is non-tautological only if $\rank_{S}T_{\psi}^{I_{\ell}}=1$. In that case, a finite projective resolution of $T_{\psi}^{I_{\ell}}$ yields a projective resolution of $(1-\Fr(v))T_{\psi}^{I_{\ell}}$ and computing $\Det_{S}(\Fr(v)-1)T_{\psi}^{I_{\ell}}\tenseur_{S}\Det^{-1}T_{\psi}^{I_{\ell}}$ using these resolutions yields the desired result.

We now assume that $\rank_{S'}T_{\phi}^{I_{\ell}}$ and $\rank_{S}T_{\psi}^{I_{\ell}}$ are equal to $r$. If $r$ is equal to zero, then both $\Xcali_{\ell}(T_{\psi})\tenseur_{S}S'$ and $\Xcali_{\ell}(T_{\phi})$ are canonically isomorphic to $(S',0)$. If $r=1$, they are both canonically isomorphic to $\Det_{S'}[S'\overset{1-\Fr(\ell)}{\fleche}S']$. If $r=2$, then both $T_{\psi}$ and $T_{\phi}$ are unramified so the canonical isomorphism
\begin{equation}\nonumber
\RGamma(G_{\Q_{\ell}}/I_{\ell},T_{\psi}^{I_{\ell}})\Ltenseur_{S}S'\isocan\RGamma(G_{\Q_{\ell}}/I_{\ell},T_{\psi}\tenseur_{S}S')
\end{equation}
yields the result after taking determinant. The second assertion is thus true.
\end{proof}

\subsubsection{Algebraic $p$-adic determinants}\label{SubAlgebraicDeterminants}
\begin{DefEnglish}\label{DefXcali}
Let $\psi:\Hs\fleche S$ be a map of flat $\zp$-algebras. The graded invertible $S$-module $\Xcali(T_{\psi})$ is defined to be
\begin{equation}\nonumber
\Det_{S}\RGamma_{c}(\Z[1/\Sigma],T_{\psi})\tenseur_{S}\produittenseur{\ell\in\Sigma^{(p)}}{}\Xcali_{\ell}(T_{\psi}).
\end{equation}
\end{DefEnglish}
We recall from the appendix that the subscript $c$ denotes étale cohomology compactly supported outside $p$. The properties of local Euler factors given in subsection \ref{SubEulerFactors} imply that the formation of $\Xcali$ often commutes with base-change.
\begin{Prop}\label{PropCompEuler}
If $\lambda_{f}$ is a modular map factoring through $\Hecke(\aid^{\red})$ and if the diagram
\begin{equation}\label{DiagCompEulerDeux}
\xymatrix{
\Hecke(\aid^{\red})\ar[r]^{\lambda_{f}}\ar[d]^{\psi}\ar[rd]^{\phi}&\Qbar_{p}\\
R\ar[r]^{\kg}&S\ar[u]
}
\end{equation}
is commutative, or more generally if $T_{\psi}^{I_{\ell}}$ and $T_{\phi}^{I_{\ell}}$ have the same ranks over $R$ and $S$ respectively,  then there is a canonical isomorphism $\Xcali(T_{\psi})\tenseur_{R,\kg}S\isocan\Xcali(T_{\phi})$.
\end{Prop}
\begin{proof}
By the commutativity of $\RGamma_{c}(\Z[1/\Sigma],-)$ with $-\Ltenseur_{R,\kg}S$, we are reduced to showing that there exists a canonical isomorphism between $\Xcali_{\ell}(T_{\psi})\tenseur_{R,\kg}S$ and $\Xcali_{\ell}(T_{\phi})$ for all $\ell\in\Sigma^{(p)}$. By lemma \ref{LemXcaliBienDef}, this amounts to showing that $T_{\psi}^{I_{\ell}}$ and $T_{\phi}^{I_{\ell}}$ have the same ranks over $R$ and $S$ respectively. This holds by assumption or follows from the proof of proposition-definition \ref{DefPropEuler}.
\end{proof}

 Though $\Xcali(T_{\psi})$ has \textit{a priori} no special relevance for an arbitrary $S[G_{\Q,\Sigma}]$-module $T_{\psi}$, we note that there are by construction canonical isomorphisms 
\begin{align}\label{EqIsoXcaliSel}
\Xcali(T_{\psi})\isocan\Det_{S}\RGamma_{f}(G_{\Q,\Sigma},T_{\psi})\isocan\Det_{S}\RGamma_{\et}(\Z[1/p],T_{\psi})
\end{align}
whenever all the objects appearing in \eqref{EqIsoXcaliSel} are well defined. This is for instance the case if $S$ is a regular local ring or if $T_{\psi}$ is a perfect complex of smooth étale sheaves on $\Spec\Z[1/p]$ (more concretely, if the $S$-module $T_{\psi}^{I_{\ell}}$ has finite projective dimension for all $\ell\in\Sigma^{(p)}$; this holds for instance if $T_{\psi}$ is minimally ramified).

\subsection{Cohomology of the tower of modular curves}
\subsubsection{Modular bases of the Betti cohomology of modular curves}\label{SubSymbols}
Let $U$ be a sufficiently small compact open subgroup of $\G(\A_{\Q}^{(\infty)})$ and let $k\geq2$ and $1\leq r\leq k-1$ be positive integers. Denote by $\Hcal^{*}_{k-2}$ the $\Z$-dual of the sheaf $\Hcal_{k-2}$ on $Y(U)(\C)$ of subsection \ref{SubBetti}. We consider the relative homology group $H_{1}(X(U)(\C),\{\cusps\},\Hcal_{k-2}^{*})$ with respect to cusps and with coefficients in $\Hcal^{*}_{k-2}$. Let $\phi$ be a continuous map from $]0,\infty[$ to $X(U)$ whose image is an arc from $0$ to $\infty$. The sheaf $\phi^{*}\Hcal^{*}$ is a rank 2 constant sheaf on $]0,\infty[$ whose stalk at $x\in]0,\infty[$ is isomorphic to $\Z xi+\Z$. As in \cite[Section 4.7]{KatoEuler}, denote by $(e_{1},e_{2})$ the basis of $\phi^{*}\Hcal^{*}$ such that, at $x\in]0,\infty[$, the stalk of $e_{1}$ is $xi$ and the stalk of $e_{2}$ is $1$ and let $\alpha\in\Gamma(]0,\infty[,\phi^{*},\Hcal^{*}_{k-2})$ be the global section $e_{1}^{r-1}e_{2}^{k-r-1}$. The image of $\alpha$ in $H_{1}(X(U)(\C),\{\cusps\},\Hcal^{*}_{k-2})$ through the map
\begin{equation}\nonumber
H_{1}([0,\infty],\{0,\infty\},\phi^{*}\Hcal_{k-2}^{*})\fleche H_{1}(X(U)(\C),\{\cusps\},\Fcal_{k-2}^{*})
\end{equation}
is denoted by $\delta^{*}_{U}(k,r)$. For $f$ an eigencuspform of weight $k$ and level $U$, the class $\delta^{*}_{U}(f,r)$ is the projection of $\delta^{*}_{U}(k,r)$ to $H_{1}(X(U)(\C),\{\cusps\},\Fcal_{k-2}^{*})(f)$ where this last space is as in previous subsections the largest quotient on which the Hecke algebra acts through $\lambda_{f}$. Applying the pairing between relative homology and compactly supported cohomology then Poincaré duality yields classes $\delta_{U}(k,r)$ and $\delta_{U}(f,r)$ in $\Hun_{\et}(Y(U),\Hcal_{k-2})$ and $\Hun_{\et}(X(U)(\C),\Fcal_{k-2})(f)$ respectively. As in section \ref{SubBetti}, these definitions make sense for any $U$ after tensor product with $\zp$ or $\Q$ by considering stack cohomology or taking invariants under $U/U'$ for $U'$ a sufficiently small normal compact open subgroup of $U$.

The classes $\delta^{*}_{U}(k,r),\delta_{U}(k,r),\delta^{*}_{U}(f,r)$ and $\delta_{U}(f,r)$ admit an alternate construction using completed cohomology solely in terms of the classes $\delta^{*}_{U}(2,1)$ (which do not require the introduction of the global section $\alpha$ and which are thus simply arcs in $Y(U)$) as we now recall. Write $U=U_{p}U^{(p)}$. For $s\geq1$, define $\delta^{*}_{U,s}$ to be the image of the arc $\phi$ in $H_{1}(X(U)(\C),\{\cusps\},\Z/p^{s}\Z)$ and $\delta_{U,s}$ the class corresponding to $\delta^{*}_{U,s}$ through the pairing with compactly supported cohomology and Poincaré duality. If $U'\subset U$, the construction of $\delta_{U,s}(2,1)$ is compatible with the trace map
\begin{equation}\nonumber
H^{1}(X(U')(\C),\Z/p^{s'}\Z)\fleche H^{1}(X(U)(\C),\Z/p^{s}\Z).
\end{equation}
Taking the inverse limit on $s$ and on compact open subgroups $U_{p}\subset\G(\qp)$ thus defines an object $\tilde{\delta}_{U^{(p)}}$ in the Poincaré dual of the completed étale cohomology group $\Htilde^{1}_{c}(U^{(p)},\zp)$. By \cite[Proposition 4.3 and Corollary 4.5]{EmertonUnitaryPadic} (see also \cite[Corollary 2.2.18, 4.3.2 and (4.3.4)]{EmertonInterpolationEigenvalues}), for all newforms $f$ of tame level $N$ and weight $k$ and all $1\leq r\leq k-1$, the class $\deltatilde_{U^{(p)}}$ yields by projection an element in $\Hun(X_{1}(Np^{s})(\C),\Fcal_{k-2})$ which coincides with $\delta_{U_{1}(Np^{s})}(f,r)$.

Let $\Sigma$ be an allowable subset of finite primes. For $\Ocal$ a discrete valuation ring finite and flat over $\zp$, we consider the completed cohomology 
\begin{equation}\nonumber
\Htilde^{1}_{\et}(U_{1}(N(\Sigma))^{(p)},\Ocal)_{\mgot_{\rhobar}}=\limproj{s}\ \limproj{U_{p}}\ H^{1}_{\et}(X(U_{1}(N(\Sigma))^{(p)}U_{p})\times_{\Q}\Qbar,\Ocal/\varpi^{s})_{\mgot_{\rhobar}}
\end{equation}
where the inverse limit on $U_{p}\subset\G(\qp)$ is taken with respect to the trace map. The inclusion $\Gamma\simeq(1+p\zp)\subset\qp\croix$ and the diagonal embedding of $\qp\croix$ in the diagonal torus of $\G(\qp)$ endows $\Htilde^{1}_{\et}(U_{1}(N(\Sigma))^{(p)},\Ocal)_{\mgot_{\rhobar}}$ with a structure of $\Lambda_{\Iw}$-module and thus of $\Hecke_{\Sigma,\Iw}$-module. We denote by $\delta_{\Sigma,\Iw}$ the element $(\deltatilde_{U_{1}(N(\Sigma))}\tenseur\chi_{\cyc}^{-1})^{+}$ and by $M_{\Sigma,\Iw}$ the free $\Hecke_{\Sigma,\Iw}$-module of rank 1 it generates inside $\Htilde^{1}_{\et}(U_{1}(N(\Sigma))^{(p)},\Ocal)_{\mgot_{\rhobar}}$.

If $\psi:\Hs\fleche S$ is a map of $\zp$-algebras, we denote by
\begin{equation}\nonumber
\delta^{\Sigma}_{\psi,\Iw}\in\Htilde^{1}_{\et}(U_{1}(N(\Sigma))^{(p)},\Ocal)_{\mgot_{\rhobar}}\tenseur_{\Hecke_{\Sigma,\Iw},\psi}S
\end{equation}
the image of $\delta_{\Sigma,\Iw}$ in $\Htilde^{1}_{\et}(U_{1}(N(\Sigma))^{(p)},\Ocal)_{\mgot_{\rhobar}}\tenseur_{\Hecke_{\Sigma,\Iw},\psi}S$ and by $M^{\Sigma}_{\psi,\Iw}$ the $S_{\Iw}$-module it generates; which is then $M_{\Sigma,\Iw}\tenseur_{\Hecke_{\Sigma,\Iw},\psi}S$.

Let $\aid^{\red}\in\Spec\Hs$ be a minimal prime ideal. There then exist a unique compact open subgroup $U^{(p)}\subset\G(\A_{\Q}^{(p\infty)})$ containing $U_{1}(N(\Sigma))^{(p)}$ and a minimal prime ideal $\aid\in\Spec\Hecke^{\new}(U)$ such that $\Hs/\aid^{\red}$ embeds in $R(\aid)_{\Iw}=\Hecke^{\new}(U)/\aid$. We denote by
\begin{equation}\nonumber
\psi(\aid):\Hs\fleche\Raid
\end{equation}
the corresponding morphism. Let $(\Taid,\rho(\aid),\R\aid)$ be the $G_{\Q,\Sigma}$-representation attached to $\psi(\aid)$. Denote by $N(\aid)$ the tame Artin conductor of $\Taid\tenseur_{\Raid}\Frac(\Raid)$. According to the equality \eqref{EqCompEul} of proposition-definition \ref{DefPropEuler}, the tame conductor of $(V_{\psi},\rho_{\psi},\Qbar_{p})$ is also equal to $N(\aid)$. Let $\delta(\aid)_{\Iw}$ be the class $\delta^{\Sigma}_{\psi(\aid),\Iw}$ inside $\Htilde^{1}_{\et}(U_{1}(N(\Sigma))^{(p)},\Ocal)_{\mgot_{\rhobar}}\tenseur_{\Hecke_{\Sigma,\Iw}}\Raid$ and let $M(\aid)_{\Iw}$ be the $\Raid$-module it generates. 
\subsubsection{Change of level}\label{SubChangeLevel}
Let $\Sigma\subset\Sigma'$ be a strict inclusion of allowable set of primes and denote for brevity $N(\Sigma)$ and $N(\Sigma')$ by $N$ and $N'$ respectively. Assume first that $\Sigma'=\Sigma\cup\{\ell\}$. For all $i\in\Z$, the twisted projections 
\application{\pi_{U_{1}(N'),U_{1}(N),\ell^{i}}}{X_{1}(N')}{X_{1}(N)}{[z,g]_{U_{1}(N')}}{[z,g\matrice{\ell^{i}}{0}{0}{1}]_{U_{1}(N)}}
induce covariant cohomological maps 
\begin{equation}\nonumber
\pi_{U_{1}(N'),U_{1}(N),\ell^{i}*}:H^{1}_{\et}(X(U_{1}(N')U_{p})\times_{\Q}\Qbar,\Ocal/\varpi^{s})_{\mgot_{\rhobar}}\fleche H^{1}_{\et}(X(U_{1}(N)U_{p})\times_{\Q}\Qbar,\Ocal/\varpi^{s})_{\mgot_{\rhobar}}
\end{equation}
for all $s\geq1$ and $U_{p}\subset\G(\qp)$. These maps are compatible with $s'\geq s$ and $U_{p}'\subset U_{p}\subset\G(\qp)$ and are $\Hs$-equivariant after tensor product of the source with $\Hs$. 

In the following definition, recall that $\chi_{\Gamma}:G_{\Q,\Sigma}\fleche\Lambda_{\Iw}\croix$ is the composition $G_{\Q,\Sigma}\surjection\Gamma\plonge\Lambda_{\Iw}\croix$. 
\begin{DefEnglish}\label{DefProjection}
If $\Sigma\subsetneq\Sigma'=\Sigma\cup\{\ell\}$, define
\begin{equation}\nonumber
\pi_{\Sigma',\Sigma,\ell}:\Htilde^{1}_{\et}(U_{1}(N(\Sigma'))^{(p)},\Ocal)_{\mgot_{\rhobar}}\tenseur_{\Hecke_{\Sigma',\Iw}}\Hs\fleche\Htilde^{1}_{\et}(U_{1}(N(\Sigma))^{(p)},\Ocal)_{\mgot_{\rhobar}}
\end{equation}
to be the map induced by
\begin{equation}\nonumber
1-T(\ell)\chi_{\Gamma}(\Fr(\ell))\pi_{\Sigma',\Sigma,\ell^{-1}*}+\diamant{\ell}\ell\chi_{\Gamma}(\Fr(\ell))^{2}\pi_{\Sigma',\Sigma,\ell^{-2}*}.
\end{equation}
For a general inclusion of allowable set of primes $\Sigma\subset\Sigma'$ sur that $\Sigma'\backslash\Sigma=\{\ell_{1},\cdots,\ell_{m}\}$, let
\begin{equation}\label{EqPiDef}
\pi_{\Sigma',\Sigma}:\Htilde^{1}_{\et}(U_{1}(N(\Sigma'))^{(p)},\Ocal)_{\mgot_{\rhobar}}\tenseur_{\Hecke_{\Sigma',\Iw}}\Hs\fleche\Htilde^{1}_{\et}(U_{1}(N(\Sigma))^{(p)},\Ocal)_{\mgot_{\rhobar}}
\end{equation}
be the composition of the maps $\pi_{\Sigma\cup\{\ell_{1},\cdots,\ell_{j}\},\Sigma\cup\{\ell_{1},\cdots,\ell_{j-1}\},\ell_{j}}$ for $j$ ranging from $m$ to $1$.
\end{DefEnglish} 
Consider as in subsection \ref{SubSymbols} a minimal prime $\aid^{\red}$ of $\Hs$ and the quotient $\Raid$ of $\Hecke^{\new}$ in which $\Hs/\aid$ embeds. Denote by $\Sigma(\aid)$ the set of primes dividing $N(\aid)$, that is to say the set of primes $\ell\nmid p$ at which $\Taid$ (and all its modular specializations) is ramified. For $\ell\in\Sigma$, denote by $e_{\ell}\in\{0,1,2\}$ the valuation at $\ell$ of $N(\Sigma)/N(\aid)$.
\begin{DefEnglish}\label{DefProjectionAid}
If $\Sigma\backslash\Sigma(\aid)=\{\ell\}$, define
\begin{equation}\nonumber
\pi_{\Sigma,\Sigma(\aid),\ell}:\Htilde^{1}_{\et}(U_{1}(N(\Sigma))^{(p)},\Ocal)_{\mgot_{\rhobar}}\tenseur_{\Hs}R(\aid)\fleche\Htilde^{1}_{\et}(U_{1}(N(\aid))^{(p)},\Ocal)_{\mgot_{\rhobar}}
\end{equation}
to be the map induced by
\begin{equation}\nonumber
\begin{cases}
1&\textrm{if $e_{\ell}=0$,}\\
1-T(\ell)\chi_{\Gamma}(\Fr(\ell))\pi_{\Sigma,\Sigma(\aid),\ell^{-1}*}&\textrm{if $e_{\ell}=1$,}\\
1-T(\ell)\chi_{\Gamma}(\Fr(\ell))\pi_{\Sigma,\Sigma(\aid),\ell^{-1}*}+\diamant{\ell}\ell\chi_{\Gamma}(\Fr(\ell))^{2}\pi_{\Sigma,\Sigma(\aid),\ell^{-2}*}&\textrm{else.}
\end{cases}
\end{equation}
For a general inclusion $\Sigma(\aid)\subset\Sigma$ sur that $\Sigma\backslash\Sigma(\aid)=\{\ell_{1},\cdots,\ell_{m}\}$, let
\begin{equation}\label{EqPiDefAid}
\pi_{\Sigma,\Sigma(\aid)}:\Htilde^{1}_{\et}(U_{1}(N(\Sigma))^{(p)},\Ocal)_{\mgot_{\rhobar}}\tenseur_{\Hs}R(\aid)\fleche\Htilde^{1}_{\et}(U_{1}(N(\aid))^{(p)},\Ocal)_{\mgot_{\rhobar}}
\end{equation}
be the composition of the maps $\pi_{\Sigma(\aid)\cup\{\ell_{1},\cdots,\ell_{j}\},\Sigma(\aid)\cup\{\ell_{1},\cdots,\ell_{j-1}\},\ell_{j}}$ for $j$ ranging from $m$ to $1$.
\end{DefEnglish} 
The following proposition is a crucial ingredient in establishing the compatibility of the ETNC with change of levels and Hecke algebras.
\begin{Prop}\label{PropIhara}
Let $\Sigma\subset\Sigma'$ be two allowable set of primes. The map $\pi_{\Sigma',\Sigma}$ induces a map of $\Hs$-modules
\begin{equation}\label{EqMorIharaSigmaprime}
\pi^{\Iw}_{\Sigma',\Sigma}:M_{\Sigma',\Iw}\tenseur_{\Hecke_{\Sigma',\Iw}}\Hs\fleche M_{\Sigma,\Iw}
\end{equation}
which sends $\delta_{\Sigma',\Iw}$ to
\begin{equation}\label{EqMorIharaDeltaprime}
\delta_{\Sigma,\Iw}\produit{\ell\in\Sigma'\backslash\Sigma}{}\Eul_{\ell}(T_{\Sigma,\Iw}).
\end{equation}
If $\Sigma(\aid)\subset\Sigma$ is the smallest allowable set of primes such that $T(\aid)$ is unramified outside $\Sigma(\aid)$, then the map $\pi_{\Sigma,\Sigma(\aid)}$ induces a map
\begin{equation}\label{EqMorIharaSigmaAid}
\pi^{\Iw}_{\Sigma,\Sigma(\aid)}:M_{\Sigma,\Iw}\tenseur_{\Hs}\Raid\fleche M(\aid)_{\Iw}
\end{equation}
which sends $\delta_{\Sigma,\Iw}$ to
\begin{equation}\label{EqMorIharaDeltaAid}
\delta(\aid)_{\Iw}\produit{\ell\in\Sigma}{}\Eul_{\ell}(T(\aid)_{\Iw}).
\end{equation}
\end{Prop}
\begin{proof}
In order to prove that the maps \eqref{EqMorIharaSigmaprime} and \eqref{EqMorIharaSigmaAid} are well-defined and satisfy \eqref{EqMorIharaDeltaprime} and \eqref{EqMorIharaDeltaAid}, it is enough to compute $\pi_{\Sigma',\Sigma}(\delta_{\Sigma',\Iw})$ and $\pi_{\Sigma,\Sigma(\aid)}(\delta_{\Sigma,\Iw})$ respectively and to compare them with $\delta_{\Sigma,\Iw}$ and $\delta(\aid)_{\Iw}$ respectively. The maps $\pi_{\Sigma',\Sigma,\ell^{i}*}$ are compatible with change of compact open subgroup $U_{p}$ and $s$ so it is enough to carry these comparisons for
\begin{equation}\nonumber
\Hun_{\et}(X_{1}(N(\Sigma')p^{t})\times_{\Q}\Qbar,\Ocal/\varpi^{s}),\ \Hun_{\et}(X_{1}(N(\Sigma)p^{t})\times_{\Q}\Qbar,\Ocal/\varpi^{s})
\end{equation}  with $t,s$ sufficiently large. Up to the isomorphism between
\begin{equation}\nonumber
\Hun(X(U)(\C),\Ocal/\varpi^{s})\simeq H_{1}(X(U)(\C),\{\cusps\},\Ocal/\varpi^{s})
\end{equation}
induced by Poincaré duality, this is then the computation of \cite[Page 558]{EmertonPollackWeston} (in order to check that the Euler factors appearing in \cite{EmertonPollackWeston} are indeed compatible with the statement of the proposition, notice that in the normalizations of this article the pseudocharacter attached to $\mgot_{\rhobar}$ sends the arithmetic, not the geometric, Frobenius morphism to $T(\ell)$, that the diamond operator $\diamant{\ell}$ acts with weight $k$ on modular forms of weight $k$, not weight $k-2$, and that the Euler factor is computed with respect to the specialization $\gamma\mapsto\chi^{-1}_{\cyc}(\gamma)$ of the cyclotomic variable or more concretely at the special value $r=1$, not for the full action of $\Gamma$).
\end{proof}
If $\Taid$ is unramified outside $\Sigma(\aid)\subset\Sigma\subset\Sigma'$, then $\pi_{\Sigma',\Sigma(\aid)}$ factors through $\pi_{\Sigma',\Sigma}$ and $\pi_{\Sigma,\Sigma(\aid)}$. The compatibility \eqref{EqCompEulerUnram} of unramified Euler factors with arbitrary change of ring of coefficients entails that  the image of $\delta_{\Sigma',\Iw}$ through $\pi_{\Sigma',\Sigma(\aid)}$ is equal to its image through $\pi_{\Sigma,\Sigma(\aid)}\circ\pi_{\Sigma',\Sigma}$.

For a more general specialization $\psi:\Hs\fleche S$, different factorizations of $\psi$ yield potentially distinct modular elements $\delta$ attached to $\psi$. First, one can consider $\delta^{\Sigma}_{\psi,\Iw}$ as in subsection \ref{SubSymbols}. Second, one can consider choose a minimal prime ideal $\aid^{\red}$ of $\Hs$ through which $\psi$ factors and consider the image $\delta^{\aid}_{\psi,\Iw}$ of $\delta(\aid)_{\Iw}$ through $\psi$. According to proposition \ref{PropIhara}, the class $\delta^{\aid}_{\psi,\Iw}$ satisfies 
\begin{equation}\nonumber
\delta^{\aid}_{\psi,\Iw}=\delta_{\psi,\Iw}^{\Sigma}\left(\produit{\ell\in\Sigma^{(p)}}{}\psi(\Eul_{\ell}(T(\aid)_{\Iw}))\right)^{-1}.
\end{equation}
Note that when $\psi$ factors through two distinct minimal ideals $\aid$ and $\mathfrak b$, the elements $\delta_{\psi,\Iw}^{\aid}$ and $\delta_{\psi,\Iw}^{\mathfrak b}$ may (and typically do) differ. Finally, one could consider 
\begin{equation}\nonumber
\delta_{\psi,\Iw}=\delta_{\psi,\Iw}^{\Sigma}\left(\produit{\ell\in\Sigma^{(p)}}{}\Eul_{\ell}(T_{\psi})\right)^{-1}.
\end{equation}
Note that because of the action of $G_{\Q,\Sigma}$ on $\Lambda_{\Iw}$, the Euler factors $\Eul_{\ell}(T(\aid)_{\Iw})$ and $\Eul_{\ell}(T_{\psi})$ are indeed invertible for all $\ell\in\Sigma^{(p)}$. If $\psi$ is a modular point, more generally if there exists a modular point factoring through $\psi$, even more generally if $\rank_{S_{\Iw}} T_{\psi}^{I_{\ell}}$ is equal to $\rank_{R(\aid)_{\Iw}}T(\aid)_{\Iw}^{I_{\ell}}$ for all $\ell\in\Sigma^{(p)}$ then $\delta^{\aid}_{\psi,\Iw}$ and $\delta_{\psi,\Iw}$ coincide. This holds in particular for $\psi$ equal to $\psi(\aid)$. In general, they may differ.
\begin{DefEnglish}
Let $\psi:\Hs\fleche S$ be a specialization with values in a flat $\zp$-algebra and factoring through $\Raid$ for a minimal prime ideal $\aid\in\Spec\Hecke^{\new}$. Denote by $M^{\Sigma}_{\psi,\Iw}$ (resp. $M^{\aid}_{\psi,\Iw}$, resp. $M_{\psi,\Iw}$) the $S_{\Iw}$-module generated by $\delta_{\psi,\Iw}^{\Sigma}$ (resp. $\delta^{\aid}_{\psi,\Iw}$, resp. $\delta_{\psi,\Iw}$).
\end{DefEnglish}
When the context is unambiguous, we often omit the subscript $\Iw$ from the notation $M_{\psi,\Iw}$ and $\delta_{\psi,\Iw}$.
\subsection{Zeta elements and zeta morphisms}
\newcommand{\Heckex}{\Hecke^{x}}
\newcommand{\Tsigmax}{T^{x}_{\Sigma,\Iw}}
In this subsection, we fix $\Sigma\supset\{\ell|N(\rhobar)p\}$ an allowable set of primes and a modular specialization $\phi:\Hs\fleche\Qbar_{p}$ of weight $k$; that is to say a system of eigenvalues attached to a classical eigencuspform $f_{\phi}$ of weight $k\geq2$. The kernel of the extension of $\phi$ to $\Hs$ is then a classical point of $\Spec\Hs[1/p]$. By definition, there then exists an integer $t$ such that $\phi$ factors through $\psi_{x}:\Hs\fleche\Hs^{x}$ where $\Hs^{x}$ is the Iwasawa-algebra attached to the local factor $\Heckex_{\Sigma}$ of the classical Hecke algebra $\Hecke^{\cl}(U(N(\Sigma)p^{t}))$ acting on $\Hun_{\et}(X(N(\Sigma)p^{t})\times_{\Q}\Qbar,\Fcal_{k-2})_{\mgot_{\rhobar}}$. The eigencuspform $f_{\phi}$ is a newform of some tame level $N(\aid_{x})|N(\Sigma)$ so there exists a unique minimal prime $\aid_{x}^{\red}$ of $\Heckex_{\Sigma}$ and unique minimal prime $\aid_{x}$ of $\Hecke^{\new}(U(N(\aid_{x})p^{t}))$ acting on $\Hun_{\et}(X(N(\aid)p^{t})\times_{\Q}\Qbar,\Fcal_{k-2})_{\mgot_{\rhobar}}$ such that $\phi$ viewed as a morphism from $\Heckex_{\Sigma}$ factors through
\begin{equation}\nonumber
\Heckex_{\Sigma}/\aid_{x}^{\red}\plonge R(\aid_{x})=\Hecke^{\new}(U(N(\Sigma)p^{t}))/\aid_{x}.
\end{equation}
Denote this map by $\psi(\aid_{x}):\Heckex_{\Sigma}\fleche R(\aid_{x})$. Write $(T_{\Sigma}^{x},\rho_{\Sigma}^{x},\Heckex_{\Sigma})$ and $(T(\aid_{x}),\rho(\aid_{x}),R(\aid_{x}))$ for the $G_{\Q,\Sigma}$-representations attached to $\psi_{x}$ and $\psi(\aid_{x})$ respectively. We also write $M^{x}_{\Sigma,\Iw}$ for the image of $M_{\Sigma,\Iw}$ through $\psi_{x}$ and $M(\aid_{x})_{\Iw}$ for its image through $\psi(\aid_{x})$. 

Our aim is to construct so called fundamental lines, zeta elements and zeta morphisms with coefficients in $\Hecke_{\Sigma,\Iw}$, $\Raid$ $\Heckex_{\Sigma,\Iw}$ and $R\aidIwx$; that is to say trivializations of free modules of rank 1 constructed from étale cohomology of deformations of $\rhobar$ and completed cohomology. As the zeta element $\z(f)_{\Iw}$ computes the special values of the $L$-function of $f$ (and its cyclotomic twists), one way to understand these objects is to consider them as $p$-adic interpolation of special values in families parametrized by Hecke algebra. We also prove an important part of the ETNC as formulated in \cite{KatoViaBdR}; namely that it is compatible with arbitrary specializations.

At first glance, the precise statements and proofs of this property in subsection \ref{SubSigmax} and \ref{SubRaidx} seem very similar. This is, however, partly deceptive: the fact that zeta morphisms with coefficients in reduced Hecke algebra are compatible with base change is mostly formal, whereas the similar property for zeta morphism with coefficients in integral domains is much subtler and requires the full strength of the Weight-Monodromy Conjecture as well as delicate properties of completed cohomology. On the other hand, the latter is also a much more precise result, as it takes into account Euler factors at places of bad reduction.  
\subsubsection{$p$-adic zeta elements}\label{SubZetaPadic}
Let $n\geq1$ be an integer. Write $N$ for $N(\Sigma)$ for brevity and fix an integer $M\geq1$. 

To $(\alpha,\beta)\in(\frac{1}{N}\Z/\Z)^{2}$ is attached in \cite[Section 1.4]{KatoEuler} a Siegel unit
\begin{equation}\nonumber
g_{\alpha,\beta}\in\Ocal(Y(Np^{n}))\croix\tenseur_{\Hecke^{\red}(Np^{n})}Q(\Hecke^{\red}(Np^{n}))
\end{equation}
satisfying remarkable distribution properties. As in \cite[Section 2.2]{KatoEuler}, the zeta element $\z_{Mp^{n},Np^{n}}$ is defined to be
\begin{equation}\nonumber
\z_{Mp^{n},Np^{n}}=\left\{g_{1/Mp^{n},0},g_{0,1/Np^{n}}\right\}\in K_{2}(Y(Mp^{n},Np^{n}))\tenseur_{\Hecke^{\red}(Mp^{n},Np^{n})}Q(\Hecke^{\red}(Mp^{n},Np^{n})).
\end{equation}
If $m\geq n$, the morphism induced by the norm map
\begin{equation}\nonumber
K_{2}(Y(Mp^{m},Np^{m}))\fleche K_{2}(Y(Mp^{n},Np^{n}))
\end{equation}
sends $\z_{Mp^{m},Np^{m}}$ to $\z_{Mp^{n},Np^{n}}$. Hence, the system $\{\z_{Mp^{n},Np^{n}}\}_{n\geq1}$ is a compatible system in the inverse limit $\limproj{n}\ K_{2}(Y(Mp^{n},Np^{n}))$ (after tensor product with the total ring of fractions of the reduced Hecke algebra). 

Following \cite[Section 8.4 and 8.9]{KatoEuler}, the composition of the Chern class map
\begin{equation}\nonumber
K_{2}(Y(Mp^{n},Np^{n}))\fleche H^{2}_{\et}(Y(Mp^{n},Np^{n}),\Z/p^{s}\Z)(2)
\end{equation}
with the spectral sequence degeneracy map 
\begin{equation}\nonumber
H^{2}_{\et}(Y(Mp^{n},Np^{n}),\Z/p^{s}\Z)\fleche\Hun(G_{\Q},\Hun_{\et}(Y(Mp^{n},Np^{n})\times_{\Q}\Qbar,\Z/p^{s}\Z))
\end{equation}
and the projection from $Y(Mp^{n},Np^{n})$ to $Y_{1}(Np^{n})\tenseur\Q(\zeta_{p^{n}})$ sends $\z_{Mp^{n},Np^{n}}$ to an element 
\begin{equation}\nonumber
\z_{Np^{n}}\in\Hun_{\et}(\Z[1/\Sigma,\zeta_{p^{n}}],\Hun_{\et}(Y_{1}(Np^{n})\times_{\Q}\Qbar,\Z/p^{s}\Z).
\end{equation}
The norm compatibility of the $\z_{Mp^{n},Np^{n}}$ implies that the system $\{\z_{Np^{n}}\}_{n\geq1}$ is compatible with the trace map from $\Q(\zeta_{p^{m}})$ to $\Q(\zeta_{p^{n}})$ and with projection from $Y_{1}(Np^{m})$ to $Y_{1}(Np^{n})$ (see \cite[Propositions 8.7 and 8.8]{KatoEuler}). Localizing at $\mgot_{\rhobar}$ and taking the inverse limit first on $n$ then on $s$ thus yields a zeta element $\z_{\Sigma,\Iw}$ in the space
\begin{equation}\nonumber
\limproj{s}\ \limproj{n}\ \Hun(G_{\Q(\zeta_{p^{n}}),\Sigma},\Hun_{\et}(Y_{1}(Np^{n})\times_{\Q}\Qbar,\Z/p^{s}\Z)_{\mgot_{\rhobar}}).
\end{equation}
The Shimura variety $Y_{1}(Np^{n})$ being an affine curve, there are isomorphisms
\begin{align}\nonumber
\limproj{s}\ \liminj{n}\ \Hun_{c}(Y_{1}(Np^{n})\times_{\Q}\Qbar,\Z_{p})/p^{s}&\simeq\limproj{s}\liminj{n}\ \Hun_{c}(Y_{1}(Np^{n})\times_{\Q}\Qbar,\Z/p^{s}\Z)\\\nonumber
&\simeq\Htildeun_{c}(Y_{1}(Np^{n})\times_{\Q}\Qbar,\zp).
\end{align} 
The observation that the first inverse system on $s$ is Mittag-Leffler and Poincaré duality thus shows that $\z_{\Sigma,\Iw}$ belongs to $\Hun(G_{\Q,\Sigma},\Htildeun_{\et}(U_{1}(N(\Sigma))^{(p)},\zp)_{\mgot_{\rhobar}})$ (see appendix \ref{AppCohomology} for the definition of this cohomology group). In fact, it is well-known that the construction of $\z_{\Sigma,\Iw}$ implies that it is unramified at $\ell\nmid p$ but we will not use this fact.
\begin{DefEnglish}\label{DefZetaUniv}
Denote by $Z_{\Sigma,\Iw}$ the $\Hecke_{\Sigma,\Iw}$-module
\begin{equation}\nonumber
\Hecke_{\Sigma,\Iw}\cdot\z_{\Sigma,\Iw}\subset\Hun_{\et}(\Z[1/\Sigma],\Htildeun_{\et}(U_{1}(N(\Sigma))^{(p)},\zp)_{\mgot_{\rhobar}})
\end{equation} 
generated by $\z_{\Sigma,\Iw}$.
\end{DefEnglish}
We prove in subsection \ref{SubSigmax} below that $Z_{\Sigma,\Iw}$ is a free $\Hecke_{\Sigma,\Iw}$-module of rank 1.
\subsubsection{Coefficients in $\Heckex_{\Sigma,\Iw}$}\label{SubSigmax}
The image of $\z_{\Sigma,\Iw}$ through $\psi_{x}$ yields by \cite[Corollary 4.3.2]{EmertonInterpolationEigenvalues} an element\begin{equation}\nonumber
\z^{x}_{\Sigma,\Iw}\in\Hun_{\et}(\Z[1/\Sigma],T^{x}_{\Sigma})\tenseur_{\Heckex_{\Sigma,\Iw}}Q(\Heckex_{\Sigma,\Iw}).
\end{equation}
which actually lies in $\Hun_{\et}(\Z[1/\Sigma],T^{x}_{\Sigma})\tenseur_{\Heckex_{\Sigma,\Iw}}\Heckex_{\Sigma,\Iw}[1/p]$ by \cite[Section 13.12]{KatoEuler} (and in fact in $\Hun_{\et}(\Z[1/p],T^{x}_{\Sigma})\tenseur_{\Heckex_{\Sigma,\Iw}}\Heckex_{\Sigma,\Iw}[1/p]$ but we will not use this fact).
\begin{LemEnglish}\label{LemDepthSigma}
The complex $\RGamma_{c}(\Z[1/\Sigma],\Tsigmax)$ is a perfect complex of $\Heckex_{\Sigma,\Iw}$-modules acyclic outside degree 1 and 2. Its first cohomology group is torsion-free and its second cohomology group is torsion.
\end{LemEnglish}
\begin{proof}
The functor $\RGamma_{c}(\Z[1/\Sigma],-)$ sends perfect complexes to perfect complexes and $\rhobar$ is irreducible so the first assertion is standard.

Choose $(\x,\y)$ a regular sequence in $\Lambda_{\Iw}$. The isomorphisms
\begin{align}\nonumber
&\RGamma_{c}(\Z[1/\Sigma],T^{x}_{\Sigma,\Iw})\Ltenseur_{\Lambda_{\Iw}}\Lambda_{\Iw}/\x\simeq\RGamma_{c}(\Z[1/\Sigma],T^{x}_{\Sigma,\Iw}/\x)\\\nonumber 
&\RGamma_{c}(\Z[1/\Sigma],T^{x}_{\Sigma,\Iw}/x)\Ltenseur_{\Lambda_{\Iw}/\x}\Lambda_{\Iw}/(\x,\y)\simeq\RGamma_{c}(\Z[1/\Sigma],T^{x}_{\Sigma,\Iw}/(\x,\y))
\end{align}
and the fact that 
\begin{equation}\nonumber
H^{0}(G_{\Q,\Sigma},\rhobar)=0
\end{equation}
show that $H^{1}_{c}(\Z[1/\Sigma],T^{x}_{\Sigma,\Iw})[\x]$ and $H^{1}_{c}(\Z[1/\Sigma],T^{x}_{\Sigma,\Iw}/\x)[\y]$ are zero. Thus $\Hun_{c}(\Z[1/\Sigma],T^{x}_{\Sigma,\Iw})$ is a $\Lambda_{\Iw}$-module of depth at least 2 and so is a free $\Lambda_{\Iw}$-module of finite rank. As $\Heckex_{\Sigma,\Iw}$ is a Cohen-Macaulay local ring, it is free over $\Lambda_{\Iw}$ and so the above argument also shows that $\Hun_{c}(\Z[1/\Sigma],\Tsigmax)$ is of depth 2 as $\Heckex_{\Sigma,\Iw}$-module and in particular torsion-free.

By definition, the modular map $\phi$ factors through $\Heckex_{\Sigma}$.  For all $\ell\nmid p$, there exists a cyclotomic twist of $T_{\phi}$ such that the eigenvalues of $\Fr(\ell)$ acting on $T_{\phi}^{I_{\ell}}$ are of non-zero weights. Hence $H^{0}(G_{\Q_{\ell}},T_{\phi})\tenseur_{\Lambda_{\Iw}}\Frac(\Lambda_{\Iw})$ vanishes for all $\ell$. By Poitou-Tate duality, this implies that the complexes
\begin{equation}\nonumber
\RGamma(G_{\Q_{\ell}},T_{\phi})\Ltenseur_{\Lambda_{\Iw}}\Frac(\Lambda_{\Iw})
\end{equation}
are acyclic and hence that $H^{2}_{c}(\Z[1/\Sigma],T_{\phi})$ and $H^{2}_{\et}(\Z[1/p],T_{\phi})$ become isomorphic after tensor product with $\Frac(\Lambda_{\Iw})$. As $H^{2}_{\et}(\Z[1/p],T_{\phi})$ is $\Lambda_{\Iw}$-torsion by \cite[Theorem 12.4]{KatoEuler}, so is $H^{2}_{c}(\Z[1/\Sigma],T_{\phi})$. The isomorphism
\begin{align}\nonumber
\RGamma_{c}(\Z[1/\Sigma],T^{x}_{\Sigma,\Iw})\Ltenseur_{\Heckex_{\Sigma},\phi}\Qbar_{p}\simeq\RGamma_{c}(\Z[1/\Sigma],T_{\phi})\\\nonumber 
\end{align}
then shows that $H^{2}_{c}(\Z[1/\Sigma],\Tsigmax)$ is $\Heckex_{\Sigma,\Iw}$-torsion.
\end{proof}

Consider $\lambda:\Heckex_{\Sigma,\Iw}\fleche F_{\pid}[G_{n}]$ a modular specialization of $\Heckex_{\Sigma}$ and $M_{\lambda}$ the modular motive attached to $\lambda$. By \cite[Theorem 5.6]{KatoEuler}, $\lambda(\z^{x}_{\Sigma})$ is equal to the product of $\z(f_{\lambda})_{\Iw}$ with Euler factors at $\ell\in\Sigma^{(p)}$ (in fact, $\z(f_{\lambda})_{\Iw}$ is defined in \cite{KatoEuler} as the quotient of $\lambda(\z^{x}_{\Sigma})$ by these Euler factors). More precisely, the equality
\begin{equation}\label{EqSigmaxL}
\per_{\C,\chi}(\per_{p}^{-1}\circ\loc_{p}\circ\lambda(\z^{x}_{\Sigma})\tenseur1)=\left(\produit{\ell\in\Sigma^{(p)}}{}\Eul_{\ell}(T_{\lambda})\right)=L_{\Sigma}(M_{\lambda}^{*}(1),\chi,r)
\end{equation}
holds for all characters $\chi\in\hat{G}_{n}$.
Applying this to $\phi$ and taking into account the fact that $\z(f_{\phi})_{\Iw}$ is non-zero shows that $\z^{x}_{\Sigma,\Iw}$ generates a free $\Heckex_{\Sigma,\Iw}$-module inside $\Hun_{c}(\Z[1/\Sigma],\Tsigmax)[1/p]$.

The morphism
\begin{equation}\nonumber
M^{x}_{\Sigma,\Iw}\tenseur_{\Heckex_{\Sigma,\Iw}}\Heckex_{\Sigma,\Iw}[1/p]\fleche\Hun_{\et}(\Z[1/p],T^{x}_{\Sigma,\Iw})\tenseur_{\Heckex_{\Sigma,\Iw}}\Heckex_{\Sigma,\Iw}[1/p]
\end{equation}
which sends $\delta^{x}_{\Sigma,\Iw}$ to $\z^{x}_{\Sigma,\Iw}$ defines by lifting a morphism of complexes 
\begin{equation}\label{EqZetaSigma}
Z^{x}_{\Sigma,\Iw}:M^{x}_{\Sigma,\Iw}\tenseur_{\Heckex_{\Sigma,\Iw}}\Heckex_{\Sigma,\Iw}[1/p]\fleche\RGamma_{\et}(\Z[1/p],T^{x}_{\Sigma,\Iw})[1]\tenseur_{\Heckex_{\Sigma,\Iw}}\Heckex_{\Sigma,\Iw}[1/p].
\end{equation}
In the above, we view $M^{x}_{\Sigma,\Iw}$ and $\Hun_{\et}(\Z[1/p],\Tsigmax)$ as complexes concentrated in degree 0.
\begin{DefEnglish}\label{DefDeltaSigma}
Define $\Delta_{\Heckex_{\Sigma,\Iw}}(\Tsigmax)$ to be the free $\Heckex_{\Sigma,\Iw}$-module of rank 1
\begin{equation}\label{EqDefDeltaSigma}
\Delta_{\Heckex_{\Sigma,\Iw}}(\Tsigmax)=\Det^{-1}_{\Heckex_{\Sigma,\Iw}}\RGamma_{c}(\Z[1/\Sigma],\Tsigmax)\tenseur_{\Heckex_{\Sigma,\Iw}}\Det^{-1}_{\Heckex_{\Sigma,\Iw}}M^{x}_{\Sigma,\Iw}.
\end{equation}
\end{DefEnglish}
Let  $\psi:\Heckex_{\Sigma,\Iw}\fleche S_{\Iw}$ be a specialization induced by a $\zp$-algebras morphism $\Heckex_{\Sigma}\fleche S$ with values in a reduced, flat $\zp$-algebra and such that $\delta_{\psi,\Iw}$ is non-zero. Specializing $\z^{x}_{\Sigma,\Iw}$ at $\psi$ yields an element
\begin{equation}\nonumber
\z_{\psi}\in\Hun_{\et}(\Z[1/\Sigma],T_{\psi})\tenseur_{S_{\Iw}}S_{\Iw}[1/p].
\end{equation}
As for the case $\psi=\phi$ which is treated in the proof of lemma \ref{LemDepthSigma} above, there is an isomorphism of perfect complexes of $Q(\Heckex_{\Sigma,\Iw})$-modules
\begin{equation}\nonumber
\RGamma_{c}(\Z[1/\Sigma],T_{\psi})\Ltenseur Q(\Heckex_{\Sigma,\Iw})\simeq\RGamma_{\et}(\Z[1/\Sigma],T_{\psi})\Ltenseur Q(\Heckex_{\Sigma,\Iw})
\end{equation}
and $H^{2}_{c}(\Z[1/\Sigma],T_{\psi})\Ltenseur Q(\Heckex_{\Sigma,\Iw})$ vanishes. The morphism sending $\delta_{\psi,\Iw}$ to $\z_{\psi}$ consequently gives as above a morphism 
\begin{equation}\nonumber
Z_{\psi}:M_{\psi}\tenseur_{S_{\Iw}}S_{\Iw}[1/p]\fleche\RGamma_{\et}(\Z[1/p],T_{\psi})[1]\tenseur_{S_{\Iw}}S_{\Iw}[1/p].
\end{equation}%
\begin{DefEnglish}\label{DefDeltaSigmaPsi}
Define $\Delta_{S_{\Iw}}(T_{\psi})$ to be the free $S_{\Iw}$-module of rank 1
\begin{equation}\label{EqDefDeltaSigmaPsi}
\Delta_{S_{\Iw}}(T_{\psi})=\Det^{-1}_{S_{\Iw}}\RGamma_{c}(\Z[1/\Sigma],T_{\psi})\tenseur_{S_{\Iw}}\Det^{-1}_{S_{\Iw}}M_{\psi}.
\end{equation}
\end{DefEnglish}
Denote $\Delta_{S_{\Iw}}(T_{\psi})$ by $\Delta_{\psi}$ for brevity. The various $\Delta_{\psi}$ are compatible with change of allowable levels and specializations.%
\begin{Prop}\label{PropCompDeltaSigma}
The $\Heckex_{\Sigma,\Iw}$-module $\Delta_{\Heckex_{\Sigma,\Iw}}(\Tsigmax)$ is equipped with a canonical morphism 
\begin{equation}\nonumber
\triv_{\Sigma}:\Delta_{\Heckex_{\Sigma,\Iw}}(\Tsigmax)\subset\Delta_{\Heckex_{\Sigma,\Iw}}(\Tsigmax)\tenseur_{\Heckex_{\Sigma,\Iw}}Q(\Heckex_{\Sigma,\Iw})\isocan Q(\Heckex_{\Sigma,\Iw}).
\end{equation}
More generally, if $\psi:\Heckex_{\Sigma,\Iw}\fleche S_{\Iw}$ is a specialization induced by a $\zp$-algebras morphism $\Hs^{x}\fleche S$ with values in a reduced, flat $\zp$-algebra such that $\z_{\psi}$ is a non zero-divisor, then the $S_{\Iw}$-module $\Delta_{S_{\Iw}}(T_{\psi})$ (which we also denote by $\Delta_{\psi}$ for brevity) is endowed with a canonical trivialization isomorphism 
\begin{equation}\nonumber
\triv_{\psi}:\Delta_{\psi}\subset\Delta_{\psi}\tenseur_{S_{\Iw}} Q(S_{\Iw})\isocan Q(S_{\Iw}).
\end{equation}
For all $\psi$ as above, denote by $\triv_{\psi}(\Delta_{\psi})\subset Q(S_{\Iw})$ the image of $\Delta_{\psi}\subset\Delta_{\psi}\tenseur_{S_{\Iw}} Q(S_{\Iw})$ through $\triv_{\psi}$. If the diagram
\begin{equation}\nonumber
\xymatrix{
\Heckex_{\Sigma,\Iw}\ar[r]^{\xi}\ar[d]_{\psi}&S_{\Iw}'\\
S_{\Iw}\ar[ru]_{\phi}
}
\end{equation}
of reduced flat $\zp$-algebra quotients of $\Heckex_{\Sigma,\Iw}$ is commutative, then there is a canonical isomorphism 
\begin{equation}\nonumber
\phi^{\Delta}:\Delta_{\psi}\tenseur_{S,\phi}S'\isocan\Delta_{\xi}
\end{equation}
compatible with $\triv_{\psi}$ and $\triv_{\xi}$ in the sense that the rightmost downward arrow of the diagram
\begin{equation}\label{DiagTrivPsiXi}
\xymatrix{
\Delta_{\psi}\ar[r]\ar[d]_{\phi^{\Delta}(-\tenseur_{S,\phi}S')}&\triv_{\psi}(\Delta_{\psi})\ar[d]^{\phi}\\
\Delta_{\xi}\ar[r]&\triv_{\xi}(\Delta_{\xi})
}
\end{equation}
exists and makes the diagram commutative.

If $\Sigma\subset\Sigma'$ is an inclusion of allowable sets, then there is a canonical isomorphism 
\begin{equation}\nonumber
\Delta_{\Heckex_{\Sigma',\Iw}}(T^{x}_{\Sigma',x})\tenseur_{\Heckex_{\Sigma',\Iw}}\Heckex_{\Sigma,\Iw}\isocan\Delta_{\Heckex_{\Sigma,\Iw}}(\Tsigmax)
\end{equation}
compatible with $\triv$ in the sense that the rightmost downward arrow of the diagram 
\begin{equation}\label{DiagTrivSigma}
\xymatrix{
\Delta_{\Heckex_{\Sigma',\Iw}}(T_{\Sigma',\Iw}^{x})\ar[d]\ar[r]&\triv_{\Sigma'}(\Delta_{\Heckex_{\Sigma',\Iw}}(T_{\Sigma',\Iw}^{x}))\ar[d]\\
\Delta_{\Heckex_{\Sigma,\Iw}}(\Tsigmax)\ar[r]&\triv_{\Sigma}(\Delta_{\Heckex_{\Sigma,\Iw}}(\Tsigmax))
}
\end{equation}
exists and makes the diagram commutative.
\end{Prop}
\begin{proof}
Fix $\psi$ as in the proposition (possibly equal to the identity of $\Heckex_{\Sigma,\Iw}$). Then the complex
\begin{equation}\nonumber
\RGamma_{c}(\Z[1/\Sigma],T_{\psi})\Ltenseur_{S_{\Iw}}Q(S_{\Iw})
\end{equation}
is concentrated in degree 1. As in the proof of lemma \ref{LemDepthSigma}, the $\Lambda_{\Iw}$-module $\Hun_{c}(\Z[1/\Sigma],T_{\psi})$ is of depth 2 and hence free of finite rank. By Poitou-Tate duality, its rank is equal to the rank of $T_{\psi}^{-}$ as $\Lambda_{\Iw}$-module and so is equal to the rank of $M_{\psi}$ as $\Lambda_{\Iw}$-module. Consequently, the cone of the morphism
\begin{equation}\nonumber
Z_{\psi}:M_{\psi}\tenseur_{S_{\Iw}}Q(S_{\Iw})\fleche\RGamma_{c}(\Z[1/\Sigma],T_{\psi})[1]\tenseur_{S_{\Iw}}Q(S_{\Iw})
\end{equation}  
is an acyclic complex and so $\Det_{Q(S_{\Iw})}\Cone Z_{\psi}=\Det_{Q(S_{\Iw})}0$ is canonically isomorphic to $Q(S_{\Iw})$. The inclusion of $\Delta_{\psi}$ in $\Delta_{\psi}\tenseur_{S_{\Iw}}Q(S_{\Iw})$ composed with the canonical isomorphism $\Det_{Q(S_{\Iw})}\Cone Z_{\psi}\isocan Q(S_{\Iw})$ then defines 
\begin{equation}\nonumber
\triv_{\psi}:\Delta_{\psi}\subset\Delta_{\psi}\tenseur_{S_{\Iw}} Q(S_{\Iw})\isocan Q(S_{\Iw}).
\end{equation}
By definition, the classes $\z_{\Sigma,\Iw}^{x}$ and $\delta_{\Sigma,\Iw}^{x}$ are compatible with specializations. Hence, the compatibility of diagram \eqref{DiagTrivPsiXi} amounts to the compatibility of $\RGamma_{c}(\Z[1/\Sigma],-)$ with arbitrary base-change of rings of coefficients and the functorial compatibility of $\Det$ with derived tensor product.

Fix $\Sigma\subset\Sigma'$ an inclusion of allowable sets. Then the isomorphism \eqref{EqMorIharaSigmaprime} of proposition \ref{PropIhara} induces a canonical isomorphism 
\begin{equation}\nonumber
\Det_{\Heckex_{\Sigma,\Iw}}(M^{x}_{\Sigma',\Iw}\tenseur_{\Heckex_{\Sigma',\Iw}}\Heckex_{\Sigma,\Iw})\tenseur_{\Heckex_{\Sigma,\Iw}}\Det^{-1}_{\Heckex_{\Sigma,\Iw}}M^{x}_{\Sigma,\Iw}\isocan\Det_{\Heckex_{\Sigma,\Iw}}[S_{\Iw}\overset{d}{\fleche}S_{\Iw}]
\end{equation}
where the complex $[S_{\Iw}\overset{d}{\fleche} S_{\Iw}]$ is placed in degree 0 and 1 and $d$ is multiplication by
\begin{equation}\nonumber
\produit{\ell\in\Sigma'\backslash\Sigma}{}\Eul_{\ell}(\Tsigmax).
\end{equation}
Hence
\begin{equation}\nonumber
\Det_{\Heckex_{\Sigma,\Iw}}(M^{x}_{\Sigma',\Iw}\tenseur_{\Heckex_{\Sigma',\Iw}}\Heckex_{\Sigma,\Iw})\tenseur_{\Heckex_{\Sigma,\Iw}}\Det^{-1}_{\Heckex_{\Sigma,\Iw}}M^{x}_{\Sigma,\Iw}\isocan\produittenseur{\ell\in\Sigma'\backslash\Sigma}{}\Det^{-1}_{\Heckex_{\Sigma,\Iw}}S_{\Iw}/\Eul_{\ell}(\Tsigmax)
\end{equation}
where each of the $S_{\Iw}/\Eul_{\ell}(\Tsigmax)$ is viewed as a complex in degree 0. On the other hand, the definition and base-change property of $\RGamma_{c}(\Z[1/\Sigma'],-)$ induces a canonical isomorphism between
\begin{equation}\nonumber
\Det_{\Heckex_{\Sigma,\Iw}}\left(\RGamma_{c}(\Z[1/\Sigma'],T_{\Sigma',\Iw}^{x})\Ltenseur_{\Heckex_{\Sigma',\Iw}}{\Heckex_{\Sigma,\Iw}}\right)\tenseur_{\Heckex_{\Sigma,\Iw}}\Det^{-1}_{\Heckex_{\Sigma,\Iw}}\RGamma_{c}(\Z[1/\Sigma],\Tsigmax)
\end{equation}
and
\begin{equation}\nonumber
\produittenseur{\ell\in\Sigma'\backslash\Sigma}{}\Det^{-1}_{\Heckex_{\Sigma,\Iw}}\RGamma(G_{\Q_{\ell}},\Tsigmax)\isocan\produittenseur{\ell\in\Sigma'\backslash\Sigma}{}\Det_{\Heckex_{\Sigma,\Iw}}S_{\Iw}/\Eul_{\ell}(\Tsigmax)
\end{equation}
where each of the $S_{\Iw}/\Eul_{\ell}(\Tsigmax)$ is again seen as a complex in degree 0. Putting these isomorphism together yields a canonical isomorphism 
\begin{equation}\label{EqIsoXsigma}
\Delta_{\Heckex_{\Sigma',\Iw}}(T_{\Sigma',\Iw}^{x})\tenseur_{\Heckex_{\Sigma',\Iw}}\Heckex_{\Sigma,\Iw}\isocan\Delta_{\Heckex_{\Sigma,\Iw}}(T_{\Sigma,\Iw}^{x}).
\end{equation}
The trivialization of $\Delta_{\Heckex_{\Sigma,\Iw}}(T_{\Sigma,\Iw}^{x})$ induced by $\triv_{\Sigma'}$ is induced by the morphism sending $\delta_{\Sigma',\Iw}$ to $\z^{x}_{\Sigma',\Iw}$. After tensor product with $Q(\Heckex_{\Sigma,\Iw})$, it thus sends $\delta_{\Sigma,\Iw}$ to $\z^{x}_{\Sigma,\Iw}$. Hence, the isomorphism \eqref{EqIsoXsigma} is compatible with the trivializations $\triv_{\Sigma'}$ and $\triv_{\Sigma}$.
\end{proof}
We end this subsection by proving that, as announced at the end of subsection \ref{SubZetaPadic}, $\z_{\Sigma,\Iw}$ generates a free $\Hs$-module inside $\Hun_{\et}(G_{\Q,\Sigma},\Htildeun_{\et}(U_{1}(N(\Sigma))^{(p)},\zp)_{\mgot_{\rhobar}})$.
\begin{Prop}\label{PropZnonTorsion}
The $\Hecke_{\Sigma,\Iw}$-module $Z_{\Sigma,\Iw}$ is free of rank 1.
\end{Prop}
\begin{proof}
As $Z_{\Sigma,\Iw}$ is cyclic, it is enough to show that it is a faithful $\Hecke_{\Sigma,\Iw}$-module. Let $m\in\Hecke_{\Sigma,\Iw}$ be an element annihilating $Z_{\Sigma,\Iw}$, let $\psi$ be a modular specialization of $\Hecke_{\Sigma,\Iw}$ corresponding to an eigencuspform $f_{\psi}$ and let $\psi_{y}:\Hecke_{\Sigma,\Iw}\fleche\Hecke^{y}_{\Sigma,\Iw}$ be the specialization with values in the local factor of the classical reduced Hecke algebra of a certain level and weight through which $\psi$ factors. Then the image of $\z_{\Sigma,\Iw}$ through $\psi_{y}$ is equal to $\z^{y}_{\Sigma,\Iw}$ in the notations of this subsection so is a non zero-divisor by \eqref{EqSigmaxL}. Hence, $m$ is in the kernel of $\psi_{y}$. By \cite[Theorem 7.4.2]{EmertonCoates}, this means that $m$ annihilates the subspace of locally $\GL_{2}(\qp)$-algebraic vectors of $\Htildeun_{\et}(U_{1}(N(\Sigma))^{(p)},\zp)_{\mgot_{\rhobar}}$.  As locally $\GL_{2}(\qp)$-algebraic vectors are dense in $\Htildeun_{\et}(U_{1}(N(\Sigma))^{(p)},\zp)_{\mgot_{\rhobar}}$ and as the action of $\Hecke_{\Sigma,\Iw}$ on this module is faithful, $m$ is zero.
\end{proof}

\subsubsection{Coefficients in $R(\aid_{x})_{\Iw}$}\label{SubRaidx}
Subsection \ref{SubSigmax} admits an important variant with $\Heckex_{\Sigma,\Iw}$ replaced with $R\aidIwx$.
\begin{LemEnglish}  
The complex $\RGamma_{\et}(\Z[1/p],T\aidIwx)$ is acyclic outside degree 1 and 2. The $R(\aid_{x})_{\Iw}$-module $\Hun_{\et}(\Z[1/p],T(\aid_{x})_{\Iw})$ is of depth 2 and of rank 1 whereas $H^{2}_{\et}(\Z[1/p],T(\aid_{x})_{\Iw})$ is $R(\aid_{x})_{\Iw}$-torsion.
\end{LemEnglish}\label{LemDepth}
\begin{proof}
That $\RGamma_{\et}(\Z[1/p],T\aidIwx)$ is acyclic outside degree 1 and 2 follows from the irreducibility of $\rhobar$.

Consider $\psi_{x}$ as having values in a discrete valuation ring $\Ocal$ finite and flat over $\zp$ with uniformizing parameter $\varpi$. The short exact sequence 
\begin{equation}\nonumber
0\fleche T_{\psi_{x}}\overset{\varpi}{\fleche} T_{\psi_{x}}\fleche T_{\psi_{x}}/\varpi\fleche 0
\end{equation}
induces an isomorphism between $\Hun(G_{\Q,\Sigma},T(\aid_{x}))[p]$ and $H^{0}(G_{\Q,\Sigma},\rhobar)$ which is zero under our ongoing assumption that $\rhobar$ is absolutely irreducible. For $\gamma$ a topological generator of $\Gamma$, the short sequence 
\begin{equation}\nonumber
\suiteexacte{\gamma-1}{}{T(\aid_{x})_{\Iw}}{T(\aid_{x})_{\Iw}}{T_{\psi_{x}}}
\end{equation}
of étale sheaves on $\Spec\Z[1/p]$ is exact so
\begin{equation}\nonumber
\Hun_{\et}(\Z[1/p],T(\aid_{x})_{\Iw})[\gamma-1]=0
\end{equation}
and $\Hun_{\et}(\Z[1/p],T(\aid_{x})_{\Iw})/(\gamma-1)$ embeds into $\Hun(G_{\Q,\Sigma},T(\aid_{x}))$ so is of depth 1. Hence the $R(\aid_{x})_{\Iw}$-module $\Hun_{\et}(\Z[1/p],T(\aid_{x})_{\Iw})$ is of depth 2. 

As $\Hun_{\et}(\Z[1/p],T(\aid_{x})_{\Iw})\tenseur_{\Lambda_{\Iw}}\Frac(\Lambda_{\Iw})$ is of dimension 1 by \cite[Theorem 12.4]{KatoEuler}, the remaining assertions follow by Poitou-Tate duality.
\end{proof}
By \cite[Section 13.9]{KatoEuler}, there exists a non-zero element
\begin{equation}\nonumber
\z(\aid_{x})_{\Iw}\in\Hun_{\et}(\Z[1/p],T(\aid_{x})_{\Iw})\tenseur_{R(\aid_{x})_{\Iw}}R(\aid_{x})_{\Iw}[1/p]
\end{equation}
such that 
\begin{equation}\label{EqSpecZetaElement}
\psi(\aid_{x})(\z_{\Sigma,\Iw}^{x})=\produit{\ell\in\Sigma^{(p)}}{}\Eul_{\ell}(T\aidIwx)\z\aidIwx.
\end{equation}
This entails that
\begin{equation}\nonumber
\lambda(\z(\aid_{x})_{\Iw}\tenseur1)=\z(f_{\lambda})_{\Iw}
\end{equation}
for all modular specializations
\begin{equation}\nonumber
\lambda:R_{\Sigma}\fleche\bar{\Z}_{p}
\end{equation}
factoring through $R(\aid_{x})$. Concretely, if $M_{\lambda}$ denote the motive attached to $f_{\lambda}$, then for all integers $n$, all modular specializations $\lambda:R(\aid_{x})_{\Iw}\fleche F_{\pid}[G_{n}]$ and all characters $\chi\in\hat{G}_{n}$, the element $\z(\aid_{x})_{\Iw}$ satisfies
\begin{equation}\nonumber
\per_{\C}\circ\per_{p}^{-1}(\loc_{p}(\lambda(\z(\aid_{x})))\tenseur1)=L_{\{p\}}(M^{*}_{\lambda}(1),\chi,r).
\end{equation}
The $R(\aid_{x})_{\Iw}$-module generated by $\z\aidIwx$ is included in $\Hun_{\et}(\Z[1/p],T\aidIwx)$ after localization at all height one prime, so if $\Hun_{\et}(\Z[1/p],T\aidIwx)$ has finite projective dimension as $R\aidIwx$-module (for instance if $R(\aid_{x})_{\Iw}$ is a regular local ring), then $\z\aidIwx$ belongs to $\Hun_{\et}(\Z[1/p],T\aidIwx)$. It is not obvious to this author that this remains true otherwise.

As $\z(f_{\lambda})_{\Iw}$ is non-zero, the $R(\aid_{x})_{\Iw}$-module $\z\aidIwx$ generates inside $\Hun_{\et}(\Z[1/p],T(\aid_{x})_{\Iw})[1/p]$ is free of rank 1 by lemma \ref{LemDepth}. The morphism
\begin{equation}\nonumber
M(\aid_{x})_{\Iw}\tenseur_{R(\aid_{x})_{\Iw}}R(\aid_{x})_{\Iw}[1/p]\fleche\Hun_{\et}(\Z[1/p],T(\aid_{x})_{\Iw})\tenseur_{R(\aid_{x})_{\Iw}}R(\aid_{x})_{\Iw}[1/p]
\end{equation} which sends $\delta(\aid_{x})_{\Iw}$ to $\z(\aid_{x})_{\Iw}$ defines by lifting a morphism of complexes 
\begin{equation}\label{EqZetaRaidx}
Z(\aid_{x})_{\Iw}:M(\aid_{x})_{\Iw}\tenseur_{}R(\aid_{x})_{\Iw}[1/p]\fleche\RGamma_{\et}(\Z[1/p],T(\aid_{x})_{\Iw})[1]\tenseur_{}R(\aid_{x})_{\Iw}[1/p].
\end{equation}
In the above, we view $M(\aid_{x})_{\Iw}$ and $\Hun_{\et}(\Z[1/p],T(\aid_{x})_{\Iw})$ as complexes concentrated in degree 0 and all tensor products are over $R\aidIwx$.
\begin{DefEnglish}\label{DefDeltaAid}
Define $\Delta_{R\aidIwx}(T\aidIwx)$ to be the free $R\aidIwx$-module of rank 1
\begin{equation}\label{EqDefDeltaAid}
\Delta_{R\aidIwx}(T\aidIwx)=\Xcali(T\aidIwx)^{-1}\tenseur_{R\aidIwx}\Det^{-1}_{R\aidIwx}M\aidIwx.
\end{equation}
\end{DefEnglish}
Let  $\psi:R(\aid_{x})_{\Iw}\fleche S_{\Iw}$ be a specialization induced by a $\zp$-algebras morphism $R(\aid_{x})\fleche S$ with values in a flat reduced $\zp$-algebra and such that
\begin{equation}\nonumber
\z_{\psi}\eqdef\psi(\z(\aidIwx))\in\Hun_{\et}(\Z[1/p],T_{\psi})\tenseur_{S_{\Iw}}S_{\Iw}[1/p]
\end{equation}
is non-zero. The morphism sending $\delta_{\psi,\Iw}$ to $\z_{\psi}$ consequently gives as above a morphism 
\begin{equation}\nonumber
Z_{\psi}:M_{\psi}\tenseur_{S_{\Iw}}S_{\Iw}[1/p]\fleche\RGamma_{\et}(\Z[1/p],T_{\psi})[1]\tenseur_{S_{\Iw}}S_{\Iw}[1/p].
\end{equation}
Here again, it is not necessary to invert $p$ when $S$ happens to be a discrete valuation ring.
\begin{DefEnglish}\label{DefDeltaPsi}
Define $\Delta_{S_{\Iw}}(T_{\psi})$ to be the free $S_{\Iw}$-module of rank 1
\begin{equation}\label{EqDefDeltaPsi}
\Delta_{S_{\Iw}}(T_{\psi})=\Xcali(T_{\psi})^{-1}\tenseur_{S_{\Iw}}\Det^{-1}_{S_{\Iw}}M_{\psi}.
\end{equation}
\end{DefEnglish}
After inverting $p$, the modules $\Delta_{S_{\Iw}}(T_{\psi})$ (including when $\psi$ is the identity) are canonically isomorphic to the cone of a morphism of complexes. In general form, this is expressed by the following proposition.
\begin{Prop}\label{PropTrivAidx}
The $R\aidIwx$-module $\Delta_{R\aidIwx}(T\aidIwx)$ is equipped with a canonical morphism 
\begin{equation}\nonumber
\triv:\Delta_{R\aidIwx}T(\aidIwx)\subset\Delta_{R\aidIwx}(T\aidIwx)\tenseur_{R\aidIwx}\Frac(R\aidIwx)\isocan\Frac(R\aidIwx).
\end{equation}
More generally, if $\psi:R(\aid_{x})_{\Iw}\fleche S_{\Iw}$ is a specialization induced by a $\zp$-algebras morphism $R(\aid_{x})\fleche S$ with values in a flat reduced $\zp$-algebra and such that $\z_{\psi}$ is non-zero, then the $S_{\Iw}$-module $\Delta_{S_{\Iw}}(T_{\psi})$ (which we also denote by $\Delta_{\psi}$ for brevity) is endowed with a canonical trivialization isomorphism 
\begin{equation}\nonumber
\triv_{\psi}:\Delta_{\psi}\subset\Delta_{\psi}\tenseur_{S_{\Iw}} Q(S_{\Iw})\isocan Q(S_{\Iw}).
\end{equation}
For all $\psi$ as above, denote by $\triv_{\psi}(\Delta_{\psi})\subset Q(S_{\Iw})$ the image of $\Delta_{\psi}\subset\Delta_{\psi}\tenseur_{S_{\Iw}} Q(S_{\Iw})$ through $\triv_{\psi}$. If the diagram
\begin{equation}\nonumber
\xymatrix{
R\aidIwx\ar[r]^{\xi}\ar[d]_{\psi}&S_{\Iw}'\\
S_{\Iw}\ar[ru]_{\phi}
}
\end{equation}
of reduced flat $\zp$-algebra quotients of $R\aidIwx$ is commutative, then there is a canonical isomorphism 
\begin{equation}\nonumber
\phi^{\Delta}:\Delta_{\psi}\tenseur_{S,\phi}S'\isocan\Delta_{\xi}
\end{equation}
compatible with $\triv_{\psi}$ and $\triv_{\xi}$ in the sense that the rightmost downward arrow of the diagram
\begin{equation}\label{DiagTrivAid}
\xymatrix{
\Delta_{\psi}\ar[r]\ar[d]_{\phi^{\Delta}(-\tenseur_{S,\phi}S')}&\triv_{\psi}(\Delta_{\psi})\ar[d]^{\phi}\\
\Delta_{\xi}\ar[r]&\triv_{\xi}(\Delta_{\xi})
}
\end{equation}
exists and makes the diagram commutative.
\end{Prop}
\begin{proof}
By lemma \ref{LemDepth}, the complex $\Cone Z(\aid_{x})_{\Iw}$ of $\Frac(R\aidIwx)$-vector spaces is acyclic and there is thus a canonical isomorphism 
\begin{equation}\label{EqTrivAcyclic}
\Det_{\Frac(R\aidIwx)}\Cone Z\aidIwx\isocan\Frac(R\aidIwx).
\end{equation}
As 
\begin{equation}\nonumber
\Delta_{R\aidIwx}(T\aidIwx)=\Xcali(T\aidIwx)^{-1}\tenseur_{R\aidIwx}\Det^{-1}_{R\aidIwx}M\aidIwx
\end{equation}
satisfies
\begin{equation}\label{EqTrivFrac}
\Delta_{R\aidIwx}(T\aidIwx)\tenseur_{R\aidIwx}\Frac(R\aidIwx)\isocan\Det_{\Frac(R\aidIwx)}\Cone Z\aidIwx
\end{equation}
by \eqref{EqIsoXcaliSel}, the inclusion $\Delta_{R\aidIwx}(T\aidIwx)$ inside $\Det_{\Frac(R\aidIwx)}\Cone Z\aidIwx$ composed with the isomorphism \eqref{EqTrivAcyclic} defines the morphism $\triv$.

Now let $\psi$ be a specialization as in the statement of the proposition. Then $H^{2}_{\et}(\Z[1/p],T_{\psi})$ is torsion as $\zp$-module so the complex $\Cone Z_{\psi}$ of $Q(S_{\Iw})$-modules is acyclic. Hence \eqref{EqIsoXcaliSel} and acyclicity define two canonical isomorphisms
\begin{equation}\nonumber
\Delta_{S_{\Iw}}(T_{\psi})\tenseur_{S_{\Iw}}Q(S_{\Iw})\isocan\Det_{Q(S_{\Iw})}\Cone Z_{\psi}\isocan Q(S_{\Iw}).
\end{equation}
Composed with the inclusion $\Delta_{S_{\Iw}}(T_{\psi})\subset\Delta_{S_{\Iw}}(T_{\psi})\tenseur_{S_{\Iw}}Q(S_{\Iw})$, this defines $\triv_{\psi}$. 

In order to prove the remaining statements, it is enough to prove that there exists a canonical isomorphism 
\begin{equation}\nonumber
\psi^{\Delta}:\Delta_{R\aidIwx}(T\aidIwx)\tenseur_{R(\aid_{x}),\psi}S\isocan\Delta_{\psi}
\end{equation}
compatible with $\triv$ and $\triv_{\psi}$ for all $\psi$ as in the statement of the proposition. By construction, there exists a modular specialization $\lambda$ factoring through $\psi$ so $\rank_{S_{\Iw}}T_{\psi}^{I_{\ell}}$ is equal to $\rank_{R(\aid_{x})_{\Iw}}T(\aid_{x})_{\Iw}^{I_{\ell}}$ for all $\ell\in\Sigma^{(p)}$. Proposition \ref{DefPropEuler} then yields a canonical isomorphism 
\begin{equation}\nonumber
\Xcali(T\aidIwx)\tenseur_{R\aidIwx}S_{\Iw}\isocan\Xcali(T_{\psi}).
\end{equation}
By definition, $\z_{\psi}$ is the specialization of $\z\aidIwx$ and $\delta_{\psi}$ coincides with $\delta_{\psi}^{\aid_{x}}$ by proposition \ref{PropIhara}. Hence, there is a canonical isomorphism 
\begin{equation}\nonumber
\psi^{\Delta}:\Delta_{R\aidIwx}(T\aidIwx)\tenseur_{R(\aid_{x}),\psi}S\isocan\Delta_{S_{\Iw}}(T_{\psi}).
\end{equation}
\end{proof}
An important consequence of the results of subsection \ref{SubAlgebraicDeterminants} and of proposition \ref{PropIhara} is the compatibility of $\Delta_{\Heckex_{\Sigma,\Iw}}(\Tsigmax)$ with $\Delta_{R\aidIwx}(T\aidIwx)$.
\begin{Prop}\label{PropIharaDelta}
The map $\psi(\aid_{x})$ induces a canonical isomorphism
\begin{equation}\nonumber
\Delta_{\Heckex_{\Sigma,\Iw}}\tenseur_{\Heckex_{\Sigma,\Iw}}R\aidIwx\isocan\Delta_{R\aidIwx}(T\aidIwx)
\end{equation}
compatible with $\triv_{\Sigma}$ and $\triv$.
\end{Prop}
\begin{proof}
The isomorphism \eqref{EqMorIharaSigmaAid} of proposition \ref{PropIhara} induces a canonical isomorphism 
\begin{equation}\nonumber
\Det_{R\aidIwx}(M^{x}_{\Sigma,\Iw}\tenseur_{\Heckex_{\Sigma,\Iw}}R\aidIwx)\tenseur_{R\aidIwx}\Det^{-1}_{R\aidIwx}M\aidIwx\isocan\Det_{R\aidIwx}[R\aidIwx\overset{d}{\fleche}R\aidIwx]
\end{equation}
where the complex $[R\aidIwx\overset{d}{\fleche} R\aidIwx]$ is placed in degree 0 and 1 and $d$ is multiplication by
\begin{equation}\nonumber
\produit{\ell\in\Sigma}{}\Eul_{\ell}(T\aidIwx).
\end{equation}
Hence
\begin{equation}\nonumber
\Det_{R\aidIwx}(M^{x}_{\Sigma,\Iw}\tenseur_{\Heckex_{\Sigma,\Iw}}R\aidIwx)\tenseur_{R\aidIwx}\Det^{-1}_{R\aidIwx}M\aidIwx\isocan\produittenseur{\ell\in\Sigma}{}\Det^{-1}_{R\aidIwx}R\aidIwx/\Eul_{\ell}(T\aidIwx)
\end{equation}
where each of the $R\aidIwx/\Eul_{\ell}(T\aidIwx)$ is viewed as a complex in degree 0. On the other hand, the definition and base-change property of $\RGamma_{c}(\Z[1/\Sigma],-)$ induces a canonical isomorphism between
\begin{equation}\nonumber
\Det_{R\aidIwx}\left(\RGamma_{c}(\Z[1/\Sigma],T_{\Sigma,\Iw}^{x})\Ltenseur_{\Heckex_{\Sigma',\Iw}}{R\aidIwx}\right)\tenseur_{R\aidIwx}\Det^{-1}_{R\aidIwx}\RGamma_{\et}(\Z[1/p],T\aidIwx)
\end{equation}
and
\begin{equation}\nonumber
\produittenseur{\ell\in\Sigma}{}\Xcali_{\ell}(T\aidIwx)^{-1}\isocan\produittenseur{\ell\in\Sigma}{}\Det_{R\aidIwx}R\aidIwx/\Eul_{\ell}(T\aidIwx)
\end{equation}
where each of the $R\aidIwx/\Eul_{\ell}(T\aidIwx)$ is again seen as a complex in degree 0. Putting these isomorphism together yields a canonical isomorphism 
\begin{equation}\label{EqIsoSigmaAidPreuve}
\Delta_{\Heckex_{\Sigma,\Iw}}\tenseur_{\Heckex_{\Sigma,\Iw}}R\aidIwx\isocan\Delta_{R\aidIwx}(T\aidIwx).
\end{equation}
The trivialization of  $\Delta_{R\aidIwx}(T\aidIwx)$ induced by $\triv_{\Sigma}$ and this isomorphism comes from the morphism sending $\psi(\aid_{x})(\delta^{x}_{\Sigma,\Iw})$ to $\psi(\aid_{x})(\z_{\Sigma,\Iw}^{x})$ and hence from the morphism verifying
\begin{equation}\nonumber
\delta\aidIwx\produit{\ell\in\Sigma^{(p)}}{}\Eul_{\ell}(T\aidIwx)\mapsto\z\aidIwx\produit{\ell\in\Sigma^{(p)}}{}\Eul_{\ell}(T\aidIwx).
\end{equation}
This is also the trivialization induced by $Z(\aid_{x})_{\Iw}$ so the compatibility of \eqref{EqIsoSigmaAidPreuve} with $\triv$ and $\triv_{\Sigma}$ is proved.
\end{proof}
We conclude this subsection with a well-known computation relating $\triv_{\psi}(\Delta_{\psi})$ with the invariants appearing in other formulation of the Iwasawa Main Conjecture when $S_{\Iw}$ is a normal Cohen-Macaulay local ring  (keeping the notations of proposition \ref{PropTrivAidx}).
\begin{LemEnglish}\label{LemRegulier}
Let $\psi:R(\aid_{x})\fleche S$ be a specialization with values in a normal Cohen-Macaulay ring and such that $\z_{\psi}$ is non-zero. Then
\begin{equation}\label{EqLemRegulier}
\triv_{\psi}(\Delta_{\psi})=\frac{\carac_{S_{\Iw}}H^{2}_{\et}(\Z[1/p],T_{\psi})}{\carac_{S_{\Iw}}\Hun_{\et}(\Z[1/p],T_{\psi})/\z_{\psi}}S_{\Iw}.
\end{equation}
\end{LemEnglish}
\begin{proof}
As $S$ is a Cohen-Macaulay ring, determinants are uniquely characterized by their localizations at height 1 prime. Because it is furthermore normal, we may and do replace $S_{\Iw}$ by one of its localization at $A$ a discrete valuation ring. Then $\z_{\psi}$ belongs to $\Hun_{\et}(\Z[1/p],T_{\psi})$ and the complex
\begin{equation}\nonumber
\Cone Z_{\psi}=\Cone\left(M_{\psi}\fleche\RGamma_{\et}(\Z[1/p],T_{\psi})[1]\right)
\end{equation}
is a perfect complex of $A$-modules concentrated in degree 0 and 1 with torsion cohomology groups $\Hun_{\et}(\Z[1/p],T_{\psi})/\z_{\psi}$ and $H^{2}_{\et}(\Z[1/p],T_{\psi})$. If $M$ is a torsion $A$-module, the image of  $\Det^{-1}_{A}M$ inside $\Frac(A)$ through the canonical isomorphism $\Det_{\Frac(A)}0\isocan\Frac(A)$ is equal to $(\carac_{A}M)A$ so the statement of the lemma follows.
\end{proof}
The typical outcome of the method of Euler systems, on which we rely, is that the numerator of the left-hand side of  \eqref{EqLemRegulier} divides its denominator or equivalently that $\triv_{\psi}(\Delta^{-1}_{\psi})$ is included in $S_{\Iw}$.
\subsubsection{Coefficients in $\Hecke_{\Sigma,\Iw}$ and $R(\aid)_{\Iw}$}
\begin{LemEnglish}
The complex $\RGamma_{c}(\Z[1/\Sigma],T_{\Sigma,\Iw})$ is a perfect complex of $\Hecke_{\Sigma,\Iw}$-modules with trivial cohomology outside degree $1,2$. After tensor product with $Q(\Hs)$, $H^{1}_{c}(\Z[1/\Sigma],\Ts)$ becomes free of rank 1 and $H^{2}_{c}(\Z[1/\Sigma],\Ts)$ vanishes.

The same assertions hold for $\RGamma_{c}(\Z[1/\Sigma],T(\aid)_{\Iw})$ after replacing $\Hs$ by $\Raid$ and $\Ts$ by $\Taid$.\end{LemEnglish}
\begin{proof}
By irreducibility of $\rhobar$, $H^{0}_{c}(\Z[1/\Sigma],T_{\Sigma,\Iw})$ vanishes. The functor $\RGamma_{c}(\Z[1/\Sigma],-)$ preserves perfect complexes and commutes with base change of ring of coefficients. Applying the latter property to $\psi_{x}:\Hs\fleche\Heckex_{\Sigma}$ shows first that $H^{3}_{c}(\Z[1/\Sigma],T_{\Sigma,\Iw})$ vanishes, then that $H^{2}_{c}(\Z[1/\Sigma],T_{\Sigma,\Iw})$ is $\Hecke_{\Sigma,\Iw}$-torsion and finally that $\Hun_{c}(\Z[1/\Sigma],\Ts)$ is torsion-free. Pick a minimal prime $\pid$ of $\Hs$. After tensor product first with $\Hs/\pid$ then with $\Frac(\Hs/\pid)$, the computation of the Euler-Poincaré characteristic shows that $\Hun_{c}(\Z[1/\Sigma],\Ts)$ is of rank 1. As $\Hs$ is reduced, this establishes that $\Hun_{c}(\Z[1/\Sigma],\Ts)\tenseur_{\Hs}Q(\Hs)$ is free of rank 1. The proofs for $\RGamma_{c}(\Z[1/\Sigma],T(\aid)_{\Iw})$ are similar but easier.
\end{proof}
We write $Z(\aid)_{\Iw}$ for the $R(\aid)_{\Iw}$-module
\begin{equation}\label{EqSpecZeta}
Z(\aid)_{\Iw}=\left(Z_{\Sigma,\Iw}\tenseur_{\Hecke_{\Sigma,\Iw}}R(\aid)_{\Iw}\right)\tenseur_{R(\aid)_{\Iw}}\produittenseur{\ell\in\Sigma^{(p)}}{}\Det_{\Raid}\left(M(\aid)_{\Iw}/\pi_{\Sigma,\Sigma(\aid)}(M_{\Sigma,\Iw})\right).
\end{equation}
By construction, $Z(\aid)_{\Iw}$ is free module of rank 1 which we may view as a submodule of 
\begin{equation}\nonumber
\Hun_{\et}(\Z[1/\Sigma],\Htildeun_{\et}(U_{1}(N(\Sigma))^{(p)},\zp)_{\mgot_{\rhobar}})\tenseur_{\Hs}\Raid
\end{equation} 
after the trivialization of $\Det_{\Raid}M(\aid)_{\Iw}/\pi_{\Sigma,\Sigma(\aid)}(M_{\Sigma,\Iw})$ induced by
\begin{equation}\nonumber
\Det_{\Raid}\left(\frac{M(\aid)_{\Iw}}{\pi_{\Sigma,\Sigma(\aid)}(M_{\Sigma,\Iw})}\right)\subset\Det_{\Frac(\Raid)}\left(\frac{M(\aid)_{\Iw}}{\pi_{\Sigma,\Sigma(\aid)}(M_{\Sigma,\Iw})}\right)\tenseur\Frac(\Raid)\isocan\Frac(\Raid)
\end{equation}
\begin{DefEnglish}\label{DefZUnivPsi}
Let $\psi$ be either a specialization $\psi:\Hs\fleche S$ with values in a reduced, flat $\zp$-algebra or in a torsion $\zp$-algebra or a specialization $\psi:\Raid\fleche S$ with values in a domain. Assume in both cases that $\z_{\psi}\neq0$.

In the first case, define $Z_{\psi}$ to be $Z_{\Sigma,\Iw}\tenseur_{\Hs}S$. In the second case, there exists a  specialization $\Hs\fleche S$ factoring through $\psi$. Define $Z_{\psi}$ by 
\begin{equation}\nonumber
Z_{\psi}=\left(Z_{\Sigma,\Iw}\tenseur_{\Hs}S\right)\tenseur_{S}\produittenseur{\ell\in\Sigma}{}\Det_{S}\left(M_{\psi}/M^{\Sigma}_{\psi}\right).
\end{equation} 
\end{DefEnglish}
We say that a specialization $\psi$ is shimmering if it is as in definition \ref{DefZUnivPsi}. Note that shimmering specializations of $\Hs$ may have value in artinian rings.
\begin{DefEnglish}\label{DefDeltaUniv}
Define $\Delta_{\Hecke_{\Sigma,\Iw}}(T_{\Sigma,\Iw})$ to be the free $\Hecke_{\Sigma,\Iw}$-module of rank 1
\begin{equation}\nonumber
\Delta_{\Hecke_{\Sigma,\Iw}}(T_{\Sigma,\Iw})=\Det^{-1}_{\Hecke_{\Sigma,\Iw}}\RGamma_{c}(\Z[1/\Sigma],T_{\Sigma,\Iw})\tenseur_{\Hecke_{\Sigma,\Iw}}\Det^{-1}_{\Hecke_{\Sigma,\Iw}}Z_{\Sigma,\Iw}.
\end{equation}
Define $\Delta_{R(\aid)_{\Iw}}(T(\aid)_{\Iw})$ to be the free $R(\aid)_{\Iw}$-module of rank 1
\begin{equation}\nonumber
\Delta_{R(\aid)_{\Iw}}(T(\aid)_{\Iw})=\Xcali(\Taid)^{-1}\tenseur_{R(\aid)_{\Iw}}\Det^{-1}_{R(\aid)_{\Iw}}Z(\aid)_{\Iw}.
\end{equation}
If $\psi$ is a shimmering specialization of $\Hs$, define $\Delta_{S}(T_{\psi})$ to be the $S$-module of rank 1
\begin{equation}\nonumber
\Delta_{S}(T_{\psi})=\Det^{-1}_{S}\RGamma_{c}(\Z[1/\Sigma],T_{\psi})\tenseur_{S}\Det^{-1}_{S}Z_{\psi}.
\end{equation}
If $\psi$ is a shimmering specialization of $\Raid$, define $\Delta_{S}(T_{\psi})$ to be the $S$-module of rank 1
\begin{equation}\nonumber
\Delta_{S}(T_{\psi})=\Xcali(T_{\psi})^{-1}\tenseur\Det^{-1}_{S}Z_{\psi}.
\end{equation}
\end{DefEnglish}
For shimmering specializations $\psi$ factoring through $\Hs^{x}$ or $R\aidIwx$, there are currently two potentially conflicting definitions of $Z_{\psi}$ and $\Delta_{S_{\Iw}}(T_{\psi})$. This conflict is resolved by the following theorem which shows that the fundamental lines $\Delta_{\Hecke_{\Sigma,\Iw}}(T_{\Sigma,\Iw})$ and $\Delta_{R(\aid)_{\Iw}}(T(\aid)_{\Iw})$ satisfy the base-change property forming one of the central part of the ETNC.
\begin{TheoEnglish}\label{TheoUnivSpec}
There are canonical isomorphisms 
\begin{equation}\nonumber
\Delta_{\Hecke_{\Sigma,\Iw}}(T_{\Sigma,\Iw})\tenseur_{\Hecke_{\Sigma,\Iw},\psi}\Heckex_{\Sigma,\Iw}\isocan\Delta_{\Heckex_{\Sigma,\Iw}}(\Tsigmax)
\end{equation}
and 
\begin{equation}\nonumber
\Delta_{R(\aid)_{\Iw}}(T(\aid)_{\Iw})\tenseur_{R(\aid)_{\Iw},\psi\aidIwx}R\aidIwx\isocan\Delta_{R\aidIwx}(T\aidIwx).
\end{equation}
If $\psi$ is a shimmering specialization, then there is a canonical isomorphism
\begin{equation}\nonumber
\Delta_{\Hs}(\Ts)\tenseur_{\Hs}S\isocan\Delta_{S}(T_{\psi})\textrm{ or }\Delta_{\Raid}(\Taid)\tenseur_{\Raid}S\isocan\Delta_{S}(T_{\psi}).
\end{equation}
depending on whether $\psi$ is a specialization of $\Hs$ or $\Raid$.
\end{TheoEnglish}
\begin{proof}
The first assertion follows immediately from the compatibility with base-change of cohomology with compact support, the base-change property of algebraic determinants over local domains and the existence of the zeta morphisms \eqref{EqZetaSigma} and \eqref{EqZetaRaidx}. 

We prove the second assertion for a specialization $\psi$ of $\Raid$. Assume first that the equality 
\begin{equation}\nonumber
\rank_{S}T_{\psi}^{I_{\ell}}=\rank_{\Raid}\Taid^{I_{\ell}}
\end{equation}
holds for all $\ell\in\Sigma^{(p)}$. As in the proof of proposition \ref{PropTrivAidx}, proposition \ref{PropCompEuler} then implies that $\Xcali(\Taid)\tenseur_{\Raid,\psi}S$ is equal to $\Xcali(T_{\psi})$ and proposition \ref{PropIhara} implies that $Z(\aid)_{\Iw}\tenseur_{\Raid,\psi}S$ is equal to $Z_{\psi}$. The claim is thus established in this case. Now assume that there exists an $\ell\in\Sigma^{(p)}$ such that
\begin{equation}\nonumber
\rank_{S}T_{\psi}^{I_{\ell}}>\rank_{\Raid}\Taid^{I_{\ell}}.
\end{equation}
The discussion in the proof of proposition \ref{PropCompEuler} then implies that $T\aidIwx^{I_{\ell}}$ is of rank 1 while $T_{\psi}$ is unramified at $\ell$. Since cohomology with compact support commutes with base-change and since $\psi(\delta_{\Sigma,\Iw})=\delta_{\psi,\Iw}^{\Sigma}$ by definition, the contribution of such $\ell$ to the definitions of $\Delta_{\Raid}(T\aidIwx)$ and $\Delta_{S}(T_{\psi})$ is purely local at $\ell$. Hence, we may further assume that there is a single such $\ell$. Denote by $\alpha$ the unique eigenvalue of $\Fr(\ell)$ acting on $T_{\psi}$ such that $\alpha$ is the image through $\psi$ of $\Fr(\ell)$ acting on $\Taid$ and denote by $\beta$ its other eigenvalue. Then
\begin{equation}\nonumber
\left(\Delta_{\Raid}(\Taid)\tenseur_{R(\aid),\psi}S\right)\tenseur_{S}\Delta_{S}(T_{\psi})^{-1}
\end{equation}
decomposes into
\begin{equation}\nonumber
\left(\Xcali_{\ell}(\Taid)\tenseur_{R(\aid),\psi}S\right)^{-1}\tenseur_{S}\Xcali_{\ell}(T_{\psi})^{}\isocan\Det^{-1}_{S}S/\beta
\end{equation}
and 
\begin{equation}\nonumber
\Det^{-1}_{S}(M(\aid)_{\Iw}\tenseur_{\Raid,\psi}S)\tenseur\Det^{}_{S}M_{\psi}.
\end{equation}
According to proposition \ref{PropIhara}, $M(\aid)_{\Iw}\tenseur_{\Raid,\psi}S$ is a lattice inside $M_{\psi}$ with index equal to $\beta$. Hence
\begin{equation}\nonumber
\Det_{S}^{-1}(M(\aid)_{\Iw}\tenseur_{\Raid,\psi}S)\tenseur\Det^{}_{S}M_{\psi}\isocan\Det^{}_{S}S/\beta\end{equation}
and finally
\begin{equation}\nonumber
\left(\Delta_{\Raid}(\Taid)\tenseur_{R(\aid_{x}),\psi}S\right)\tenseur_{S}\Delta_{S}(T_{\psi})^{-1}\isocan S.
\end{equation}

Finally, if $\psi$ is a shimmering specialization of $\Hs$, then $\RGamma_{c}(\Z[1/\Sigma],\Ts)\Ltenseur_{\Hs}S$ is isomorphic to $\RGamma_{c}(\Z[1/\Sigma],T_{\psi})$ and $Z_{\Sigma,\Iw}\tenseur_{\Hs}S$ is equal to $Z_{\psi}$ by definition.
\end{proof}
\paragraph{Remark:}Though it follows from the definition of $\Delta_{\Hs}(\Ts)$, the compatibility of fundamental lines with shimmering specializations of $\Hs$ which are not classical is far from formal. Consider for instance the commutative diagram 
\begin{equation}\nonumber
\xymatrix{
\Hs\ar[d]_{\psi_{x}}\ar[rd]^{\pi}\\
\Hs^{x}\ar[r]^(0.4){\pi}&\Hs/\mgot^{n}
}
\end{equation}
in which $x$ is a classical prime and $\pi$ is the canonical projection (which is shimmering for large enough $n$). Then the compatibility of fundamental lines with $\psi_{x}$ and $\pi$ asserts that the reduction modulo $\mgot^{n}$ of Kato's Euler system is equal to the image of the zeta element in completed cohomology of definition \ref{DefZetaUniv} and so ultimately relies on the $G_{\Q,\Sigma}$-equivariant isomorphism 
\begin{equation}\nonumber
\Htilde^{1}_{\et}(U_{1}(N(\Sigma))^{(p)},\Ocal)_{\mgot_{\rhobar}}/\mgot^{n}_{\Ocal}\Htilde^{1}_{\et}(U_{1}(N(\Sigma))^{(p)},\Ocal)_{\mgot_{\rhobar}}\simeq\Htilde^{1}_{\et}(U_{1}(N(\Sigma))^{(p)},\Ocal/\mgot_{\Ocal}^{n})_{\mgot_{\rhobar}}.
\end{equation}
In particular, a variant of theorem \ref{TheoUnivSpec} for the fundamental lines of \cite[Conjecture 3.2]{KatoViaBdR}, in which the role completed cohomology plays in this manuscript is played by by $\Ts$ or $\Taid$ themselves, would in most likelihood be false for the natural specialization maps on $\Ts$ or $\Taid$ (for a counterexample, fix $\ell\nmid p$ a prime in $\Sigma$ and consider the non-classical specialization corresponding to the intersection of two irreducible components of $\Spec\Hs[1/p]$ for which the generic rank of $I_{\ell}$-invariants are distinct).
\subsection{Statement of the conjectures}\label{SubConj}
We are in position to state the Equivariant Tamagawa Number Conjectures with coefficients in Hecke algebras. In addition to the full form of the conjecture, we consider a weaker form whose main feature of interest is that it is amenable to proof. 
\subsubsection{Conjectures with coefficients in $\Heckex_{\Sigma,\Iw}$}\label{SubConjSigma}
\begin{Conj}\label{ConjETNCsigmax}
The trivialization morphism is an isomorphism 
\begin{equation}\nonumber
\triv_{\Sigma}:\Delta_{\Heckex_{\Sigma,\Iw}}(\Tsigmax)\simeq\Heckex_{\Sigma,\Iw}.
\end{equation}
\end{Conj}
A weaker form of conjecture \ref{ConjETNCsigmax}, which we refer to as the weak ETNC with coefficients in $\Heckex_{\Sigma,\Iw}$ is as follows.
\begin{Conj}\label{WeakConjETNCsigmax}
The trivialization morphism induces an inclusion 
\begin{equation}\nonumber
\triv_{\Sigma}:\Delta_{\Heckex_{\Sigma,\Iw}}(\Tsigmax)^{-1}\plonge\Heckex_{\Sigma,\Iw}.
\end{equation}
\end{Conj}
It follows at once from proposition \ref{PropCompDeltaSigma} that conjectures \ref{ConjETNCsigmax} and \ref{WeakConjETNCsigmax} are compatible with change of allowable set $\Sigma$ and specializations.
\begin{Prop}\label{PropCompConjSigmax}
Let $\Sigma\subset\Sigma'$ be an inclusion of allowable sets of primes. Conjecture \ref{ConjETNCsigmax} (resp. \ref{WeakConjETNCsigmax}) is true for $\Sigma$ if it is true for $\Sigma'$. If conjecture \ref{WeakConjETNCsigmax} is true for $\Sigma'$ and conjecture \ref{ConjETNCsigmax} is true for $\Sigma$ then, conjecture \ref{ConjETNCsigmax} is true for $\Sigma'$.

If conjecture \ref{ConjETNCsigmax} is true, then $\triv_{\psi}$ induces an isomorphism 
\begin{equation}\nonumber
\triv_{\psi}:\Delta_{\psi}\simeq S_{\Iw}
\end{equation}
for all $\psi:\Heckex_{\Sigma,\Iw}\fleche S_{\Iw}$ as in proposition \ref{PropCompDeltaSigma}. If conjecture \ref{WeakConjETNCsigmax} is true and there exists $\psi$ as above such that $\triv_{\psi}$ is an isomorphism 
\begin{equation}\nonumber
\triv_{\psi}:\Delta_{\psi}\simeq S_{\Iw}
\end{equation}
then conjecture \ref{ConjETNCsigmax} is true for $\Sigma$.
\end{Prop}
\begin{proof}
If $\triv_{\Sigma'}$ is either an isomorphism $\Delta_{\Heckex_{\Sigma',\Iw}}(T_{\Sigma',\Iw}^{x})\simeq\Heckex_{\Sigma',\Iw}$ or an embedding
\begin{equation}\nonumber
\Delta_{\Heckex_{\Sigma',\Iw}}(T_{\Sigma',\Iw}^{x})^{-1}\plonge\Heckex_{\Sigma',\Iw},
\end{equation}
then the commutativity of \eqref{DiagTrivSigma} and the fact that a local morphism sends units to units and non-units to non-units implies that $\triv_{\Sigma}$ is respectively an isomorphism or an embedding. 

Likewise, if $\triv_{\Sigma'}$ is an embedding $\Delta_{\Heckex_{\Sigma',\Iw}}(T_{\Sigma',\Iw}^{x})^{-1}\plonge\Heckex_{\Sigma',\Iw}
$ and $\triv_{\Sigma}$ is an isomorphism $\Delta_{\Heckex_{\Sigma,\Iw}}(\Tsigmax)\simeq\Heckex_{\Sigma,\Iw}$, then \eqref{DiagTrivSigma} implies that $\triv_{\Sigma'}$ is an isomorphism.

The other assertions are proved exactly in the same way by appealing to the (easier) commutativity of the diagram \eqref{DiagTrivPsiXi} in place of the commutativity of the diagram \eqref{DiagTrivSigma}.
\end{proof}
\subsubsection{Conjectures with coefficients in $R\aidIwx$}\label{SubConjRaidx}
\begin{Conj}\label{ConjETNCaidx}
The trivialization morphism is an isomorphism
\begin{equation}\nonumber
\triv:\Delta_{R\aidIwx}(T\aidIwx)\simeq R\aidIwx.
\end{equation}
\end{Conj}
A weaker form of conjecture \ref{ConjETNCaidx}, which we refer to the weak ETNC with coefficients in $R\aidIwx$ is as follows.
\begin{Conj}\label{WeakConjETNCaidx}
The trivialization morphism induces an inclusion
\begin{equation}\nonumber
\triv:\Delta_{R\aidIwx}(T\aidIwx)^{-1}\plonge R\aidIwx.
\end{equation}
\end{Conj}
It follows from proposition \ref{PropTrivAidx} and proposition \ref{PropIharaDelta} that both the ETNC and the weak ETNC with coefficients in $R\aidIwx$ are compatible with specializations and change of $\Sigma$.
\begin{Prop}\label{PropCompETNCaidx}
If conjecture \ref{ConjETNCaidx} (resp. conjecture \ref{WeakConjETNCaidx}) is true, then $\triv_{\psi}$ induces an isomorphism $\Delta_{S_{\Iw}}(T_{\psi})\simeq S_{\Iw}$ (resp. an embedding $\Delta_{S_{\Iw}}^{-1}\plonge S_{\Iw}$) for all specializations $\psi$ as in proposition \ref{PropTrivAidx}. If conjecture \ref{WeakConjETNCaidx} is true and if there exists a $\psi$ such that $\triv_{\psi}$ induces an isomorphism $\Delta_{S_{\Iw}}(T_{\psi})\simeq S_{\Iw}$, then conjecture \ref{ConjETNCaidx} is true.

Conjecture \ref{ConjETNCsigmax} (resp. conjecture \ref{WeakConjETNCsigmax}) implies conjecture \ref{ConjETNCaidx} (resp. conjecture \ref{WeakConjETNCaidx}). If conjecture \ref{WeakConjETNCsigmax} and conjecture \ref{ConjETNCaidx} hold, then conjecture \ref{ConjETNCsigmax} holds.
\end{Prop}
\begin{proof}
If $\triv$ is either an isomorphism $\Delta_{R\aidIwx}(T\aidIwx)\simeq R\aidIwx$ or an embedding
\begin{equation}\nonumber
\Delta_{R\aidIwx}(T\aidIwx)^{-1}\plonge R\aidIwx
\end{equation}
then the commutativity of \eqref{DiagTrivAid} and the fact that a local morphism sends units to units and non-units to non-units implies that $\triv_{\psi}$ is respectively an isomorphism or an embedding. 

Likewise, if $\triv$ is an embedding $\Delta_{R\aidIwx}(T\aidIwx)^{-1}\plonge R\aidIwx$ and $\triv_{\psi}$ is an isomorphism $\Delta_{\psi}\simeq S_{\Iw}$, then \eqref{DiagTrivAid} implies that $\triv$ is an isomorphism.

If $\triv_{\Sigma}$ is either an isomorphism $\Delta_{\Heckex_{\Sigma,\Iw}}(\Tsigmax)\simeq\Heckex_{\Sigma,\Iw}$ or an embedding 
\begin{equation}\nonumber
\Delta_{\Heckex_{\Sigma,\Iw}}(\Tsigmax)^{-1}\plonge\Heckex_{\Sigma,\Iw},
\end{equation}
then proposition \ref{PropIharaDelta} implies in a similar way that $\triv$ is either an isomorphism
\begin{equation}\nonumber
\Delta_{R\aidIwx}(T\aidIwx)\simeq R\aidIwx
\end{equation}
or an embedding $\Delta_{R\aidIwx}(T\aidIwx)^{-1}\plonge R\aidIwx$. Conversely, proposition \ref{PropIharaDelta} implies that $\Delta_{\Heckex_{\Sigma,\Iw}}(\Tsigmax)$ is isomorphic to $\Heckex_{\Sigma,\Iw}$ if $\triv_{\Sigma}$ realizes an embedding of one into the other and if $\triv$ is an isomorphism $\Delta_{R\aidIwx}(T\aidIwx)\simeq R\aidIwx$.
\end{proof}
An interesting corollary of proposition \ref{PropCompETNCaidx} is the compatibility of conjecture \ref{ConjETNCsigmax} with modular specialization, and hence with conjecture \ref{ConjETNC}. Note that it is not a mere restatement of proposition \ref{PropCompConjSigmax} for $\psi=\lambda$, as the image of $\z_{\Sigma,\Iw}^{x}$ through $\lambda$ is not $\z(f_{\lambda})_{\Iw}$. 
\begin{CorEnglish}
Let $\lambda:\Heckex_{\Sigma}\fleche\Qbar_{p}$ be a modular specialization. If conjecture \ref{ConjETNCsigmax} is true, then conjecture \ref{ConjIMC} for $\lambda$ is true. If conjecture \ref{WeakConjETNCsigmax} is true and conjecture \ref{ConjIMC} is true for $\lambda$ then conjecture \ref{ConjETNCsigmax} is true.
\end{CorEnglish}
\begin{proof}
This follows formally by two applications of proposition \ref{PropCompETNCaidx}, first to $\Delta_{R\aidIwx}(T\aidIwx)$ and $\lambda$, then to $\Delta_{\Heckex_{\Sigma,\Iw}}(\Tsigmax)$ and $\Delta_{R\aidIwx}(T\aidIwx)$.
\end{proof}
\subsubsection{Conjectures with coefficients in $\Hecke_{\Sigma,\Iw}$ and $R(\aid)_{\Iw}$}

The complex $\RGamma_{c}(\Z[1/\Sigma],\Ts)$ is of perfect amplitude $[0,3]$ and has trivial cohomology in degree 2 so admits a presentation $[C_{0}\fleche C_{1}\fleche C_{2}]$. Hence, the determinant of the complex $[\Heckes\oplus C_{0}\fleche C_{1}\fleche C_{2}]$ is canonically isomorphic to $\Delta_{\Heckes}(\Ts)$ through the identification between $\Hs$ and $Z_{\Sigma,\Iw}$ sending $1$ to $\z_{\Sigma,\Iw}$.

After tensor product with $Q(\Hs)$, the complex $[\Heckes\oplus C_{0}\fleche C_{1}\fleche C_{2}]$ becomes quasi-isomorphic to $[Q(\Heckes)\fleche Q(\Heckes)]$ in degree 0 and 1 and with zero differential map. Hence, there is a canonical injection
\begin{equation}\label{EqInjectionSigma}
\Det_{\Hs}[\Heckes\fleche\Heckes]\plonge\Delta_{\Hs}(\Ts)\tenseur_{\Hs} Q(\Ts)
\end{equation}
obtained by composing $\Det_{\Hs}[\Heckes\rightarrow\Heckes]\subset\Det_{Q(\Hs)}[Q(\Hs)\rightarrow Q(\Hs)]$ with the canonical isomorphism $\Det_{Q(\Hs)}[Q(\Hs)\rightarrow Q(\Hs)]\simeq\Delta_{\Hs}(\Ts)\tenseur Q(\Ts)$.
\begin{Conj}\label{ConjETNCuniv}
Inside $\Delta_{{\Hecke_{\Sigma,\Iw}}}(T_{\Sigma,\Iw})\tenseur_{\Hecke_{\Sigma,\Iw}}Q(\Hecke_{\Sigma,\Iw})$, there is an equality
\begin{equation}\nonumber
\Delta_{{\Hecke_{\Sigma,\Iw}}}(T_{\Sigma,\Iw})=\Det_{\Hecke_{\Sigma,\Iw}}[\Hecke_{\Sigma,\Iw}\overset{0}{\fleche}\Hecke_{\Sigma,\Iw}]
\end{equation}
in which the complex on the right-hand side is concentrated in degree 0 and 1 and is seen inside $\Delta_{{\Hecke_{\Sigma,\Iw}}}(T_{\Sigma,\Iw})\tenseur_{\Hecke_{\Sigma,\Iw}}Q(\Hecke_{\Sigma,\Iw})$ through \eqref{EqInjectionSigma}.
\end{Conj}
The exact same constructions and arguments as in the beginning of this subsection but over $\Raid$ yield a canonical injection 
\begin{equation}\label{EqInjectionRaid}
\Det_{\Raid}[\Raid\overset{0}{\fleche}\Raid]\plonge\Delta_{\Raid}(\Taid)\tenseur\Frac(\Raid).
\end{equation}
More generally, there is a canonical injection
\begin{equation}\nonumber
\Det_{S}[S\overset{0}{\fleche}S]\plonge\Delta_{S}(T_{\psi})\tenseur Q(S)
\end{equation}
for all shimmering specialization $\psi:\Hs\fleche S$ with values in reduced ring and a canonical injection
\begin{equation}\nonumber
\Det_{S}[S\overset{0}{\fleche}S]\plonge\Delta_{S}(T_{\psi})\tenseur\Frac(S)
\end{equation}
for all shimmering specializations $\psi:\Raid\fleche S$.
\begin{Conj}\label{ConjETNCRaid}
Inside $\Delta_{{\Raid}}(\Taid)\tenseur_{\Raid}\Frac(\Raid)$, there is an equality
\begin{equation}\nonumber
\Delta_{{\Raid}}(\Taid)=\Det_{\Raid}[\Raid\overset{0}{\fleche}\Raid]
\end{equation}
in which the complex on the right-hand side is concentrated in degree 0 and 1 and is seen inside $\Delta_{{\Raid}}(\Taid)\tenseur_{\Raid}\Frac(\Raid)$ through \eqref{EqInjectionRaid}.
\end{Conj}
\begin{Conj}\label{ConjETNCShimmering}
Let $\psi:\Hs\fleche S$ be a shimmering specialization with values in a reduced ring. Inside $\Delta_{S}(T_{\psi})\tenseur_{S}Q(S)$, there is an equality
\begin{equation}\nonumber
\Delta_{S}(T_{\psi})=\Det_{S}[S\overset{0}{\fleche}S].
\end{equation}
\end{Conj}

As in the previous subsections, we consider the weaker conjectures.
\begin{Conj}\label{WeakConjETNCuniv}
Inside $\Delta_{{\Hecke_{\Sigma,\Iw}}}(T_{\Sigma,\Iw})\tenseur_{\Hecke_{\Sigma,\Iw}}Q(\Hecke_{\Sigma,\Iw})$, there is an inclusion
\begin{equation}\nonumber
\Delta_{{\Hecke_{\Sigma,\Iw}}}(T_{\Sigma,\Iw})^{-1}\subset\Det_{\Hecke_{\Sigma,\Iw}}[\Hecke_{\Sigma,\Iw}\overset{0}{\fleche}\Hecke_{\Sigma,\Iw}]
\end{equation}
in which the complex on the right-hand side is concentrated in degree 0 and 1 and is seen inside $\Delta_{{\Hecke_{\Sigma,\Iw}}}(T_{\Sigma,\Iw})\tenseur_{\Hecke_{\Sigma,\Iw}}Q(\Hecke_{\Sigma,\Iw})$ through \eqref{EqInjectionSigma}. Equivalently, there exists a non-zero divisor $x\in\Hs$ such that the equality
\begin{equation}\nonumber
\Delta_{{\Hecke_{\Sigma,\Iw}}}(T_{\Sigma,\Iw})^{-1}=\Det_{\Hecke_{\Sigma,\Iw}}[x\Hecke_{\Sigma,\Iw}\overset{0}{\fleche}\Hecke_{\Sigma,\Iw}]
\end{equation}
holds.
\end{Conj}
\begin{Conj}\label{WeakConjETNCRaid}
Inside $\Delta_{{\Raid}}(\Taid)\tenseur_{\Raid}\Frac(\Raid)$, there  is an inclusion
\begin{equation}\nonumber
\Delta_{{\Raid}}(\Taid)^{-1}\subset\Det_{\Raid}[\Raid\overset{0}{\fleche}\Raid]
\end{equation}
in which the complex on the right-hand side is concentrated in degree 0 and 1 and is seen inside $\Delta_{{\Raid}}(T_{\Sigma,\Iw})\tenseur_{\Raid}\Frac(\Raid)$ through \eqref{EqInjectionRaid}. Equivalently, there exists a non-zero $x\in\Raid$ such that the equality
\begin{equation}\nonumber
\Delta_{{\Raid}}(\Taid)^{-1}=\Det_{\Raid}[x\Raid\overset{0}{\fleche}\Raid]
\end{equation}
holds.
\end{Conj}
\begin{Conj}\label{WeakConjETNCShimmering}
Let $\psi:\Hs\fleche S$ be a shimmering specialization with values in a reduced ring. Inside $\Delta_{S}(T_{\psi})\tenseur_{S}Q(S)$, there is an inclusion
\begin{equation}\nonumber
\Delta_{S}(T_{\psi})^{-1}\subset\Det_{S}[S\overset{0}{\fleche}S].
\end{equation}
Equivalently, there exists a non-zero divisor $x\in S$ such that the equality 
\begin{equation}\nonumber
\Delta_{S}(T_{\psi})^{-1}=\Det_{S}[xS\overset{0}{\fleche}S].
\end{equation}
holds.
\end{Conj}

As their counterparts after specialization at a classical prime, these conjectures are compatible with $\psi(\aid)_{\Iw}:\Hs\fleche\Raid$.
\begin{Prop}\label{PropCompSigmaAid}
Conjecture \ref{ConjETNCuniv} (resp. conjecture \ref{WeakConjETNCuniv}) implies conjecture \ref{ConjETNCRaid} (resp. conjecture \ref{WeakConjETNCRaid}). 
\end{Prop}
\begin{proof}
The proof is similar to that of proposition \ref{PropIharaDelta}, but easier.
\end{proof}
The analogue of proposition \ref{PropCompETNCaidx} for shimmering specializations also holds. Because we lack a genuine zeta morphism with coefficients in $\Hs$, it is not quite formal even at modular specializations.
\begin{Prop}\label{PropCompUniv}
Let $\psi$ be a shimmering specialization with values in a reduced ring. Conjecture \ref{ConjETNCuniv} (resp. conjecture \ref{WeakConjETNCuniv}) implies conjecture \ref{ConjETNCShimmering} (resp. conjecture \ref{WeakConjETNCShimmering}) for $T_{\psi}$.
\end{Prop}
\begin{proof}
We prove the statement for a shimmering specialization $\psi:\Raid\fleche S$ as the proof of a shimmering specialization of $\Hs$ is similar but easier. By theorem \ref{TheoUnivSpec}, there is a canonical isomorphism $\Delta_{\Raid}(\Taid)\tenseur_{\psi}S\isocan\Delta_{\psi}$. It is thus enough to show that the trivialization $\triv_{\psi}$ of $\Delta_{S}(T_{\psi})$ is the same as the trivialization induced by the equality
\begin{equation}\nonumber
\Delta_{\Raid}(\Taid)=\Det_{\Raid}[\Raid\fleche\Raid],
\end{equation}
change of ring of coefficients to $S$ and the canonical isomorphism between $\Det_{S}[S\fleche S]$ and $S$. The first trivialization compares the relative positions of $\Xcali(T_{\psi})$ and the inverse of the determinant of $S$. The conjectures respectively assert that these two modules are equal or that the latter is included in the former. The second trivialization identifies $\Xcali(T_{\psi})$ with the determinant of $S$ placed in degree 1 and the conjectures then asserts that $S\z_{\psi}$ is identified with $S$ or that it inverse is included in it.
\end{proof}
\section{Main results and their proofs}
\subsection{The main theorem}\label{SubStatement}
\subsubsection{Statement}
Recall that $p$ is an odd prime and put $p^{*}=(-1)^{(p-1)/2}p$.
\begin{TheoEnglish}\label{TheoCorps}
Assume $\rhobar$ satisfies the following properties.
\begin{enumerate}
\item\label{ItemIrr} The representation $\rhobar|_{G_{\Q(\sqrt{p^{*}})}}$ is irreducible.
\item\label{ItemEmerton} The semisimplification of the representation $\rhobar|_{G_{\qp}}$ is not scalar.
\item\label{ItemTechnique} The order of the image of $\rhobar$ is divisible by $p$.
\suspend{enumerate}
Then conjecture \ref{WeakConjETNCuniv} is true for $\Hs$. Assume moreover that the following condition holds.
\resume{enumerate}
\item\label{ItemCongruence} There exists a modular specialization $x:\Hs\fleche\Qbar_{p}$ attached to a newform $f$ such that conjecture \ref{ConjIMC} holds for $f$.
\end{enumerate}
Then conjecture \ref{ConjETNCuniv} is true for $\Hs$.
\end{TheoEnglish}
Before proving theorem \ref{TheoCorps}, we record the following corollaries.
\begin{CorEnglish}\label{CorWeak}
Assume $\rhobar$ satisfies assumptions \ref{ItemIrr}, \ref{ItemEmerton} and \ref{ItemTechnique} of \ref{TheoCorps}. Let $\aid$ be a minimal prime of $\Hecke^{\new}_{\Sigma}$ and $x$ be a modular specialization of $\Hs$. Then the conjectures \ref{WeakConjETNCRaid} for $\Raid$, \ref{WeakConjETNCsigmax} for $\Heckex_{\Sigma,\Iw}$, \ref{WeakConjETNCaidx} for $R(\aid_{x})_{\Iw}$ and \ref{WeakConjIMC} for $M(f_{x})$ all hold.
\end{CorEnglish}
\begin{proof}
Combine the result of theorem \ref{TheoCorps} with theorem \ref{TheoUnivSpec} andpropositions \ref{PropCompSigmaAid} and \ref{PropCompUniv}.
\end{proof}
As announced after \ref{WeakConjIMC}, the last statement of corollary \ref{CorWeak} eliminates the error term which appears in \cite[Theorem 12.5]{KatoEuler} in the presence of an exceptional zero of the $p$-adic $L$-function.
\begin{CorEnglish}\label{CorStrong}
Assume $\rhobar$ satisfies assumption \ref{ItemIrr} of \ref{TheoCorps}. Assume moreover that either of the following conditions hold.
\begin{enumerate}
\item\label{ItemOrd} There exists $\ell||N(\rhobar)$, the semisimplification of the representation $\rhobar|_{G_{\qp}}$ is isomorphic to $\chi\oplus \psi$ with $\chi\neq\psi$ and $\det\rhobar$ is unramified outside $p$.
\item\label{ItemWan} There exists $\ell||N(\rhobar)$ and there exists a modular specialization $\psi:\Hs\fleche\Qbar_{p}$ with values in $\Q$ such that $f_{\psi}$ belongs to $S_{2}(\Gamma_{0}(N))$ and verifies $a_{p}(f_{\psi})=0$.
\end{enumerate}
Let $\aid$ be a minimal prime of $\Hecke^{\new}_{\Sigma}$ and $x$ be a modular specialization of $\Hs$ attached to an eigencuspform $f$. Then the conjectures \ref{ConjETNCuniv} for $\Hs$, \ref{ConjETNCRaid} for $\Raid$, \ref{ConjETNCsigmax} for $\Heckex_{\Sigma,\Iw}$, \ref{ConjETNCaidx} for $R(\aid_{x})_{\Iw}$ and \ref{ConjIMC} for $f$ all hold.
\end{CorEnglish}
\begin{proof}
Under the assumptions of the corollary, hypothesis \ref{ItemCongruence} of theorem \ref{TheoCorps} holds by theorem \ref{TheoBibliographique}.
\end{proof}
In plain language, corollary \ref{CorStrong} asserts that under the three first hypotheses of theorem \ref{TheoCorps} and either of its hypothesis, all the conjectures considered in this manuscript are true. In particular, the Iwasawa Main Conjecture is then true for all modular point in $\Hs$.
\subsubsection{Examples}\label{SubExamples}
In this subsection, we give numerical examples of Hecke algebras $\Hs$ attached to various $\rhobar$ which satisfies all the hypotheses of theorem \ref{TheoCorps} and of classical points of $\Hs$, which are thus known to satisfy conjecture \ref{ConjIMC} though they do not satisfy the hypotheses of theorem \ref{TheoBibliographique}. The first three hypotheses are easily seen to be true as $\rhobar$ is the residual representation of an elliptic curve with good supersingular reduction at $p$ and is ramified at some prime $\ell\nmid p$ dividing only once $N(\rhobar)$. In order to verify that $\Hs$ satisfies hypothesis \ref{ItemCongruence} of this theorem, we find a classical point $x\in\Spec\Hs$ for which $H^{2}_{\et}(\Z[1/p],T_{x,\Iw})$ is computable and is equal to zero. Then $\z(f_{x})_{\Iw}$ generates $\Hun_{\et}(\Z[1/p],T_{x,\Iw})$ and the divisibility \eqref{EqDivisibility} is an equality.

In order to find classical points of $\Spec\Hs$ for which $H^{2}_{\et}(\Z[1/p],T_{x,\Iw})$ is necessarily non-trivial, and hence for which corollary \ref{CorWeak} does not reduce to the divisibility \eqref{EqDivisibility}, we then make systematic use of \cite[Theorem 1]{RibetRaising}. More precisely, we consider pairs of congruent newforms $f,g$ attached to classical points $x,y$ of $\Hs$ of level $N$ and $N\ell$ respectively and such that the Euler factor at $\ell$ of $f$ evaluated at 1 is not a $p$-adic unit whereas all the Euler factors at primes dividing $N\ell$ of $g$ evaluated at 1 are $p$-adic units. Under our hypotheses, the special value $L_{\{p\}}(f,\chi,1)/\Omega_{f}$ where $\chi$ is a conductor of finite $p$-power order is an algebraic integer so $L_{\Sigma}(f,\chi,1)/\Omega_{f}$ is not a $p$-adic unit (as it is multiplied by the Euler factor of $f$ at $\ell$). This in turn implies that $H^{2}_{\et}(\Z[1/\Sigma],T(f)_{\Iw})$ does not vanish. As $H^{2}_{\et}(\Z[1/\Sigma],T(f)_{\Iw})\tenseur_{\Lambda_{\Iw}}\Lambda_{\Iw}/\mgot$ is isomorphic to $H^{2}_{\et}(\Z[1/\Sigma],T(g)_{\Iw})\tenseur_{\Lambda_{\Iw}}\Lambda_{\Iw}/\mgot$, we deduce that $H^{2}_{\et}(\Z[1/\Sigma],T(g)_{\Iw}$ is not trivial. This in turns imply that $L_{\Sigma}(g,\chi,1)/\Omega_{g}$ is not a $p$-adic unit, and so then is also $L_{\{p\}}(g,\chi,1)$ as the Euler factors intervening in the quotient of these two special values are by assumptions all $p$-adic units. This forces $H^{2}(\Z[1/p],T(g)_{\Iw})$ to be non-trivial.

The newforms in examples \ref{Example3adic} and \ref{Example7adic} in the 3-adic and 7-adic paragraphs below are defined over the rather large number fields $K_{3}=\Q[X]/P_{3}$ and $K_{7}=\Q[X]/P_{7}$ respectively. These polynomials are given in appendix \ref{AppPoly}. We warn the reader that the verification that the relevant forms indeed have finite non-zero slope at some prime above $i$ in the field $K_{i}$ ($i=3,7$) is somewhat computationally intensive.
\paragraph{A $3$-adic example:}Consider
\begin{equation}\nonumber
\rhobar:G_{\Q}\fleche\GL_{2}(\Fp_{3})
\end{equation}
the only $G_{\Q}$-representation of Serre weight 2 and level $40$ satisfying 
\begin{equation}\nonumber
\tr(\rhobar(\Fr(\ell)))=\begin{cases}
-1&\textrm{if $\ell=7$,}\\
1&\textrm{if $\ell=11$,}\\
1&\textrm{if $\ell=13$,}\\
-1&\textrm{if $\ell=17$,}\\
1&\textrm{if $\ell=19$,}\\
1&\textrm{if $\ell=23$.}\\
\end{cases}
\end{equation}
The morphism $\rhobar$ is surjective, so $\rhobar$ satisfies in particular assumptions \ref{ItemIrr} and \ref{ItemTechnique} of theorem \ref{TheoCorps}. The restriction of $\rhobar$ to $G_{\Q_{3}}$ is irreducible so assumption \ref{ItemEmerton} is also satisfied. Hence conjecture \ref{WeakConjETNCuniv} holds for the Hecke algebra $\Hs$ attached to $\rhobar$.

Here follows some newforms attached to classical points of $\Spec\Hs[1/p]$ (for $\Sigma$ containing the set $\{2,3,5,7,13,19,41\}$).
\begin{enumerate}
\item The modular form $f_{1}\in S_{2}(\Gamma_{0}(40))$ with $q$-expansion starting with 
\begin{equation}\nonumber
q+q^{5}-4q^{7}-3q^{9}+4q^{11}-2q^{13}+2q^{17}\cdots\in S_{2}(\Gamma_{0}(40)).
\end{equation}
The abelian variety attached to $f_{1}$ in the Jacobian of $X_{0}(40)$ is the elliptic curve 
$$E_{1}:y^{2}=x^{3}-7x-6,$$
which has good supersingular reduction at $3$ with $a_{3}(f_{1})=0$.
\item The modular form $f_{2}\in S_{4}(\Gamma_{0}(40))$ with $q$-expansion starting with 
\begin{equation}\nonumber
q-6q^{3}-5q^{5}-34q^{7}+9q^{9}+16q^{11}+58q^{13}+30q^{15}-70q^{17}+\cdots\in S_{4}(\Gamma_{0}(40)),
\end{equation}
which has finite, non-zero slope at $3$. This example was communicated to us by R.Pollack.
\item The modular form $f_{3}\in S_{2}(\Gamma_{0}(520))$ with $q$-expansion starting with 
\begin{equation}\nonumber
q +\sqrt{6}q^{3}+q^5 + 2q^7 + 3q^9+(-2+\sqrt{6})q^{11} - q^{13}+\sqrt{6}q^{15}+(2-2\sqrt{6})q^{17}+\cdots\in S_{2}(\Gamma_{0}(520)),
\end{equation}
with has finite, non-zero slope at the prime $(3+\sqrt{6})$ above $3$ in $\Q[\sqrt{6}]$.
\item The modular form $f_{4}\in S_{2}(\Gamma_{0}(760))$ with $q$-expansion starting with 
\begin{equation}\nonumber
q + 3q^3 + q^5 - q^7 + 6q^9 + 4q^{11} + q^{13} + 3q^{15} - 7q^{17}-q^{19}+\cdots\cdots\in S_{2}(\Gamma_{0}(760)).
\end{equation}
The abelian variety attached to $f_{4}$ in the Jacobian of $X_{0}(760)$ is the elliptic curve
$$E_{4}:y^2=x^3-67x+926,$$
which has good supersingular reduction at $3$ with $a_{3}(f_{4})\neq0$.
\item The modular form $f_{5}\in S_{2}(\Gamma_{0}(1640))$ with $q$-expansion starting with 
\begin{equation}\nonumber
q+q^5+2q^7-3q^{9}-2q^{11}-2q^{13}-4q^{17}-2q^{19}+\cdots\in S_{2}(\Gamma_{0}(1640)).
\end{equation}
The abelian variety attached to $f_{5}$ in the Jacobian of $X_{0}(1640)$ is the elliptic curve 
$$E_{5}:y^2=x^{3}-31307x-1717706,$$
which has good supersingular reduction at $3$.
\item The modular form $f_{6}\in S_{4}(\Gamma_{0}(280))$ with $q$-expansion starting with
\begin{equation}\nonumber
q+\frac{1}{2}xq^3-5q^{5}+7q^{7}+ \frac{1}{4}(x^2 -108)q^9+\frac{1}{12}(-x^2 + 44x + 144)q^{11}+\cdots\in S_{4}(\Gamma_{0}(280))
\end{equation}
which has finite, non-zero slope at the prime ideal $(3,(x^{2}-14x-96)/48)$ above 3 in the ring of integers of $\Q[x]/P$ with
\begin{equation}\nonumber
P=x^{3}-4x^{2}-236x+192.
\end{equation}
\item\label{Example3adic} Among the newforms in $S_{4}(\Gamma_{0}(1640))$, there exists exactly one, which we denote by $f_{7}$, such that $a_{41}(f_{7})$ is equal to 41 and which has finite, non-zero slope at a single prime above 3. This newform is attached to a classical point of $\Hs$ and has coefficients in the ring of integers of $\Q[x]/P_{3}$ with $P_{3}$ equal to 
\begin{equation}\nonumber
x^{15}-269x^{13}+98x^{12}+27795x^{11}+\cdots+960245792x^3-3558446016x^2+ 1598326848x + 886331520.
\end{equation}
\end{enumerate}
It is easy to check that conjecture \ref{ConjIMC} is true for $f_{1}$ (the algebraic special value of the $L$-function is a $3$-adic unit). It follows that conjecture \ref{ConjETNCuniv} is true for $\rhobar$ and that conjecture \ref{ConjIMC} is true for all $f_{i}$ (none of these forms satisfy the hypotheses of the main theorems of \cite{SkinnerUrban,XinWanIMC}). The Euler factor of $f_{1}$ at 41 evaluated at 1 is congruent to 0 modulo 3 whereas the Euler factor of $f_{5}$ at 41 evaluated at 1 is congruent to 2 modulo 3. Both the special values $L_{\{2,3,5,41\}}(f_{1},\chi,1)$ and $L_{\{2,3,5,41\}}(f_{5},\chi,1)$ (for $\chi$ of finite $3$-power order) are computed through the image of $\z_{\{2,3,5,41\},\Iw}$ and all Euler factors of $f_{1}$ and $f_{5}$ at $\{2,5\}$ are $3$-adic units. This entails that $L_{\{3\}}(f_{5},\chi,1)/\Omega_{E_{5}}$ has to be a non-unit modulo 3 (as it indeed is). The same argument entails that $L_{\{3\}}(f_{7},\chi,1)/\Omega_{f_{7}}$ is not a $3$-adic unit. Note that the key hypotheses of theorem \ref{TheoBibliographique} are either that $f$ be ordinary or that it be of weight 2 and with $a_{p}(f)$ equal to $0$. The form $f_{7}$ satisfies none of these hypotheses.
\paragraph{A $5$-adic example:}Consider
\begin{equation}\nonumber
\rhobar:G_{\Q}\fleche\GL_{2}(\Fp_{5})
\end{equation}
the only $G_{\Q}$-representation of Serre weight 2 and level $34$ satisfying 
\begin{equation}\nonumber
\tr(\rhobar(\Fr(\ell)))=\begin{cases}
-2&\textrm{if $\ell=3$,}\\
1&\textrm{if $\ell=7$,}\\
1&\textrm{if $\ell=11$,}\\
2&\textrm{if $\ell=13$,}\\
1&\textrm{if $\ell=19$,}\\
0&\textrm{if $\ell=23$.}\\
\end{cases}
\end{equation}
The morphism $\rhobar$ is surjective, so $\rhobar$ satisfies in particular assumptions \ref{ItemIrr} and \ref{ItemTechnique} of theorem \ref{TheoCorps}. The restriction of $\rhobar$ to $G_{\Q_{5}}$ is irreducible so assumption \ref{ItemEmerton} is also satisfied. Hence conjecture \ref{WeakConjETNCuniv} holds for the Hecke algebra $\Hs$ attached to $\rhobar$.

Here follows some newforms attached to classical points of $\Spec\Hs[1/p]$ (for $\Sigma$ containing the set $\{2,5,17,29,59\}$).
\begin{enumerate}
\item The newform $f_{1}\in S_{2}(\Gamma_{0}(34))$ with $q$-expansion starting with 
\begin{equation}\nonumber
q + q^2 - 2q^3 + q^4-2q^{6}-4q^{7}+q^{8}+q^{9}+6q^{11}\cdots\in S_{2}(\Gamma_{0}(34)).
\end{equation}
The abelian variety attached to $f_{1}$ in the Jacobian of $X_{0}(34)$ is the elliptic curve 
$$E_{1}:y^{2}+xy=x^{3}-3x+1,$$
which has good supersingular reduction at $5$.
\item The newform $f_{2}\in S_{6}(\Gamma_{0}(34))$ with $q$-expansion starting with 
\begin{equation}\nonumber
q-4q^{2}+(3\sqrt{69}-3)q^{3}+16q^{4}+(4\sqrt{69}-18)q^{5}+(-12\sqrt{69}+12)q^{6}+(\sqrt{69}-1)q^{7}-64q^{8}+\cdots\in S_{6}(\Gamma_{0}(34)),
\end{equation}
which has finite, non-zero slope at the prime ideal $(7-3\sqrt{69})/2$ above $5$ in the ring of integers of $\Q[\sqrt{69}]$.
\item The newform $f_{3}\in S_{2}(\Gamma_{0}(986))$ with $q$-expansion starting with 
\begin{equation}\nonumber
q + q^2 + (x - 1)q^3 + q^4+ \frac{1}{93}(-4x^6 + 37x^5-49x^4 - 293x^3 + 525x^2 + 527x - 531)q^5+\cdots\in S_{2}(\Gamma_{0}(986)),
\end{equation}
which has finite, non-zero slope at the prime $(5,x+1)$ above 5 in the ring of integers of $\Q[x]/P$ with 
\begin{equation}\nonumber
P=x^7 - 10x^6 + 25x^5 +35x^4 - 192x^3 + 112x^2 + 156x - 117.
\end{equation}
\item The newform $f_{4}\in S_{2}(\Gamma_{0}(2006))$ with $q$-expansion starting with
\begin{equation}\nonumber
 q+q^{2}+\frac{1}{2}(-3+\sqrt{5})q^3+\frac{1}{2}(-5-\sqrt{5})q^{5}+\frac{1}{2}(-3+\sqrt{5})q^{6}+q^{7}+q^{8}+\cdots\in S_{2}(\Gamma_{0}(2006)),
\end{equation}
which has finite, non-zero slope at the prime $(\sqrt{5})$ above 5 in the ring of integers of $\Q[\sqrt{5}]$.
\end{enumerate}
It is easy to check that conjecture \ref{ConjIMC} is true for $f_{1}$ (the algebraic special value of the $L$-function is a $5$-adic unit). It follows that conjecture \ref{ConjETNCuniv} is true for $\rhobar$ and that conjecture \ref{ConjIMC} is true for all $f_{i}$ (the forms $f_{i}$ for $i\neq1$ do not satisfy the hypotheses of the main theorems of \cite{SkinnerUrban,XinWanIMC}). 

The Euler factor of $f_{1}$ at the primes 29 and 59 evaluated at 1 vanish modulo 5 so, as in the previous example, theorem \ref{TheoCorps} further entails that $L_{\{5\}}(f_{3},\chi,1)/\Omega_{f_{3}}$ and $L_{\{5\}}(f_{4},\chi,1)/\Omega_{f_{4}}$ (for $\chi$ of finite 5-power order) are not $5$-adic units.
\paragraph{A $7$-adic example:}Consider
\begin{equation}\nonumber
\rhobar:G_{\Q}\fleche\GL_{2}(\Fp_{7})
\end{equation}
the only $G_{\Q}$-representation of Serre weight 2 and level $48$ satisfying 
\begin{equation}\nonumber
\tr(\rhobar(\Fr(\ell)))=\begin{cases}
-2&\textrm{if $\ell=5$,}\\
3&\textrm{if $\ell=11$,}\\
-2&\textrm{if $\ell=13$,}\\
2&\textrm{if $\ell=17$,}\\
-3&\textrm{if $\ell=19$,}\\
1&\textrm{if $\ell=23$.}\\
\end{cases}
\end{equation}
The morphism $\rhobar$ is surjective, so $\rhobar$ satisfies in particular assumptions \ref{ItemIrr} and \ref{ItemTechnique} of theorem \ref{TheoCorps}. The restriction of $\rhobar$ to $G_{\Q_{7}}$ is irreducible so assumption \ref{ItemEmerton} is also satisfied. Hence conjecture \ref{WeakConjETNCuniv} holds for the Hecke algebra $\Hs$ attached to $\rhobar$.

Here follows some newforms attached to classical points of $\Spec\Hs[1/p]$ (for $\Sigma$ containing the set $\{2,3,59,239,373\}$).
\begin{enumerate}
\item The newform $f_{1}\in S_{2}(\Gamma_{0}(48))$ with $q$-expansion starting with 
\begin{equation}\nonumber
q+q^{3}-2q^{5}+q^{9}-4q^{11}-2q^{13}-2q^{15}\cdots\in S_{2}(\Gamma_{0}(48)).
\end{equation}
The abelian variety attached to $f_{1}$ in the Jacobian of $X_{0}(48)$ is the elliptic curve
$$E_{1}:y^{2}=x^{3}+x^{2}-4x-4,$$
which has good supersingular reduction at $7$.
\item The newform $f_{2}\in S_{8}(\Gamma_{0}(48))$ with $q$-expansion starting with 
\begin{equation}\nonumber
q-27q^3+110q^5-504q^7+729q^9-3812q^{11}+9574q^{13}-2970q^{15}+ 26098q^{17}+\cdots\in S_{8}(\Gamma_{0}(48)),
\end{equation}
which has finite, non-zero slope at $7$.
\item The newform $f_{3}\in S_{2}(2832)$ with $q$-expansion starting with
\begin{equation}\nonumber
q+q^{3}+(x+1)q^{5}+(-x^{2}/2-x/2+3)q^{7}+q^{9}+3q^{11}+(-x^{2}-3x-5)+\cdots\in S_{2}(\Gamma_{0}(2832))
\end{equation}
which has finite, non-zero slope at the prime ideal $(-x^2/2 - x/2 + 3)$ above $7$ in the ring of integers of $\Q[x]/(x^{3}+3x^{2}-6x-4)$.
\item The newform $f_{4}\in S_{2}(11472)$ with $q$-expansion starting with
\begin{equation}\nonumber
q+q^{3}+(2\sqrt{2}-1)q^{5}+q^{9}+(3-2\sqrt{2})q^{11}+(4-2\sqrt{2})q^{13}+\cdots\in S_{2}(11472)
\end{equation}
which has infinite slope at the prime $(-1+2\sqrt{2})$ above $7$ in the ring of integers of $\Q[\sqrt{2}]$.
\item\label{Example7adic} The newform $f_{5}\in S_{2}(17904)$ which is the only one with coefficients in the ring of integers of $\Q[x]/P_{7}$ with 
\begin{equation}\nonumber
P_{7}=x^{26} + 58x^{25} + 1528x^{24} + 24066x^{23} + 250201x^{22} +\cdots+ 1393890560x - 422379776.
\end{equation}
The form $f_{5}$ has finite, non-zero slope for some prime above $7$ in $\Q[x]/P_{7}$.
\end{enumerate}
It is easy to check that conjecture \ref{ConjIMC} is true for $f_{1}$ (the algebraic special value of the $L$-function is a $7$-adic unit). It follows that conjecture \ref{ConjETNCuniv} is true for $\rhobar$ and that conjecture \ref{ConjIMC} is true for all $f_{i}$ (none of these forms satisfy the hypotheses of \cite{SkinnerUrban,XinWanIMC}). Because the Euler factor of $f_{1}$ at 373 evaluated at 1 vanishes modulo 7, it is likely that $L(f_{5},\chi,1)/\Omega_{f_{5}}$ (for $\chi$ of finite 7-power order) is not a $7$-adic unit, though we have not checked this computationally.
\subsection{Proof of the main theorem}
\subsubsection{The Taylor-Wiles system of fundamental lines}
We give the proof of theorem \ref{TheoCorps}. Alongside the compatibility of the ETNC with specializations and change of levels, the crucial ingredient in the proof is the existence of a Taylor-Wiles system in the sense of \cite{WilesFermat,TaylorWiles} as further refined in \cite{DiamondHecke,FujiwaraDeformation,BockleDensity,DiamondFlachGuo}. Henceforth, we consistently assume that $\rhobar$ satisfies the first two hypotheses of theorem \ref{TheoCorps}; namely, we assume that it satisfies assumptions \ref{HypIrr} and \ref{HypLocal} below.
\begin{HypEnglish}\label{HypIrr}
The $G_{\Q(\sqrt{p^{*}})}$-representation $\rhobar$ is absolutely irreducible.
\end{HypEnglish}
\begin{HypEnglish}\label{HypLocal}
The semisimplification of the $G_{\qp}$-representation $\rhobar|G_{\qp}$ is not scalar.
\end{HypEnglish}
Fix an allowable set of primes $\Sigma$. Under the previous assumptions, the outcome of the Taylor-Wiles system method guarantees the existence of the following objects. There exists a set $X$ whose elements are the empty set and finite sets $Q$ of common cardinality $r>0$ of rational primes $q\equiv 1\modulo p$ such that, for all $n\in\N$, the set
\begin{equation}\nonumber
X_{n}\overset{\operatorname{def}}{=}\{Q\in X|\forall q\in Q,\ q\equiv1\modulo p^{n}\}
\end{equation}
is infinite and such that $\Sigma(\rhobar)\cap Q$ is empty for all $Q\in X$. There exists a regular local ring $\Lambda$ with maximal ideal $\mgot_{\Lambda}$, flat over $\zp$ of relative dimension 3 and, for all $Q\in X$, the ring $R_{Q}\eqdef\Hecke_{\Sigma\cup Q,\Iw}$ is a flat, reduced $\Lambda$-algebra generated by at most $2r$ elements. Denote by $R$ the ring $R_{\vide}$, which is by construction equal to $\Hs$, and fix for all $Q\in X$ a surjection $\pi_{Q}:\Lambda[[X_{1},\cdots,X_{2r}]]\surjection R_{Q}$. Note that because we chose $\Lambda$ to be of large dimension, the rings $R_{Q}$ are not classical Hecke algebra and in particular admit many non-classical specializations. 

If $\ell$ is a prime in $Q\in X_{n}$, the maximal torus of $\G(\Q_{\ell})$ acts on $R_{Q}$. For such a $Q$, denote by $\Gamma_{Q}$ the group through which the maximal torus of the product of the $\G(\Q_{\ell})$ for $\ell\in Q$ acts on $R_{Q}$ and fix for all $Q$ an isomorphism of commutative groups between $\Gamma_{Q}$ and $(\Z/p^{n}\Z)^{2r}$. Then $R_{Q}$ is a $\Lambda[\Gamma_{Q}]$-algebra and the quotient map $R_{Q}/I_{Q}\surjection R$ given by the quotient by the augmentation ideal $I_{Q}$ of $\Lambda[\Gamma_{Q}]$ is a surjection which is equal to the identity when $Q=\vide$. For $Q\in X_{n}$, denote by $J_{Q,n}$ the ideal of $\Lambda[\Gamma_{Q}]$ generated by $\mgot_{\Lambda}^{n}$ and the $\gamma^{p^{n}}-1$ for all $\gamma\in\Gamma_{Q}$. For all $n>0$, the ring $\Lambda_{n}\eqdef\Lambda/\mgot^{n}\Lambda[(\Z/p^{n}\Z)^{2r}]$ is isomorphic to $\Lambda[\Gamma_{Q}]/J_{Q,n}$ through the isomorphism between $\Gamma_{Q}$ and $(\Z/p^{n}\Z)^{2r}$ fixed above. Denote by $\Lambda_{\infty}$ the inverse limit on $n$ of the $\Lambda_{n}$ 
\begin{equation}\nonumber
\Lambda_{\infty}=\limproj{n}\ \Lambda_{n}=\Lambda[[\zp^{2r}]]
\end{equation}
and by $I_{\infty}$ the augmentation ideal of $\Lambda_{\infty}$. For $Q\in X$, write $R_{Q,n}$ for the $\Lambda_{n}$-algebra equal to $R_{Q}/J_{Q,n}R_{Q}$.

There exists a sequence $\{Q(n)\}_{n\in\N}$ satisfying the following properties. For all $n\in\N$, $Q(n)$ is an element of $X_{n}$ and if $m\leq n$, then there is an isomorphism $\pi_{n,m}:R_{Q(n),m}\simeq R_{Q(m),m}$ making the diagram
\begin{equation}\nonumber
\xymatrix{
&R_{Q(n)}\ar[rd]^{\modulo J_{Q(n),m}}&\\
\Lambda[[X_{1},\cdots,X_{2r}]]\ar[ru]^(0.5){\pi_{Q(n)}}\ar[rd]_{\pi_{Q(m)}}&&R_{Q(n),m}\simeq R_{Q(m),m}\\
&R_{Q(m)}\ar[ru]_{\modulo J_{Q(m),m}}&
}
\end{equation}
commutative. The system $\{R_{Q(n),n}\}_{n\in\N}$ with transition maps
\begin{equation}\nonumber
\pi_{n+1,n}\circ(-\modulo J_{Q(n+1),n+1}):R_{Q(n+1),n+1}\fleche R_{Q(n+1),n}\simeq R_{Q(n),n}
\end{equation}
is projective with surjective transition maps. Let $R_{\infty}=\limproj{n}\ R_{Q(n),n}$ be its inverse limit.

By construction, there is a surjective map
\begin{equation}\nonumber
\pi_{\infty}:\Lambda[[X_{1},\cdots,X_{2r}]]\fleche R_{\infty}
\end{equation}
which is in fact an isomorphism so $R_{\infty}$ is a regular local ring of Krull dimension $4+2r$. As $\Lambda_{\infty}$-module, it is finite and free and the natural maps $R_{\infty}/I_{\infty}R_{\infty}\fleche R$ and $R_{\infty}\tenseur_{\Lambda_{\infty}}\Lambda_{n}\fleche R_{Q(n),n}$ induced by the construction of $R_{\infty}$ as the inverse limit of the $R_{Q(n),n}$ are isomorphisms. The ideal $I$ generated inside $R_{\infty}$ by the elements of $I_{\infty}$ is generated by a regular sequence of $R_{\infty}$ of length $2r$ so $\Lambda\fleche R$ is a relative global complete intersection morphism of dimension 0, hence flat. 

We recall how the previous results are obtained. When (in contrast with our own setting), $\Lambda$ is chosen to be a discrete valuation ring flat over $\zp$, $R$ is chosen to be a suitable classical Hecke algebra of weight $k=2$ and finite level and $\rhobar|G_{\qp}$ is assumed to be either reducible or to arise after a twist by a character and extension of scalars from the $G_{\qp}$-module attached to the generic fiber of a $\zp$-scheme in $\bar{\Fp}_{q}$-vector spaces of type $(p,p)$ (in the sense of \cite{RaynaudSchemas}), the existence of the objects above satisfying the stated properties follows from \cite{WilesFermat,TaylorWiles,DiamondHecke,FujiwaraDeformation}. Under the same choice of $\Lambda$ but with $\rhobar|G_{\qp}$ now assumed to be irreducible and with $R$ a Hecke algebra of finite level and weight $1\leq k-1\leq p-1$, it follows from \cite[Theorem 3.6]{DiamondFlachGuo}. In \cite[Section 3,4]{BockleDensity}, it is explained how to deduce from these results the existence of our objects over $\Lambda$ and for the $p$-adic Hecke algebra $\Hs$ provided $\rhobar|G_{\qp}$ is reducible with non-scalar semisimplification or the universal deformation ring of the local representation $\rhobar|G_{\qp}$ is a regular local ring. By assumption \ref{HypLocal}, either $\rhobar|G_{\qp}$ is reducible with non-scalar semisimplification or $\rhobar|G_{\qp}$ is irreducible and its universal deformation ring is indeed a power-series ring by \cite{RamakrishnaFlat,DiamondFlachGuo}. In all the references above, the main interest was in showing that $\Hs$ was identified with the relevant universal deformation ring and was a complete intersection  ring. For that reason, the construction of the Taylor-Wiles system is usually carried only for $\Sigma=\{\ell|N(\rhobar)p\}$. Suppose however that $R$ is known to be the relevant universal deformation ring and a relative complete intersection over $\Lambda$; hence of the form $R_{\infty}/I_{\infty}R_{\infty}$ for some power-series ring $R_{\infty}$ with coefficients in $\Lambda$ and some ideal $I_{\infty}$ generated by a regular sequence. Then the quotients $R_{\infty}\tenseur_{\Lambda_{\infty}}\Lambda_{n}$ are seen to be equal to the quotients $R_{Q(n)}/J_{Q(n),n}R_{Q(n)}$.

For all $Q\in X$, put $\Delta_{{Q}}=\Delta_{R_{Q}}(T_{\Sigma\cup Q,\Iw})$. The injection $\triv_{Q}:\Delta_{Q}\plonge Q(R_{Q})$ of conjecture \ref{ConjETNCuniv} (for the set of primes $\Sigma\cup Q$) induces an isomorphism 
\begin{equation}\nonumber
\Delta_{Q}^{-1}\simeq\Det_{R_{Q}}[y_{Q}R_{Q}\fleche x_{Q}R_{Q}]
\end{equation}
which depends on a choice of $(x_{Q},y_{Q})\in R_{Q}^{2}$ a pair of non zero-divisors and where the complex in the right-hand side is placed in degree $-1$ and $0$. More generally, suppose $\Sigma$ and $\Sigma'$ are admissible set of primes and let $\Sigma''$ be their union. Let $J$ and $J'$ be two ideals of $\Hs$ and $\Hecke_{\Sigma',\Iw}$ respectively such that $\Hs/J$ is isomorphic to $\Hecke_{\Sigma',\Iw}/J'$. Fix a choice of $(x,y)\in\Hs$ and $(x',y')\in\Hecke_{\Sigma',\Iw}$ such that $\triv_{\Sigma}$ and $\triv_{\Sigma'}$ respectively induce isomorphisms
\begin{equation}\nonumber
\Delta_{\Hs}(\Ts)^{-1}\simeq\Det_{\Hs}[y\Hs\fleche x\Hs]
\end{equation}
and
\begin{equation}\nonumber
\Delta_{\Hecke_{\Sigma',\Iw}}(T_{\Sigma',\Iw})^{-1}\simeq\Det_{\Hecke_{\Sigma',\Iw}}[y'\Hecke_{\Sigma',\Iw}\fleche x'\Hecke_{\Sigma',\Iw}].
\end{equation}
Assume moreover that $x,y$ do not belong to $J$ and that $x',y'$ do not belong to $J'$ and define
\begin{equation}\nonumber
\phi_{\Sigma^{*}}=\frac{y^{*}}{x^{*}}\triv_{\Sigma^{*}}
\end{equation}
for $*=\vide,'$ or $''$ (equivalently, this corresponds to the trivialization of
\begin{equation}\nonumber
[y^{*}\Hecke_{\Sigma^{*},\Iw}\fleche x^{*}\Hecke_{\Sigma^{*},\Iw}]
\end{equation}
obtained by identifying the two terms of the complex). Then all the arrows in the diagram 
\begin{equation}\label{EqDiagCompXY}
\xymatrix{
&\Delta_{\Hecke_{\Sigma',\Iw}}\ar[r]^{\phi_{\Sigma'}}&\Hecke_{\Sigma',\Iw}\ar[dr]^{\mod J'}\\
\Delta_{\Hecke_{\Sigma'',\Iw}}\ar[ur]^{\pi_{\Sigma'',\Sigma'}}\ar[dr]_{\pi_{\Sigma'',\Sigma}}\ar[r]^{\phi_{\Sigma''}}&\Hecke_{\Sigma'',\Iw}\ar[rd]\ar[ru]&&R\\
&\Delta_{\Hs}\ar[r]^{\phi_{\Sigma}}&\Hs\ar[ru]_{\mod J}
}
\end{equation}
send a basis to a basis and so the diagram commutes up to a choice of units. Once a choice of $(x,y)$ is fixed, it is thus possible to choose $(x'',y'')$ in the pre-image of $(x,y)$ under $\Hecke_{\Sigma'',\Iw}\fleche\Hs$ and similarly for $(x',y')$.

It is enough to prove theorem \ref{TheoCorps} to show that $y$ can be chosen in $\Hs\croix$ in the isomorphism
\begin{equation}\label{EqObjectif}
\Delta_{\Hs}(\Ts)^{-1}\simeq\Det_{\Hs}[y\Hs\fleche x\Hs]
\end{equation}
induced by $\triv_{\Sigma}$. Assume by way of contradiction that this is not possible and fix a choice of $(x,y)$. Then for $N\in\N$ large enough, $x$ and $y$ do not belong to $\mgot_{\Hs}^{N}$ so there exist, by diagram \eqref{EqDiagCompXY}, a sequence $(a_{n})\in\N^{\N}$, a choice of $Q(a_{n})$ and compatible choices of isomorphisms
\begin{equation}\nonumber
\Delta_{Q(a_{n})}^{-1}\simeq\Det_{R_{Q(a_{n})}}[y_{Q(a_{n})}R_{Q(a_{n})}\fleche x_{Q(a_{n})}R_{Q(a_{n})}]
\end{equation}
with $y_{Q(a_{n})}\in\mgot_{R_{Q(a_{n})}}$ and $x_{Q(a_{n}))}/y_{Q(a_{n})}\notin R_{Q(a_{n})}$. Because the system $\{\Delta_{Q(a_{n})}\}_{n\in\N}$ is a sub-system of the Taylor-Wiles system, the transition maps  
\begin{equation}\nonumber
\Delta_{Q(a_{n+1})}\tenseur R_{Q(a_{n+1}),a_{n}}\fleche\Delta_{Q(a_{n})}\tenseur R_{Q(a_{n}),a_{n}}
\end{equation}
are canonical isomorphisms. Renaming $a_{n}$ as $n$, we get a projective system $\{\Delta_{Q(n)}\}$ with, for each $n$, a specified isomorphism
\begin{equation}\nonumber
\Delta_{Q(n)}^{-1}\simeq\Det_{R_{Q(n)}}[y_{Q(n)}R_{Q(n)}\fleche x_{Q(n)}R_{Q(n)}]
\end{equation}
compatible with $R_{Q(n)}\fleche R_{Q(m)}$ for $n\geq m$ and in particular such that $y_{Q(n)}$ belongs to the maximal ideal of $R_{Q(n)}$ for all $n$. By construction, the inverse limit of this projective system
\begin{equation}\nonumber
\Delta_{\infty}\eqdef\limproj{n}\ \Delta_{Q(n),n}
\end{equation}
is endowed with a specified isomorphism (which depends highly on all our previous choices) 
\begin{equation}\nonumber
\Delta_{\infty}^{-1}\simeq\Det_{R_{\infty}}[y_{\infty}R_{\infty}\fleche x_{\infty}R_{\infty}]
\end{equation}
with $y_{\infty}$ in the maximal ideal of $R_{\infty}$ and $y_{\infty}\nmid x_{\infty}$. 
\subsubsection{Reduction to the discrete valuation ring case}
The ring $R_{\infty}$ is a power-series ring of Krull dimension at least 4 equal to the inverse limit of a projective system with surjective transition maps so it admits a local morphism $\psi_{0}:R_{\infty}\fleche S_{0}$ satisfying the following properties.
\begin{enumerate}
\item The ring $S_{0}$ is a principal artinian ring of residual characteristic $p$ with maximal ideal $\mgot_{S_{0}}$.
\item Neither $\psi_{0}(y_{\infty})$ nor $\psi_{0}(x_{\infty})$ are zero and 
\begin{equation}\nonumber
\max\{n\in\N|\psi_{0}(x_{\infty})\in\mgot^{n}_{S_{0}}\}<\max\{n\in\N|\psi_{0}(y_{\infty})\in\mgot^{n}_{S_{0}}\}.
\end{equation}
\item There exists a surjection $R_{Q(n),n}\surjection S_{0}$ for some $n\in\N$.
\suspend{enumerate}
As $R_{Q(n)}$ is the ring parametrizing deformations of $\rhobar$ with coefficients in artinian $\zp$-algebra and unramified outside $\Sigma\cup Q(n)$, there exist an allowable set $\Sigma_{0}$ of primes and a $G_{\Q,\Sigma_{0}}$-representation $(T_{\psi_{0}},\rho_{\psi_{0}},S_{0})$ such that $\rhobar_{\psi_{0}}$ is isomorphic to $\rhobar$. Hence, there exist a specialization $\psi:\Hecke_{\Sigma_{0},\Iw}\fleche S$ with values in a characteristic zero, reduced, flat one-dimensional $\zp$-algebra, a $G_{\Q,\Sigma_{0}}$-representation $(T_{\psi},\rho_{\psi},S)$ and an integer $n\in\N$ such that $S/\mgot^{n}_{S}$ is isomorphic to $S_{0}$. In particular, for any choice of $(x,y)\in S^{2}$ such that there is an isomorphism \begin{equation}\nonumber
\Delta_{S}(T_{\psi})^{-1}\simeq\Det_{S}[y_{\psi}S\fleche x_{\psi}S],
\end{equation}
the image of $y_{\psi}$ in $S/\mgot^{n}_{S}\simeq S_{0}$ belongs to a higher power of the maximal ideal than the image of $x_{\psi}$ and so $\triv_{\psi}$ does not induce an inclusion $\Delta_{S}(T_{\psi})^{-1}\plonge S$. Because the union of the sets disobeying the following conditions is of large codimension, by changing our choice of $\psi_{0}$ if necessary, we may and do further assume that $\z_{\psi}$ is not a zero divisor and that $H^{2}(G_{\qp},T_{\psi}\tenseur_{S}S_{\Iw})$ is a finite group. At least one of the irreducible component of $S$ satisfies the same properties as $S$ and so does further its normalization. Hence, we may and do assume that $S$ is a discrete valuation ring. For the same reasons of dimensionality as above, we may then assume that $\psi$ factors through a single irreducible component $\Raid$ of $\Hecke_{\Sigma_{0},\Iw}$. As this is convenient, though not strictly logically necessary, we do so. By construction, $\triv_{\psi}$ does not induce an inclusion $\Delta_{S}(T_{\psi})^{-1}\plonge S$.

By assumption, the image of $\rhobar$ acts irreducibly on $\Fpbar_{p}^{2}$ and is of order divisible by $p$ so contains a subgroup conjugate to $\SL_{2}(\Fp_{q})$ for some $q=p^{n}$ and hence an element $\bar{\s}\neq\Id$ which is unipotent. Let $\s$ be a lift of $\bar{\s}$ to $\rho_{\psi}(G_{\Q})$. Then the kernel of $\s-1$ is strictly included in $T_{\psi}$ and its cokernel is dimension 1 after tensor product with $S/\mgot_{S}$. Hence, the cokernel of $\s-1$ is not torsion and is generated by a single element, so it is free of rank 1. The representation $T_{\psi}$ thus satisfies hypotheses $(\operatorname{i_{{str}}})$, $\operatorname{(ii_{{str}}})$ and $(\operatorname{iv}_{\pid})$ of \cite[Theorem 0.8]{KatoEulerOriginal}. As $T_{\psi}$ is absolutely irreducible and not abelian, the conclusion of \cite[Proposition 8.7]{KatoEulerOriginal} also holds. As establishing this proposition is the sole function of hypothesis $(\operatorname{iii})$ of \cite[Theorem 0.8]{KatoEulerOriginal}, the conclusion of this theorem holds for $T_{\psi}$. As $H^{2}(G_{\qp},T_{\psi}\tenseur_{S}S_{\Iw})$ is finite, its localization at localization at any $\pid\in\Spec S_{\Iw}$ of height 1 vanishes and so there is a canonical isomorphism between its determinant and $S_{\Iw}$. Hence, there is a canonical isomorphism 
\begin{equation}\label{EqUnIso}
\Det_{S_{\Iw}}\RGamma_{\et}(\Z[1/p],T_{\psi}\tenseur_{S}S_{\Iw})\tenseur_{S_{\Iw}}M_{\psi,\Iw}\isocan\Det_{S_{\Iw}}[S_{\Iw}\fleche x_{S_{\Iw}}S_{\Iw}].
\end{equation}
Denote by $\psi^{*}:\Raid\fleche S$ the specialization through which $\psi$ factors. As $S_{\Iw}$ is a regular local ring, the left-hand side of \eqref{EqUnIso} is canonically isomorphic to $\Delta_{S_{\Iw}}(T_{\psi^{*}}\tenseur_{S}S_{\Iw})$ and the isomorphism \eqref{EqUnIso} coincides with the trivialization of conjecture \ref{WeakConjETNCShimmering}. For all $\ell\nmid p$, the rank of the submodule of $I_{\ell}$-invariants is unchanged under characteristic zero specialization of $S_{\Iw}$ so the quotient of $S_{\Iw}$ by its augmentation ideal induces a canonical isomorphism 
\begin{equation}\nonumber
\Delta_{S_{\Iw}}(T_{\psi^{*}}\tenseur_{S} S_{\Iw})\tenseur_{S_{\Iw}}S\isocan\Delta_{S}(T_{\psi^{*}})
\end{equation}
compatible with \eqref{EqUnIso} and $\triv_{\psi}$. This implies that taking tensor product with $\Frac(S)$ and identifying the determinants of acyclic complexes with $\Frac(S)$ induces an isomorphism
\begin{equation}\label{EqTrivIwa}
\Delta_{S}(T_{\psi^{*}})^{-1}\simeq\Det_{S}[S\fleche x_{S}S].
\end{equation}
Removing Euler factors then shows that $\triv_{\psi}$ induces an injection $\Delta_{S}(T_{\psi})^{-1}\plonge S$, in contradiction with the properties of $\psi$.

Hence, it is possible to chose $y\in\Hs\croix$ in equation \eqref{EqObjectif} and the first assertion of theorem \ref{TheoCorps} is established. We make a choice of $(x,y)$ in equation \ref{EqObjectif} with $y\in\Hs\croix$. If moreover assumption \ref{ItemCongruence} of theorem \ref{TheoCorps} holds, then by theorem \ref{TheoUnivSpec} there exists a classical point $z\in\Spec\Hs[1/p]$ such that the image of $x$ in $\Hs^{z}$ is a unit. Hence, $x$ is an element of $\Hs\croix$ and the second assertion of \ref{TheoCorps} thus holds.
\paragraph{Concluding remarks:}

(i) Suppose that the semisimplification of $\rhobar|G_{\qp}$ is scalar but that $\rhobar|G_{\qp}$ itself is not scalar, or more concretely suppose that there is an isomorphism
\begin{equation}\nonumber
\rhobar|G_{\qp}\simeq\matrice{\bar{\chi}}{*}{0}{\bar{\chi}}
\end{equation}
with $*\in\Fpbar_{p}\croix$. Then the statement of conjecture \ref{WeakConjETNCuniv} and its proof under the hypotheses of theorem \ref{TheoCorps} go through unchanged once $\Hs$ is replaced everywhere by its quotient $\Hs^{\ord}$ parametrizing nearly-ordinary deformations; that is to say deformations with reducible associated $G_{\qp}$-representation (the ring $\Hs^{\ord}$ is then a Hida-Hecke algebra). Indeed, the only place in which we made use of the hypothesis that the semisimplification of $\rhobar|G_{\qp}$ is non-scalar is in our appeal to the results of \cite{BockleDemuskin,BockleDensity} where it is used to pass from the nearly-ordinary universal deformation ring to the universal deformation ring.

(ii) In the absence of assumption \ref{HypIrr}, so when $\rhobar$ is reducible, the method of Taylor-Wiles system is developed in \cite{SkinnerWilesDur}. It would be interesting to know whether the Euler systems method can be used conjointly with it in this context. 

(iii) It would also be interesting to investigate whether the combination of the Euler and Taylor-Wiles systems method can be used in the context of anticyclotomic Iwasawa theory of self-dual twists of modular forms. Yet another intriguing question is whether the use of horizontal congruences can be put to use alongside the method of Euler systems to prove the ETNC with coefficients in group-algebra. The first case to consider would be the direct generalization to modular motives of \cite{BurnsGreitherTamagawa}; that is to say to establish the ETNC for a modular motive with coefficients in $\zp[[\Gal(\Q(\zeta_{Np^{\infty}})/\Q)]]$ with $p|\phi(N)$. The dream would be of course to unite all these strands.

(iv) In a related direction, it would be interesting to write down an axiomatic version of the method of Kolyvagin systems in which the Kolyvagin argument is conducted in parallel with the horizontal congruences of the Taylor-Wiles system method.

(v) Attentive readers may have remarked that the results of section \ref{SubHecke} use precious few specific properties of modular forms and would carry over essentially unchanged for the completed cohomology in middle degree of a tower of Shimura varieties. Missing ingredients are the results on the Weight-Monodromy Conjecture in subsection \ref{SubEulerFactors}, which have been generalized to this setting in \cite{SahaPreprint}, and-much more seriously-the existence of zeta elements. Our results hint at the fact that the putative zeta elements and Euler systems attached to a tower of Shimura varieties for a reductive group $\G$ split at $p$ should be closely linked to the Kirilov model of the $\G(\qp)$-representation given by completed cohomology.
\appendix
\numberwithin{equation}{subsection}
\section{Appendices}
\subsection{The determinant functor}\label{AppDeterminant}
\newcommand{\ctf}{\operatorname{ctf}}
Let $\Lambda$ be a ring. A complex of $\Lambda$-modules is perfect if and only if it is quasi-isomorphic to a bounded complex of projective $\Lambda$-modules. A $\Lambda$-module is perfect if it is a perfect complex, that is to say and if and only if it has finite projective dimension over $\Lambda$. If $\Lambda$ is a local noetherian ring, then all bounded complexes of $\Lambda$-modules are perfect if (and only if) $\Lambda$ is regular by the theorem of Auslander-Buchsbaum and Serre.

The determinant functor $\Det_{\Lambda}(-)$ of \cite{MumfordKnudsen} (see also \cite{DeligneDeterminant}) is a functor
\begin{equation}\nonumber
\Det_{\Lambda}P=\left(\underset{\Lambda}{\overset{{\rank_{\Lambda}P}}{\bigwedge}}P,\rank_{\Lambda}P\right)
\end{equation}
from the category of projective $\Lambda$-modules (with morphisms restricted to isomorphisms) to the symmetric monoidal category of graded invertible $\Lambda$-modules (with morphisms restricted to isomorphisms). The determinant functor admits an extension to a functor from the category of perfect complexes of $\Lambda$-modules with morphisms restricted to quasi-isomorphisms to the category of graded invertible $\Lambda$-modules by setting
\begin{equation}
\Det_{\Lambda}C^{\bullet}=\underset{i\in\Z}{\bigotimes}{}\Det_{\Lambda}^{(-1)^{i}}C^{i}.
\end{equation}
This extension is the unique up to canonical isomorphism which satisfies the properties of \cite[Definition I]{MumfordKnudsen}.
In particular, $\Det_{\Lambda}(-)$ commutes with derived tensor product; there is a canonical isomorphism between $\Det_{\Lambda}(0)$ and $(\Lambda,0)$ and there exists a canonical isomorphism
\begin{equation}\label{EqCanIso}
\iota_{\Lambda}(\alpha,\beta):\Det_{\Lambda}C_{2}^{\bullet}\isocan\Det_{\Lambda}C_{1}^{\bullet}\tenseur_{\Lambda}\Det_{\Lambda}C^{\bullet}_{3}
\end{equation}
compatible with base-change whenever
\begin{equation}\nonumber
\suiteexacte{\alpha}{\beta}{C_{1}^{\bullet}}{C_{2}^{\bullet}}{C_{3}^{\bullet}}
\end{equation}
is a short exact sequence of complexes. If $\Lambda$ is reduced, $\Det_{\Lambda}(-)$ further extends to the derived category of perfect complexes of $\Lambda$-modules with morphisms restricted to quasi-isomorphisms and \eqref{EqCanIso} extends to distinguished triangles. The unit object $(\Lambda,0)$ of this category is simply written $\Lambda$.
\subsection{Étale and Galois cohomology}\label{AppCohomology}
A finite $p$-ring $\Lambda$ is a ring which is of finite cardinality as set and such that for every element $x\in\Lambda$, there exists $n\in\N$ such that $p^{n}x=0$. A compact $p$-ring $\Lambda$ is a compact inverse limit of finite $p$-rings. A $p$-ring with $p$ inverted $\Lambda$ is a ring $A[1/p]$ with $A$ a compact $p$-ring. 

For $\Lambda$ a finite $p$-ring killed by a power of $p$ and $X$ a noetherian scheme, let $D(X,\Lambda)$ be the derived category of perfect complexes of sheaves of $\Lambda$-modules on $X_{\et}$. A complex of sheaves $\Fcali$ in $D(X,\Lambda)$ is said to be smooth if its cohomology satisfies the following properties.
\begin{enumerate}
\item The cohomology sheaves $H^{i}(X_{\et},\Fcali)$ are constructible (\cite[Définition 4.3.2]{SGA41/2}) for all $i\in\Z$ and zero for $i$ outside a finite range.
\item For all $x\in X$, the stalk $\Fcali_{\bar{x}}$ at $x\in X$ is a perfect complex of $\Lambda$-modules.
\end{enumerate}
Let $D_{\ctf}(X,\Lambda)$ be the full sub-category of $D(X,\Lambda)$ whose objects are the smooth complexes of étale sheaves. Suppose now that $\Lambda$ is a compact $p$-ring. As in \cite[Section 3.1]{KatoViaBdR}, a smooth complex of étale sheaves $\Fcali$ of $\Lambda$-modules on $X$ is by definition a projective system of  smooth complexes of sheaves
\begin{equation}\nonumber
(\Fcali)_{I\in U}\in (D_{\ctf}(X,\Lambda/I))_{I\in U}
\end{equation}
indexed by the set $U$ of open ideals of $\Lambda$ partially ordered by inclusion such that for all pairs $I\supset J$ of open ideals, the transition map $p_{I,J}:\Fcali_{J}\fleche\Fcali_{I}$ factors through an isomorphism $\pi_{I,J}:\Fcali_{J}\Ltenseur_{\Lambda/J}\Lambda/I\simeq\Fcali_{I}$ verifying $\pi_{I,I}=\Id$ and $\pi_{J,K}\circ\pi_{I,J}=\pi_{I,K}$ if $I\supset J\supset K$. The category $D_{\ctf}(X,\Lambda)$ is then the category of projective systems of smooth complexes of étale sheaves. Finally, suppose that $\Lambda$ is a $p$-ring with $p$-inverted, so is of the form $A[1/p]$ where $A$ is a compact $p$-ring. Then $D_{\ctf}(X,\Lambda)$ is the quotient of the category $D_{\ctf}(X,A)$ by the subcategory of $p$-torsion complex of sheaves (see \cite[Section 2.8,2.9]{SGA41/2}). When $\Lambda$ is a finite product of finite extensions of $\qp$, this recovers the definition of \cite[Section 3.1.2 (3)]{KatoViaBdR}.

Let $\Lambda$ and $\Lambda'$ be $p$-rings as above. Let $\Fcali\in D_{\ctf}(X,\Lambda)$ be a smooth complex of étale sheaves and let $\Gcali$ be the sheaf $\Fcali\Ltenseur_{\Lambda}\Lambda'$ in $D_{\ctf}(X,\Lambda')$. The functor $\RGamma_{\et}(X,-)$ being triangulated and way-out, $\RGamma_{\et}(X,\Fcali)$ is a perfect complex of $\Lambda$-modules and there is a canonical isomorphism
\begin{equation}\label{EqCommute}
\RGamma_{\et}(X,\Fcali)\Ltenseur_{\Lambda}\Lambda'\isocan\RGamma_{\et}(X,\Gcali)
\end{equation} 
inducing a canonical isomorphism
\begin{align}\label{EqCommuteDet}
\large(\Det_{\Lambda}\RGamma_{\et}(X,\Fcali)\large)\tenseur_{\Lambda}\Lambda'&\isocan\Det_{\Lambda'}\RGamma_{\et}(X,\Fcali\Ltenseur_{\Lambda}\Lambda').
\end{align}
In the body of the manuscript, the only $X$ we have cause to consider arise in the following way. Let $F/\Q$ be a finite extension of $\Q$ with ring of integers $\Ocal_{F}$ and let $\Sigma$ be a finite set of primes of $\Ocal_{F}$ containing the primes above $p$. Let $X$ be either $\Spec\Ocal_{F}[1/\Sigma]$, or $\Spec\R$ or $\Spec F_{v}$ for $v$ a finite prime (which may divide $p$) and denote by $G$ the group $\pi_{1}^{\et}(X)$ (which is thus either $G_{F,\Sigma}$, or $\Gal(\C/\R)$, or $G_{F_{v}}$). In that case, if $M^{\bullet}$ is a perfect complex of $\Lambda$-modules which are ind-admissible $\Lambda[G_{F,\Sigma}]$-modules (in the sense of \cite[Section 3.3]{SelmerComplexes}), then $M^{\bullet}$ gives rise to a smooth étale sheaf $\Fcali$ on $X$ and $\RGamma_{\et}(X,\Fcali)$ is canonically isomorphic to the complex $\RGamma(G,M^{\bullet})$ computing continuous group cohomology. 	  Let $\Lambda$ and $\Lambda'$ be $p$-rings as above. Indeed, this is true if $\Lambda$ is a finite $p$-ring and the general result reduces to this case as both $\RGamma_{\et}(X,-)$ and $\RGamma(G,-)$ are triangulated and way-out.

We note that when $X=\Spec\Ocal_{F}[1/\Sigma]$, the complex $\RGamma_{\et}(X,\Fcali)$ need not be perfect if $\Fcali$ is only known to belong to $D(X,\Lambda)$ and that, as 2 is invertible in $\Lambda$, the functor $H^{0}_{\et}(\Spec\R,-)=H^{0}(\Gal(\C/\R),-)$ sends perfect complexes to perfect complexes and commutes with change of ring of coefficients.
%

\subsection{Selmer complexes}\label{AppSelmer}
For $X$ equal to $\Spec\Ocal_{F}[1/\Sigma]$, the complex $\RGamma_{c}(X,\Fcali)$ of étale cohomology with compact support outside $p$ is defined to be the complex
\begin{equation}\nonumber
\Cone\left(\RGamma_{\et}(X,\Fcali)\oplus\sommedirecte{v|p}{}\RGamma_{\et}(\Spec F_{v},\Fcali)\fleche\sommedirecte{v\in\Sigma}{}\RGamma_{\et}(\Spec F_{v},\Fcali)\right)[-1].
\end{equation}
Note that, contrary to normal usage, we do not impose a condition at primes dividing $p$. The complex $\RGamma_{c}(X,\Fcali)$ of $\Lambda$-modules satisfies  
\begin{equation}\label{EqCommuteC}
\RGamma_{c}(X,\Fcali)\Ltenseur_{\Lambda}\Lambda'\isocan\RGamma_{c}(X,\Gcali)
\end{equation}
if $\Gcali=\Fcali\Ltenseur_{\Lambda}\Lambda'$. Hence the determinant of compactly supported cohomology is defined if $\Fcali$ is in $D_{\ctf}(X,\Lambda)$ and satisfies
\begin{align}\label{EqCommuteDetC}
\large(\Det_{\Lambda}\RGamma_{c}(X,\Fcali)\large)\tenseur_{\Lambda}\Lambda'&\isocan\Det_{\Lambda'}\RGamma_{c}(X,\Fcali\Ltenseur_{\Lambda}\Lambda').
\end{align}
If $\Fcali^{\bullet}$ is a complex of ind-admissible $\Lambda[G_{F,\Sigma}]$-representations $\Fcali^{\bullet}$, the \Nekovar-Selmer complex $\RGamma_{f}(G_{F,\Sigma},\Fcali^{\bullet})$ of $\Fcali^{\bullet}$ is the object 
\begin{equation}\nonumber
\Cone\left(\Ccont(G_{F,\Sigma},\Fcali^{\bullet})\oplus\sommedirecte{v\in\Sigma^{(p)}}{}C^{\bullet}_{f}(G_{F_{v}},\Fcali^{\bullet})\fleche\sommedirecte{v\in\Sigma^{(p)}}{}\Ccont(G_{F_{v}},\Fcali^{\bullet})\right)[-1]
\end{equation}
viewed in in the derived category and in which $C^{\bullet}_{f}(G_{F_{v}},\Fcali^{\bullet})$ is defined as in \cite[Chapter 7]{SelmerComplexes}. When $\Fcali$ belongs to $D_{\ctf}(X,\Lambda)$ and arises from a complex of ind-admissible $\Lambda[G_{F,\Sigma}]$-representations $\Fcali^{\bullet}$, there is a canonical isomorphism
\begin{equation}\nonumber
\RGamma_{f}(G_{F,\Sigma},\Fcali^{\bullet})\isocan\RGamma_{\et}(\Spec\Z[1/p],\Fcali).
\end{equation}
This follows from the excision sequence in étale cohomology if $\Lambda$ is finite and, as above, the general case is reduced to the finite case as $\RGamma_{f}(G_{F,\Sigma},-)$ and $\RGamma_{\et}(\Z[1/\Sigma],-)$ are triangulated and way-out. %
\subsection{The polynomials of section \ref{SubExamples}}\label{AppPoly}
\begin{align}\nonumber
P_{3}=&x^{15}-269x^{13}+98x^{12}+27795x^{11}-22052x^{10}-1385007x^9\\\nonumber
&+1763658x^8+33697748x^7-60675152x^6-347894604x^5\\\nonumber
&+838659848x^4+960245792x^3-3558446016x^2\\\nonumber
&+ 1598326848x + 886331520
\end{align}
\begin{align}\nonumber
P_{7}=&x^{26} + 58x^{25} + 1528x^{24} + 24066x^{23} + 250201x^{22}\\\nonumber
& + 1777238x^{21} + 8485597x^{20} + 24131426x^{19} + 14289099x^{18}\\\nonumber
& - 194507055x^{17} - 852543623x^{16} - 1353592301x^{15}\\\nonumber
& + 1212201450x^{14} + 9280787596x^{13} + 14283381200x^{12}\\\nonumber
& - 3899190818x^{11} - 41570797764x^{10} - 44310276332x^{9} \\\nonumber
&+ 20464624672x^{8} + 75496919254x^{7} + 35942792436x^{6}\\\nonumber
& - 35420299000x^{5} - 37588576512x^{4} - 903997728x^{3}\\\nonumber
& + 8639549952x^{2} + 1393890560x - 422379776
\end{align}
\bibliographystyle{amsalpha}
\def\Dbar{\leavevmode\lower.6ex\hbox to 0pt{\hskip-.23ex \accent"16\hss}D}
  \def\cfac#1{\ifmmode\setbox7\hbox{$\accent"5E#1$}\else
  \setbox7\hbox{\accent"5E#1}\penalty 10000\relax\fi\raise 1\ht7
  \hbox{\lower1.15ex\hbox to 1\wd7{\hss\accent"13\hss}}\penalty 10000
  \hskip-1\wd7\penalty 10000\box7}
  \def\cftil#1{\ifmmode\setbox7\hbox{$\accent"5E#1$}\else
  \setbox7\hbox{\accent"5E#1}\penalty 10000\relax\fi\raise 1\ht7
  \hbox{\lower1.15ex\hbox to 1\wd7{\hss\accent"7E\hss}}\penalty 10000
  \hskip-1\wd7\penalty 10000\box7} \def\Dbar{\leavevmode\lower.6ex\hbox to
  0pt{\hskip-.23ex \accent"16\hss}D}
  \def\cfac#1{\ifmmode\setbox7\hbox{$\accent"5E#1$}\else
  \setbox7\hbox{\accent"5E#1}\penalty 10000\relax\fi\raise 1\ht7
  \hbox{\lower1.15ex\hbox to 1\wd7{\hss\accent"13\hss}}\penalty 10000
  \hskip-1\wd7\penalty 10000\box7}
  \def\cftil#1{\ifmmode\setbox7\hbox{$\accent"5E#1$}\else
  \setbox7\hbox{\accent"5E#1}\penalty 10000\relax\fi\raise 1\ht7
  \hbox{\lower1.15ex\hbox to 1\wd7{\hss\accent"7E\hss}}\penalty 10000
  \hskip-1\wd7\penalty 10000\box7} \def\Dbar{\leavevmode\lower.6ex\hbox to
  0pt{\hskip-.23ex \accent"16\hss}D}
  \def\cfac#1{\ifmmode\setbox7\hbox{$\accent"5E#1$}\else
  \setbox7\hbox{\accent"5E#1}\penalty 10000\relax\fi\raise 1\ht7
  \hbox{\lower1.15ex\hbox to 1\wd7{\hss\accent"13\hss}}\penalty 10000
  \hskip-1\wd7\penalty 10000\box7}
  \def\cftil#1{\ifmmode\setbox7\hbox{$\accent"5E#1$}\else
  \setbox7\hbox{\accent"5E#1}\penalty 10000\relax\fi\raise 1\ht7
  \hbox{\lower1.15ex\hbox to 1\wd7{\hss\accent"7E\hss}}\penalty 10000
  \hskip-1\wd7\penalty 10000\box7} \def\Dbar{\leavevmode\lower.6ex\hbox to
  0pt{\hskip-.23ex \accent"16\hss}D}
  \def\cfac#1{\ifmmode\setbox7\hbox{$\accent"5E#1$}\else
  \setbox7\hbox{\accent"5E#1}\penalty 10000\relax\fi\raise 1\ht7
  \hbox{\lower1.15ex\hbox to 1\wd7{\hss\accent"13\hss}}\penalty 10000
  \hskip-1\wd7\penalty 10000\box7}
  \def\cftil#1{\ifmmode\setbox7\hbox{$\accent"5E#1$}\else
  \setbox7\hbox{\accent"5E#1}\penalty 10000\relax\fi\raise 1\ht7
  \hbox{\lower1.15ex\hbox to 1\wd7{\hss\accent"7E\hss}}\penalty 10000
  \hskip-1\wd7\penalty 10000\box7} \def\Dbar{\leavevmode\lower.6ex\hbox to
  0pt{\hskip-.23ex \accent"16\hss}D}
  \def\cfac#1{\ifmmode\setbox7\hbox{$\accent"5E#1$}\else
  \setbox7\hbox{\accent"5E#1}\penalty 10000\relax\fi\raise 1\ht7
  \hbox{\lower1.15ex\hbox to 1\wd7{\hss\accent"13\hss}}\penalty 10000
  \hskip-1\wd7\penalty 10000\box7}
  \def\cftil#1{\ifmmode\setbox7\hbox{$\accent"5E#1$}\else
  \setbox7\hbox{\accent"5E#1}\penalty 10000\relax\fi\raise 1\ht7
  \hbox{\lower1.15ex\hbox to 1\wd7{\hss\accent"7E\hss}}\penalty 10000
  \hskip-1\wd7\penalty 10000\box7} \def\cprime{$'$}
  \def\cftil#1{\ifmmode\setbox7\hbox{$\accent"5E#1$}\else
  \setbox7\hbox{\accent"5E#1}\penalty 10000\relax\fi\raise 1\ht7
  \hbox{\lower1.15ex\hbox to 1\wd7{\hss\accent"7E\hss}}\penalty 10000
  \hskip-1\wd7\penalty 10000\box7}
\providecommand{\bysame}{\leavevmode\hbox to3em{\hrulefill}\thinspace}
\providecommand{\MR}{\relax\ifhmode\unskip\space\fi MR }
\providecommand{\MRhref}[2]{%
  \href{http://www.ams.org/mathscinet-getitem?mr=#1}{#2}
}
\providecommand{\href}[2]{#2}

\bigskip{\footnotesize%
  \textrm{Olivier Fouquet} \par
  \textsc{Départment de Mathématiques, Bâtiment 425, Faculté des sciences d'Orsay Université Paris-Sud} \par  
  \textit{E-mail address}: \texttt{olivier.fouquet@math.u-psud.fr} \par
  \textit{Telephone number}: \texttt{+33169155729}
  \textit{Fax number}: \texttt{+33169156019}
  }

\end{document}